\documentclass[psamsfonts]{amsart}
\usepackage[utf8]{inputenc}

\usepackage{amssymb,amsfonts}
\usepackage{enumerate}
\usepackage{mathrsfs}
\usepackage{url}
\usepackage{mathtools}
\usepackage[font=small,labelfont=bf]{caption}
\usepackage[all,arc]{xy}
\usepackage{xcolor}
\usepackage{scalerel}
\usepackage{stackengine,wasysym}
\usepackage{todonotes}

\newcommand\reallywidetilde[1]{\ThisStyle{%
  \setbox0=\hbox{$\SavedStyle#1$}%
  \stackengine{-.1\LMpt}{$\SavedStyle#1$}{%
    \stretchto{\scaleto{\SavedStyle\mkern.2mu\AC}{.5150\wd0}}{.6\ht0}%
  }{O}{c}{F}{T}{S}%
}}

\newtheorem{thm}{Theorem}[section]
\newtheorem{cor}[thm]{Corollary}
\newtheorem{prop}[thm]{Proposition}

\newtheorem{lem}[thm]{Lemma}

\theoremstyle{definition}
\newtheorem{defn}[thm]{Definition}

\newtheorem{fact}[thm]{Fact}

\theoremstyle{remark}
\newtheorem{rem}[thm]{Remark}

\makeatletter
\let\c@equation\c@thm
\makeatother
\numberwithin{equation}{section}

\def\Ind{\setbox0=\hbox{$x$}\kern\wd0\hbox to 0pt{\hss$\mid$\hss} \lower.9\ht0\hbox to 0pt{\hss$\smile$\hss}\kern\wd0} 

\def\Notind{\setbox0=\hbox{$x$}\kern\wd0\hbox to 0pt{\mathchardef \nn=12854\hss$\nn$\kern1.4\wd0\hss}\hbox to 0pt{\hss$\mid$\hss}\lower.9\ht0 \hbox to 0pt{\hss$\smile$\hss}\kern\wd0} 

\def\ind{\mathop{\mathpalette\Ind{}}}

\title[Construction of PAC fields]{Classification Theory and the Construction of PAC Fields}

\thanks{This work was supported by NSF grant DMS-2246992 and NSF CAREER Award DMS-2442011.}

\author[N. Ramsey]{Nicholas Ramsey}
\address{Department of Mathematics \\
University of Notre Dame\\
 USA}
\email{sramsey5@nd.edu}

\date{\today}

\begin{document}

\maketitle

\begin{abstract}
We analyze a construction of Cherlin, van den Dries, and Macintyre for coding graphs in PAC fields. We show that, in many cases, model-theoretic properties of the graph are preserved in the passage from the graph to the field. As a corollary, we show that the SOP$_{n}$ hierarchy is strict in the class of fields. The main ingredient is a detailed treatment of the model theory of inverse systems of certain profinite groups that code graphs and can be realized as the absolute Galois groups of PAC fields. 
\end{abstract}

\setcounter{tocdepth}{1}
\tableofcontents

One important thread within model theory focuses on whether the various regions of the model-theoretic map of the universe have (theories of) fields in them and, if so, what those fields might look like. Some old and central conjectures, like the stable fields conjecture (predicting that all stable fields are separably closed) or the supersimple fields conjecture (predicting that all supersimple fields are pseudo-algebraically closed) concern this issue, as well as new ones, regarding dp-finite, NIP, or NIP$_{k}$ fields. The general picture presented by these conjectures is that model-theoretic tameness notions tend to place very strong constraints on what a field can look like; the fact that many of them remain open suggests that finding new fields with prescribed model-theoretic properties is quite challenging. With groups, the situation is rather different. Although model-theoretic tameness notions also place strong constraints on what groups can look like, the problem of constructing a group with a specific degree of tameness and complexity is often tractable. Coding constructions of Ershov and of Mekler allow one to encode arbitrary graphs into nilpotent groups of class $2$ and exponent $p$ \cite{ervsov1974theories, mekler1981stability}. Mekler's analysis of this construction showed that the type-counting function of the group is equal to that of the graph, and so properties like stability and NIP are preserved. Subsequent research has elaborated this idea, showing that many additional dividing lines are preserved by this graph-coding construction \cite{boissonneau2024mekler, ahn2019mekler, chernikov2019mekler, baudisch2002mekler}.  

This paper is concerned with analyzing a graph-coding construction of Cherlin, van den Dries, and Macintyre \cite{cherlin1980elementary} with the aim of demonstrating that it can be regarded as a kind of `Mekler construction' for fields. Recall that a field $K$ is called \emph{pseudo-algebraically closed} (PAC) if every absolutely irreducible variety defined over $K$ has a $K$-rational point. These fields entered model theory through Ax's celebrated characterization of pseudo-finite fields as exactly the perfect PAC fields with absolute Galois group $\hat{\mathbb{Z}}$ \cite{ax1968elementary}. In \cite{cherlin1980elementary}, the authors give a systematic account of the model theory of PAC fields, showing that many model-theoretic questions about such fields reduce to questions about the absolute Galois group.  In connection with issues around decidability, they show that an arbitrary graph can be coded into a profinite group that can be realized as the absolute Galois group of a PAC field in such a way that the graph can be interpreted in the field (and this coding was recently revisited, again in connection with decidability, in \cite{tyrrell2024finite}). This construction yields a method for producing complicated PAC fields.  To see it as a method for producing PAC fields with a prescribed degree of complexity and tameness, one would need additionally to know that the complexity of the field produced by this construction does not overshoot that of the graph.

Unfortunately, the complexity of the field produced by CDM graph coding often \emph{will} exceed that of the graph.  When the input is an infinite graph, the resulting PAC field has unbounded Galois group and therefore has both TP$_{2}$ and IP$_{k}$ for all $k$ \cite{ZoePAC2, hempel2016n}.  This might have once seemed to model theorists to be the end of the story, as structures with these properties were formerly thought to inhabit an impossibly complicated region of the model-theoretic map. But current work in neostability has found a wide array of tameness properties that are compatible with these properties.  And, indeed, work of Chatzidakis \cite{chatzidakis2002properties} has shown that for certain unbounded PAC fields\textemdash namely the $\omega$-free PAC fields and, more broadly, the Frobenius fields\textemdash there is a nice theory of independence and amalgamation that can be leveraged to treat core model-theoretic problems concerning them. It was this family of examples, in fact, that motivated the development of Kim-independence for NSOP$_{1}$ theories in its early stages. Thus, provided we are willing to take the leap of regarding theories with TP$_{2}$ and higher arity analogues of the independence property as potentially tame, it makes sense to ask whether the complexity of the fields produced by the CDM procedure can be determined as a function of the input graph. 

The construction proceeds in two steps. In the first step, one takes a graph $\Gamma$ and produces a projective profinite group $\widetilde{G_{\Gamma}}$ whose inverse system $S(\widetilde{G_{\Gamma}})$ interprets $\Gamma$. In the second step, one realizes $\widetilde{G_{\Gamma}}$ as the absolute Galois group of a PAC field. This latter step is possible thanks to a theorem of van den Dries and Lubotzky, showing that the class of projective profinite groups is exactly the class of absolute Galois groups of PAC fields \cite{lubotzky1981subgroups}. The question of whether some model-theoretic property of the graph is preserved in the resulting field, then, breaks up into two questions: one about preservation in the passage from the graph to the group, the second about preservation from the group to the field. A specific test question in this vein comes from another theorem of Chatzidakis. She showed that the theory of a PAC field has Shelah's property SOP$_{n}$ if and only if the theory of the inverse system of its absolute Galois group has SOP$_{n}$, for $n \geq 3$. It was left open whether there are examples showing that the SOP$_{n}$ hierarchy is strict in fields, that is, showing that, for all $n \geq 3$, a field can have SOP$_{n}$ and NSOP$_{n+1}$ theory. We address this by showing that the inverse system $S(\widetilde{G_{\Gamma}})$ of the profinite group produced by CDM graph coding has SOP$_{n}$ theory if and only if the input graph $\Gamma$ has SOP$_{n}$ theory. Realizing these groups as absolute Galois groups of PAC fields, then, provides the desired family of examples. In a sequel, we use this construction again to produce PAC fields with prescribed amounts of independent amalgamation. These examples give a proof of concept, but it is our hope that this construction will later serve as an engine for producing new algebraic examples across the model-theoretic map. 

The paper is structured as follows. In Section \ref{sec: profinite}, we describe the basic framework for viewing the inverse system of profinite groups as model-theoretic structures. We also survey the mostly elementary facts about profinite groups we will need, especially concerning projective profinite groups and Frattini covers. In Section \ref{sec: CDM}, we explain the CDM graph coding construction. Several subsections deal with specific algebraic issues involving the groups produced by this construction. The real model theory begins in Section \ref{sec: identification}, which describes all structures elementarily equivalent to the inverse system of a group produced by the CDM construction applied to a graph $\Gamma$. The simplest imaginable situation would have been for $S \equiv S(\widetilde{G_{\Gamma}})$ to imply that $S \cong S(\widetilde{G_{\Gamma'}})$ for some graph $\Gamma' \equiv \Gamma$; in other words, $S$ being elementarily equivalent to an inverse system produced by the CDM construction entails that $S$ is also produced by the CDM construction applied to an elementarily equivalent graph. This, it turns out, is not the case, but it is almost correct. We show that, for an infinite graph $\Gamma$, if $S \equiv S(\widetilde{G_{\Gamma}})$, then $S \cong S(\widetilde{G_{\Gamma'} \times C_{2}^{I}})$, where $C_{2}$ denotes the group $\{\pm 1\}$, $\Gamma'$ is a graph elementarily equivalent to $\Gamma$, $I$ is some index set, and the tilde denotes the universal Frattini cover. This means that an arbitrary $S \equiv S(\widetilde{G_{\Gamma}})$ is determined by a choice of $\Gamma' \equiv \Gamma$ and the dimension of an $\mathbb{F}_{2}$-vector space (namely, the continuous dual of $\mathbb{F}_{2}^{I}$). Section \ref{sec: model theory} leverages this representation to give a detailed model-theoretic analysis of algebraic closure and types in $\mathrm{Th}(S(\widetilde{G_{\Gamma}}))$ and proves several preservation results, showing that properties like NSOP$_{n}$ are preserved in the move from $\Gamma$ to $S(\widetilde{G_{\Gamma}})$. We show that the model-theoretic properties of $S(\widetilde{G_{\Gamma}})$ are determined by the structure induced on the subsystem corresponding to the maximal pro-$2$ quotient of $\widetilde{G_{\Gamma}}$, which is a free pro-$2$ group. At this stage in the analysis, we were able to leverage Chatzidakis's work on inverse systems of free pro-$p$ groups \cite{chatzidakis1998model} and results on independence and amalgamation in $H$-structures, due to Berenstein, Carmona, and Vassiliev \cite{berenstein2017supersimple}. The arguments in this section draw inspiration from related work on preservation theorems for generic expansions \cite{kruckman2022interpolative,conant2025enriching}. Section \ref{sec: fields} finally applies these preservation results to obtain PAC fields with prescribed model-theoretic properties. 

We have tried to make this paper as user-friendly as possible. A reader who is primarily interested in applying these results to produce examples of fields with certain properties should probably skip straight to Section \ref{sec: model theory}, which gives the model-theoretic analysis, and then refer back to the earlier sections as needed.

This paper is long overdue. The core ideas were worked out during a week-long visit to Zo\'e Chatzidakis at ENS in 2017 and we talked about it, and related problems, for a long while after. Virtually every aspect of this paper owes a great debt both to her mathematical contributions to the study of PAC fields and to her personal mentorship and guidance. 

\section{Profinite groups and their logic} \label{sec: profinite}

\subsection{The inverse system language}

Let $G$ be a profinite group.  To $G$, we will associate a first-order structure $S(G)$ that encodes the inverse system of $G$.  The language $L_{\mathrm{IS}}$ of the inverse system consists of sorts $X_{n}$ for all natural numbers $n \geq 1$, binary relations $\leq$ and $C$, and a ternary relation $P$.

The underlying set of $S(G)$ consists of all cosets of open normal subgroups of $G$. That is, taking $\mathcal{N}(G)$ to be the set of open normal subgroups of $G$, we have 
$$
S(G) = \{gN : N \in \mathcal{N}(G)\}.
$$
A coset $gN$ belongs to sort $X_{i}$ if and only if $[G:N] \leq i$. We will often write $S(G)_{i}$ for the elements of $S(G)$ of sort $i$.  

We interpret $C$ so that $C(gN,hM)$ if and only if $N \subseteq M$ and $gM = hM$; in other words, $C(gN,hM)$ holds when $N \subseteq M$ and $\pi(gN) = hM$ where $\pi : G/N \to G/M$ is the canonical quotient map. Finally, $P$ is interpreted so that $P(g_{1}N_{1},g_{2}N_{2},g_{3}N_{3})$ holds if and only if $N_{1} = N_{2} = N_{3}$ and $g_{1}g_{2}N_{1} = g_{3}N_{1}$.

We define $gN \leq hM$ if and only if $N \subseteq M$. We write $\sim$ for the equivalence relation defined by $gN \sim hM$ if and only if $gN \leq hM$ and $hM \leq gN$, or equivalently, if and only if $M = N$. We write $[gN]$ for the equivalence class of $gN$ under $\sim$.  The equivalence class $[gN] = [N]$ consists of the elements of the quotient $G/N$ and has a group law given by $P$. For $N,M \in \mathcal{N}(G)$, we define $[N]\vee[M] = [NM]$ and $[N] \wedge [M] = [N \cap M]$.  Then $[N] \vee [M]$ is the least upper bound and $[N] \wedge [M]$ is the greatest lower bound of $[N]$ and $[M]$. This turns $S(G)/\sim$ into a modular lattice. 

An unusual feature here is that the sorts are not disjoint and the relations $\leq$, $P$, and $C$ are not described with specified sorts. This follows the presentation of \cite[Appendix 1]{chatzidakis2002properties} and is a convention adopted for convenience. But, technically, one should have different relation symbols for each of the relevant tuples of sorts and, rather than allowing the sorts to intersect, one should have disjoint sorts and maps between the sorts that determine how elements of one sort get identified with elements of another. This more formally standard approach tends to obscure more than it illuminates, so we will stick with Chatzidakis's conventions. 

There is an $L_{\mathrm{IS}}$-theory $T_{\mathrm{IS}}$ which axiomatizes structures of the form $S(G)$ for a profinite group $G$.  We refer to $L_{\mathrm{IS}}$-structures satisfying this theory as \emph{complete systems}. Conversely, if $S \models T_{\mathrm{IS}}$, then one may form a profinite group $G$ by setting 
$$
G(S) = \varprojlim_{\alpha \in S} [\alpha]
$$
where the inverse limit is taken with respect to the maps $\pi_{\alpha,\beta} : [\alpha] \to [\beta]$ for $\alpha \leq \beta$ and where the graph of the $\pi_{\alpha,\beta}$ is given by $C\cap
([\alpha] \times[\beta])$.  These operations are dual in the sense that $G$ and $G(S(G))$ can be naturally identified for any profinite group $G$ and, likewise, $S$ and $S(G(S))$ may be naturally identified for any $S \models T_{\mathrm{IS}}$.  Epimorphisms of profinite groups $G \to H$ give rise to embeddings of $L_{\mathrm{IS}}$-structures $S(H) \to S(G)$. 

For elements $\alpha,\beta \in S$, we will write $\alpha \vee \beta$ for the identity element of $[\alpha] \vee [\beta]$ and likewise $\alpha \wedge \beta$ for the identity element of $[\alpha] \wedge [\beta]$. With this convention, $\vee$ and $\wedge$ pick out canonical (identity-coset) representatives of the join and meet on $S/\sim$. Note that $\vee$ and $\wedge$ are definable operations on $S$. We will often make use of the fact that, if $\alpha \vee \beta = \delta$, then 
$$
[\alpha \wedge \beta] \cong [\alpha] \times_{[\delta]} [\beta] = \{(a,b) \in [\alpha] \times [\beta] : \pi_{\alpha,\delta}(a) = \pi_{\beta,\delta}(b)\}.
$$
Another important fact is the \emph{modular law} \cite[Lemma 1.9]{chatzidakis1998model}, which says that for $\alpha \leq \beta$:
$$
\alpha \vee (\beta \wedge \epsilon) = \beta \wedge (\alpha \vee \epsilon).
$$
It is also worth noting that for any finite group $A$ with $|A| = k$, the set $\{\alpha \in X_{k}(S(G)) : [\alpha] \cong A\}$ is a definable set (defined by asserting that the $k$ elements in $[\alpha]$ have the same multiplication table as $A$). If $A$ is a finite group and $[\alpha] \cong A$, we refer to $\alpha$ as an $A$\emph{-element}. 

\begin{defn}
    Suppose $S \models T_{\mathrm{IS}}$.  We say $S_{0} \subseteq S$ is a \emph{subsystem} of $S$ if it satisfies the following:
    \begin{itemize}
        \item (Upwards closed) If $\alpha \leq \beta$ and $\alpha \in S_{0}$, then $\beta \in S_{0}$. 
        \item (Downwards directed) If $\alpha,\beta \in S_{0}$, then there is some $\gamma \in S_{0}$ with $\gamma \leq \alpha$ and $\gamma \leq \beta$. 
    \end{itemize}
    If $A$ is an arbitrary subset of $S$, we denote by $\langle A \rangle$ the smallest subsystem of $S$ containing $A$. It is obtained by closing under $\wedge$ and closing upwards. 
\end{defn}

If $S_{0}, S_{1} \subseteq S$ are subsystems, we can write $S_{0} \vee S_{1}$ and $S_{0} \wedge S_{1}$ for the subsystems generated by $\{\alpha \vee \beta : \alpha \in S_{0}, \beta \in S_{1}\}$ and by $\{\alpha \wedge \beta : \alpha \in S_{0}, \beta \in S_{1}\}$, respectively.  By passing to the inverse limit, the observation before the above definition entails 
$$
G(S_{0} \wedge S_{1}) \cong G(S_{0}) \times_{G(S_{0} \vee S_{1})} G(S_{1}).
$$
The modular law likewise holds for subsystems, for $S_{0} \subseteq S_{1}$: $S_{0} \vee (S_{1} \wedge S_{2}) = S_{1} \wedge (S_{0} \vee S_{2})$.  If $\alpha \in S$ and $S_{0} \subseteq S$ is a subsystem, we will write $S_{0} \wedge \alpha$ for $S_{0} \wedge \langle \alpha \rangle$ and likewise for $S_{0} \vee \alpha$. 
\subsection{Projective profinite groups}

\begin{defn}
We say $G$ is \emph{projective} if, whenever $\varphi: G \to A$ and $\psi : B \to A$ are continuous epimorphisms with $A$ and $B$ finite groups, there is a continuous homomorphism $\gamma :G \to B$ so that $\psi \circ \gamma = \varphi$ as in the following diagram:
$$
\xymatrix{ & G\ar@{-->}[dl]  \ar@{->>}[d] \\
B \ar@{->>}[r] & A }
$$
We say that $G$ has the \emph{embedding property} if, whenever $A$ and $B$ are finite groups, $B$ is a quotient of $G$, and $\varphi: G \to A$, $\psi : B \to A$ are continuous epimorphisms, there is a continuous epimorphism $\gamma: G \to B$ making the above diagram commute. A group $G$ is called \emph{superprojective} if it is projective and has the embedding property. 
\end{defn}

Examples of projective profinite groups include the free profinite groups $\hat{F}_{n}$ and free pro-$p$ groups $\hat{F}_{n}(p)$, though there are many others. 

\begin{defn} \label{Frattini} Let $G$ be a profinite group.
  \begin{enumerate}
\item
   The \emph{Frattini subgroup} $\Phi(G) \leq G$ is the intersection of
   all maximal open subgroups of $G$. 
\item  (\cite[Definition 22.5.1]{fried2008field})
A homomorphism $\varphi: H \to G$ is called a \emph{Frattini cover} if one of the following equivalent conditions holds:
\begin{itemize}
\item[(a)] $\varphi$ is surjective and $\ker(\varphi) \leq \Phi(H)$.
\item [(b)] A closed subgroup $H_{0} \leq H$ is equal to $H$ if and only if $\varphi(H_{0}) = G$.
\item [(c)] A subset $S \subseteq H$ generates $H$ if and only if $\varphi(S)$
  generates $G$.
\end{itemize}
\item A Frattini cover $\alpha: H\to G$ is called \emph{universal} if whenever
  $\beta: H'\to G$ is another Frattini cover of $G$, there exists an
  epimorphism $\gamma:H\to H'$ such that $\beta\circ\gamma =\alpha$, as
  in the following diagram:
  $$
\xymatrix{ & H\ar@{-->>}[dl]  \ar@{->>}[d] \\
H' \ar@{->>}[r] & G }
$$
\end{enumerate}
\end{defn}

\begin{fact} \label{fact2} Let $G$ be a profinite group.
  \begin{enumerate}
    \item A universal Frattini cover of $G$ exists.  It is unique up to isomorphism over $G$ \cite[Proposition 22.6.1]{fried2008field}, so will be referred to as \emph{the} universal Frattini cover of $G$ and will be denoted $\tilde
  G$. We will usually suppress mention of the map $\tilde{\varphi}: \tilde{G} \to G$.
  \item The universal Frattini cover $\tilde{G}$ of $G$ is projective and is, moreover, the smallest projective cover of $G$, in the sense that if $H$ is a projective group, $\alpha : H \to G$ is an epimorphism and $\tilde{\varphi}: \tilde{G} \to G$ is the universal Frattini cover of $G$, then  there is an epimorphism $\beta : H \to \tilde{G}$ making the following diagram commute \cite[Proposition 22.6.1]{fried2008field}:
    $$
\xymatrix{ & H\ar@{-->>}[dl]^{\beta}  \ar@{->>}[d]^{\alpha} \\
\tilde{G} \ar@{->>}[r]^{\tilde{\varphi}} & G }
$$
  \item Let $\tilde{\varphi}: \tilde{G} \to G$ be the universal Frattini cover of a profinite group $G$. Then a profinite group $H$ is a quotient of $\tilde{G}$ if and only if $H$ is a Frattini cover of a quotient of $G$ \cite[Lemma 22.6.3]{fried2008field}.
  \item Suppose $G$ is a profinite group, $F$ is a free profinite group, and $\varphi : F \to G$ is an epimorphism. If $H \leq F$ is a minimal closed subgroup such that $\varphi|_{H}$ is an epimorphism onto $G$, then $\varphi|_{H}$ is a universal Frattini cover. In particular, the universal Frattini cover of $G$ has the same rank as $G$. \cite[Corollary 22.5.3]{fried2008field}
  \end{enumerate}
\end{fact}

Fact \ref{fact2}(3) has the following consequence, which is worth spelling out explicitly:

\begin{fact} \label{fact:quotientbyfrattini}
    Let $\tilde{\varphi}: \tilde{G} \to G$ be the universal Frattini cover of a profinite group $G$. Then a finite group $B$ is a quotient of $\tilde{G}$ if and only if $B/\Phi(B)$ is a quotient of $G$. 
\end{fact}

\begin{proof}
    Clearly $B$ is a Frattini cover of $B/\Phi(B)$, so if $B/\Phi(B)$ is a quotient of $G$, then $B$ is a quotient of $\tilde{G}$ by Fact \ref{fact2}(3).  For the other direction, suppose $\psi : B \to A$ is a Frattini cover of $A$, for some quotient $A$ of $G$.  Then $A \cong B/\ker(\psi)$ and $\ker(\psi) \subseteq \Phi(B)$ so $B/\Phi(B)$ is a quotient of $A$ and thus of $G$. 
\end{proof}

Note that, by Fact \ref{fact2}(2), if $\rho : G \to H$ is an epimorphism of profinite groups and $\varphi_{G} : \tilde{G} \to G$ and $\varphi_{H} : \tilde{H} \to H$ are the universal Frattini covers of $G$ and $H$ respectively, then $\rho \circ \varphi_{G} : \tilde{G} \to H$ is an epimorphism from a projective group to $H$ and, since $\tilde{H}$ is the smallest projective group with an epimorphism to $H$, there must be some epimorphism $\tilde{\rho}: \tilde{G} \to \tilde{H}$ such that $\rho \circ \varphi_{G} = \varphi_{H} \circ \tilde{\rho}$.  We will refer to $\tilde{\rho}$ as a \emph{lift} of the epimorphism $\rho: G\to H$.  

\begin{fact} \label{fact: Frattini char}
    Suppose $G$ is a profinite group and $N \unlhd G$ is a closed normal subgroup. Assume that, for any open normal subgroup $M \unlhd G$, the quotient map $G/M \to G/NM$ is a Frattini cover. Then the quotient map $G \to G/N$ is a Frattini cover. 
\end{fact}

\begin{proof}
By hypothesis, $G/M \to G/NM$ is a Frattini cover for every open normal $M \unlhd G$, so $NM/M \leq \Phi(G/M)$. Taking the inverse limit over $M$, we have
$$
N = \varprojlim_{M \unlhd G} NM/M \leq \varprojlim_{M \unlhd G} \Phi(G/M) = \Phi(G)
$$
by \cite[Proposition 2.8.2]{ribes2010profinite}. This shows that $G \to G/N$ is a Frattini cover.
\end{proof}

Finally, we will need the following fact concerning the model theory of projective profinite groups:

\begin{fact} \label{fact:axiomatizable} \cite[Proposition 2.18]{frohn2011model}
The class of complete inverse systems of projective profinite groups is $L_{\mathrm{IS}}$-axiomatizable. 
\end{fact}

\section{CDM graph coding} \label{sec: CDM} 

\subsection{The basic set-up} 

In this subsection, we will explain a construction introduced by Cherlin, van den Dries, and Macintyre for coding graphs into profinite groups \cite{cherlin1980elementary}. We will follow the presentation from \cite[Chapter 28]{fried2008field}. 

For a natural number $n$, we write $C_{n}$ for the cyclic group of order $n$. We fix, once and for all, distinct odd primes $p$ and $q$. The dihedral group $D_{p}$ is defined by 
$$
D_{p} = \langle \beta, \gamma : \beta^{2} = 1, \gamma^{p} = 1, \beta \gamma \beta^{-1} = \gamma^{-1} \rangle.
$$
Let $\tau : D_{p} \to C_{2} = \{\pm 1\}$ denote the epimorphism obtained by quotienting $D_{p}$ by the normal subgroup generated by $\gamma$. Then we may define an action of $D_{p} \times D_{p}$ on $C_{q}$ by setting, for $(x,y) \in D_{p}\times D_{p}$ and $g \in C_{q}$, 
$$
(x,y) \cdot g = g^{\tau(x)\tau(y)}.
$$
In other words, if $\tau(x)$ and $\tau(y)$ agree, then $(x,y)$ acts trivially on $g$ and if $\tau(x) \neq \tau(y)$, then $(x,y)$ acts by inverting $g$. We let $W = C_{q} \rtimes (D_{p}\times D_{p})$ be the associated semi-direct product and let $\lambda : W \to D_{p} \times D_{p}$ be the quotient map. When doing calculations involving $W$, it is often useful to represent elements as triples $(a,b,c)$ where $a \in C_{q}$, and $(b,c) \in D_{p} \times D_{p}$, so then $\lambda$ becomes the projection to the second and third coordinates. 

\begin{defn}
 \begin{enumerate}  
 \item Given a graph $\Gamma = (V,R)$ with vertex set $V$ and edge relation $R$, we define the group $G_{\Gamma}$ by
$$
G_{\Gamma} = \{((a_{v})_{v \in V}, (b_{r})_{r \in R}) \in D_{p}^{V} \times W^{R} : r = (v,v') \in R \implies \lambda(b_{r}) = (a_{v},a_{v'}) \}.
$$
The group $G_{\Gamma}$ is a closed subgroup of the profinite group $D_{p}^{V} \times W^{R}$ and is therefore profinite. When $\Gamma$ has no edges, we set $G_{\Gamma} = D_{p}^{V}$.
\item Given a profinite group $G$, we define a graph $\Gamma(G) = (V_{G},R_{G})$ where the vertex set $V_{G}$ is given by
$$V_{G} = \{N \in \mathcal{N}(G) : G/N \cong D_{p}\}$$
and edge set $R$ given by  
$$
R_{G} = \{(N,N') \in V_{G}^{2} : N \neq N' \wedge (\exists M \in \mathcal{N}(G))(M \subseteq N \cap N' \text{ and } G/M \cong W)\}.
$$
\end{enumerate}
\end{defn}
%

\begin{rem}
    The construction of the graph $\Gamma(G)$ from the group $G$ can be viewed as an interpretation of the graph, viewed as a first-order structure, in the structure $S(G)$. The vertex set is defined in $S(G)$ by 
    $$
    V_{G} = \{ \alpha \in X_{2p}(S(G)) : P(\alpha,\alpha,\alpha) \text{ and } [\alpha] \cong D_{p}\}
    $$
    and the edge set is defined by 
    $$
    R_{G} = \{(\alpha,\alpha') \in V_{G}^{2} : \alpha \neq \alpha' \wedge (\exists \beta \in X_{4p^{2}q}(S(G)) )(\beta \leq \alpha \wedge \alpha' \text{ and } [\beta] \cong W)\}.
    $$
    When $S$ is an inverse system, we will abuse notation by writing $\Gamma(S)$ instead of $\Gamma(G(S))$, and will refer to $\Gamma(S)$ as the graph interpreted in $S$. 
\end{rem}

The key fact about these constructions is that the second construction undoes the first:

\begin{fact} \label{fact:graphrecovery} \cite[Proposition 28.8.3]{fried2008field}
    Suppose $\Gamma = (V,R)$ is a graph. Then the graph $\Gamma$ is isomorphic to $\Gamma(G_{\Gamma})$ via the map $v \mapsto \ker\pi_{v}$. 
\end{fact}
Underlying Fact \ref{fact:graphrecovery} above is the following:
\begin{fact} \label{fact:projfact} \cite[Example 28.7.4]{fried2008field}
If $\pi: D_{p}^{I} \to D_{p}$ is an epimorphism, then for some $i \in I$, $\ker(\pi_{i}) = \ker(\pi)$ where $\pi_{i}$ is the projection onto the $i$th coordinate.
\end{fact}

We will need to upgrade this slightly: 

\begin{lem} \label{lem: no unexpected quotients}
Suppose $V$, $V'$, and $I$ are disjoint finite sets.
\begin{enumerate}
    \item If $\pi : D_{p}^{V} \times C_{2}^{I} \to D_{p}$ is an epimorphism with $N = \ker(\pi)$, then there is some $v \in V$ such that $N = \ker(\pi_{v})$. 
    \item Suppose $|V| \geq |V'|$ and $\pi : D_{p}^{V} \times C_{2}^{I} \to D_{p}^{V'}$ is an epimorphism.  Then there is $V_{0} \subseteq V$ with $|V_{0}| = |V'|$ such that $\ker(\pi) = \bigcap_{v \in V_{0}} \ker(\pi_{v})$. 
\end{enumerate}
\end{lem}

\begin{proof}
    (1) Note that there must be some $v \in V$ such that $(C_{p})_{v} \cap N =1$, since otherwise, the map $\pi$ factors through $C_{2}^{V} \times C_{2}^{I}$, which is impossible. If there is some $i \in I$ such that $N \cap (C_{2})_{i}$ is trivial, then it follows that $(D_{p}^{V} \times (C_{2})^{I})/N$ must contain $D_{p}\times C_{2}$ and hence $C_{2} \times C_{2}$, which $D_{p}$ does not.  Hence for all $i$, $(C_{2})_{i} \subseteq N$ and hence the result follows by Fact \ref{fact:projfact}.

    (2) By (1), we know that, for each $w \in V'$, there is some $j(w) \in V$ such that $\ker(\pi_{w} \circ \pi) = \ker(\pi_{j(w)})$.  Let $V_{0} = \{j(w) : w \in V'\}$.  Since an element $g$ of $D_{p}^{V} \times C_{2}^{I}$ is in the kernel of $\pi$ if and only if $\pi(g)$ is the identity in each coordinate of $D_{p}^{V'}$, it follows that
    $$
    \ker(\pi) = \bigcap_{v \in V_{0}} \ker(\pi_{v}).  
    $$
    Moreover, the map $w \mapsto j(w)$ is injective, since if $w \neq w'$ and $j(w) = j(w')$, then it follows that all elements in the image of $\pi$ agree in both the $w$ and $w'$ coordinate, which is impossible since $\pi$ is surjective. This implies $|V_{0}| = |V'|$. 
\end{proof}

For reasons that will soon become clear, we will consider groups of the form $G_{\Gamma} \times C_{2}^{I}$ for a set $I$. For a vertex $v \in V$, we will write $\pi_{v} : G_{\Gamma} \times C_{2}^{I} \to D_{p}$ for the projection to the $v$ coordinate and likewise for $\pi_{r} : G_{\Gamma} \times C_{2}^{I} \to W$.  

Lemma \ref{lem: no unexpected quotients} has the following consequence:
\begin{cor} \label{cor: no unexpected frattini quotients}
    Suppose $\Gamma = (V,R)$ is a graph and $I$ is a set. Then the graph $\Gamma$ is isomorphic to $\Gamma(\widetilde{G_{\Gamma}\times C_{2}^{I}})$. 
\end{cor}

\begin{proof}
    It follows from Lemma \ref{lem: no unexpected quotients} and Fact \ref{fact:graphrecovery} that $\Gamma \cong \Gamma(G_{\Gamma} \times C_{2}^{I})$ via the map $v \mapsto \ker(\pi_{v})$. Then the conclusion follows from \cite[Lemma 28.6.1]{fried2008field}, which shows that $\Gamma(H) \cong \Gamma(\tilde{H})$ for any profinite group $H$, where $\tilde{H}$ denotes the universal Frattini cover of $H$. 
\end{proof}

\begin{lem}
   Suppose $\Gamma = (V,R)$ is a graph. Then $\Phi(G_{\Gamma}) = 1$.
\end{lem}

\begin{proof}
    We use the description of $G_{\Gamma} = C_{q}^{R} \rtimes D_{p}^{V}$. For each vertex $v \in V$, consider the pair of open subgroups
    \begin{eqnarray*}
        A_{v} &=& C_{q}^{R} \rtimes (D_{p}^{V \setminus\{v\}}\times \langle \beta \rangle) \\
        A'_{v} &=& C_{q}^{R} \rtimes (D_{p}^{V \setminus \{v\}}\times \langle \gamma \beta \rangle),
    \end{eqnarray*}
    which are $\langle \beta \rangle$ and $\langle \gamma \beta \rangle$ in the $v$ coordinate, respectively. Now $\beta$ and $\gamma\beta$ are both reflections in $D_{p}$ so $A_{v}$ and $A'_{v}$ both have index $p$ in $G_{\Gamma}$ and are therefore maximal subgroups.  For each $r \in R$, we can define 
    $$
    B_{r} = (C_{q}^{R \setminus \{r\}} \times 1) \rtimes D_{p}^{V},
    $$
    which is an open subgroup of $G_{\Gamma}$ of index $q$ and therefore also maximal. We have 
    $$
    \Phi(G_{\Gamma}) \subseteq \bigcap_{v \in V} (A_{v} \cap A'_{v}) \cap \bigcap_{r \in R} B_{r} = 1,
    $$
    so the Frattini subgroup of $G_{\Gamma}$ is trivial. 
\end{proof}

\subsection{A bounding lemma} 

For a vertex $v \in V$, there is a function $D_{p} \to G_{\Gamma} \times C_{2}^{I}$ denoted $g \mapsto (g)_{v}$, which is defined so that $(g)_{v}$ is the identity in all coordinates except in the $v$ coordinate, where it is $g$, and in the coordinates $r = (v,v')$, where it is equal to $(1,g,1)$ in $W$. Similarly, for an edge $r = (v,v')$, there is a map $W \to G_{\Gamma} \times C_{2}^{I}$ denoted $g \mapsto (g)_{r}$ defined so that if $g = (a,b,c) \in W$, then $(g)_{r}$ is the identity in all coordinates except the $v$ coordinate, where it is $b$, the $v'$ coordinate, where it is $c$, and the $r$ coordinate, where it is $g$. We will write $(C_{p})_{v}$ and $(D_{p})_{v}$ for the images of $C_{p}$ and $D_{p}$ under the map $g \mapsto (g)_{v}$; we similarly write $(C_{q})_{r}$ and $(W)_{r}$ for the corresponding images, for an edge $r \in R$. 

The group $G_{\Gamma}$ can be alternatively described as $C_{q}^{R} \rtimes D_{p}^{V}$ where $(D_{p})_{v}$ acts trivially on $(C_{q})_{r}$ for all $r$ not incident to $v$, and $(D_{p})_{v} \times (D_{p})_{v'}$ acts on $(C_{q})_{r}$ for $r = (v,v')$ as in $W$ \cite[Lemma 28.8.2]{fried2008field}. 

\begin{lem} \label{lem:easy}
    Suppose $\Gamma = (V,R)$ is a graph and $I$ is a set. 
    \begin{enumerate}
        \item Suppose $N \unlhd D_{p}^{V}$ is an open normal subgroup. If $\pi_{v}(N) \supseteq C_{p}$ for some $v \in V$, then $(C_{p})_{v} \subseteq N$.
        \item Suppose $N \unlhd G_{\Gamma} \times C_{2}^{I}$ is an open normal subgroup.  If $\pi_{v}(N) \supseteq C_{p}$ for some $v \in V$, then $(C_{p})_{v} \subseteq N$. 
        \item Suppose $N \unlhd G_{\Gamma} \times C_{2}^{I}$ is an open normal subgroup. If $C_{q} \subseteq \pi_{r}(N)$ for some $r \in R$, then $(C_{q})_{r} \subseteq N$. 
        \item Suppose $N \unlhd G_{\Gamma} \times C_{2}^{I}$ and $N \cap C_{q}^{R} = 1$. Then we have 
        $$
        N \subseteq \{(\overline{a},\overline{b},\overline{c}) \in G_{\Gamma} \times C_{2}^{I} : r = (v,v') \in R \implies \tau(a_{v}) = \tau(a_{v'}) \text{ and } b_{r} = (1,a_{v},a_{v'})\}.
        $$
    \end{enumerate}
\end{lem}

\begin{proof}
(1) We will use the description of $D_{p} = \langle \beta, \gamma | \beta^{2} = 1, \gamma^{p} = 1, \gamma^{\beta} = \gamma^{-1}\rangle$. Pick $g \in N$ such that $\pi_{v}(g) = \gamma$ and let $h$ be the element of $D_{p}^{V}$ which is $1$ in the $v$ coordinate and $\beta$ in all other coordinates.  Then $h = h^{-1}$, so $g' := hghg \in N$. We have $\pi_{v}(g') = \gamma^{2}$, while for every $v' \neq v$, $\pi_{v'}(g') = \beta\,\pi_{v'}(g)\,\beta\,\pi_{v'}(g) \in C_{p}$; thus $g' \in C_{p}^{V}$. Repeating the operation, $hg'hg' = (\gamma^{4})_{v} \in N$, and $(\gamma^{4})_{v}$ generates $(C_{p})_{v}$ since $p$ is odd.

(2) Let $\pi : G_{\Gamma} \times C^{I}_{2} \to D_{p}^{V}$ be the projection to the $V$ coordinates. By (1), $(C_{p})_{v} \subseteq \pi(N)$, so we can find some $g = (\overline{a},\overline{b},\overline{c}) \in N$ such that $a_{v}$ generates $(C_{p})_{v}$ and $a_{v'} = 1$ for all $v' \neq v$ (here $\overline{a} \in D_{p}^{V}$, $\overline{b} \in W^{R}$, and $\overline{c} \in C_{2}^{I}$). By the definition of $G_{\Gamma}$, for an edge $r \in R$, if $v$ is not incident to $r$, then $b_{r} = (x,1,1)$ for some $x \in C_{q}$, so $b_{r}^{q} = (1,1,1) \in W$. If, on the other hand, $r = (v,v')$, then $b_{r} = (x,a_{v},1)$ for some $x \in C_{q}$ and $b_{r}^{q} = (1,a_{v}^{q},1) \in W$. Also, every element of $C_{2}$ has order $2$. Thus $g^{2q} \in N$ and generates $(C_{p})_{v}$. 

(3) Suppose $r = (v,v')$ and let $\delta$ denote the generator of $C_{q}$. Pick $g \in N$ with $\pi_{r}(g) = \delta$. After replacing $g$ with $g^{2p}$, we may assume $g \in N \cap C_{q}^{R}$, that is, it is the identity in all the $V$ and $I$ coordinates, since $\delta^{2p}$ is also a generator of $C_{q}$. Let $r' \in R$ be any edge with $r' \neq r$ (if there is no such, we are already done) and let $w \in V \setminus \{v,v'\}$ be a vertex incident to the edge $r'$ (which exists since $r' \neq r$). Then $(\beta)_{w} = (\beta)_{w}^{-1}$ and we have $g' = (\beta)_{w}g(\beta)_{w}g \in N$.  Note that since $w$ is not incident to $r$, we have $\pi_{r}(g') = \delta^{2}$. If $\pi_{r'}(g) = (\delta^{i},1,1)$, then we have 
$$
\pi_{r'}(g') = (1,\beta,1)(\delta^{i},1,1)(1,\beta,1)(\delta^{i},1,1)  = (\delta^{-i},1,1)(\delta^{i},1,1) = (1,1,1).
$$
Hence some power of $g'$ will be in $N$ and agree with $(\delta)_{r}$ in both the $r$ and $r'$ coordinate. Iterating, we can find an element of $N$ which agrees with $(\delta)_{r}$ on any finitely many coordinates. Since $N$ is closed, get $(\delta)_{r} \in N$.

(4) First, we prove the corresponding statement for $W$: if $N \unlhd W$ and $N \cap C_{q} = 1$, then $N \subseteq \{(1,b,c) : \tau(b) = \tau(c)\}$. First, suppose towards contradiction that there is some $g = (a,b,c) \in N$ with $\tau(b) \neq \tau(c)$.  Again letting $\delta$ denote the generator of $C_{q}$, we have that the commutator $[(\delta,1,1),g] \in N$ and we calculate
$$
[(\delta,1,1),g] = (\delta^{-1},1,1)(a,b,c)(\delta,1,1)(a,b^{-1},c^{-1}) = (\delta^{-2},1,1),
$$
and since $\delta^{-2}$ also generates $C_{q}$, we have $C_{q} \subseteq N$. This shows that if $(a,b,c) \in N$, then $\tau(b) = \tau(c)$.  Next, suppose there is some $g = (a,b,c) \in N$ with $a \neq 1$.  Since $\tau(b) = \tau(c)$, we have $g^{2p} = (a^{2p}, b^{2p},c^{2p}) = (a^{2p},1,1)$ and since the order of $a$ is prime to $2p$, we get that $C_{q} \subseteq N$. This contradiction completes the proof. 

Now the statement follows. If $N \unlhd G_{\Gamma} \times C_{2}^{I}$ is an open normal subgroup and $C_{q}^{R} \cap N = 1$, then, for each $r \in R$, $\pi_{r}(N)$ is a normal subgroup of $W$ with $C_{q} \cap \pi_{r}(N) = 1$. Thus if $(\overline{a},\overline{b},\overline{c}) \in N$, it follows that, for all $r = (v,v') \in R$, $b_{r} = (1,a_{v},a_{v'})$ with $\tau(a_{v}) = \tau(a_{v'})$. 
\end{proof}

\begin{lem} \label{lem: still proper}
    Suppose $\Gamma = (V,R)$ is a graph. Let $P = C_{q}^{R}C_{p}^{V}$ be the (normal) product of the $p$- and $q$-Sylow subgroups of $G_{\Gamma}$. If $N \unlhd G_{\Gamma}$ is a proper normal subgroup, then $NP$ is also a proper normal subgroup.
\end{lem}

\begin{proof}
    Suppose towards contradiction that $NP = G_{\Gamma}$. Fix some $v \in V$. Then there are $a \in P$ and $b \in N$ such that $ab = (\beta)_{v}$, where $(\beta)_{v}$ denotes the element which is an involution in the $v$ coordinate and the identity in every other coordinate. We will show that $(\beta)_{v} \in N$. 

    Let $c = b^{pq} \in N$. Then $\pi_{v}(c) = \pi_{v}(b^{p}) = \beta$ and $\pi_{v'}(c) = \pi_{v'}(b^{p}) = 1$ for all $v' \neq v$ in $V$. Thus if $r$ is an edge that is not incident to $v$, we have $\pi_{r}(b^{p}) = (x_{r},1,1) \in C_{q} \rtimes (D_{p}\times D_{p}) = W$ so $\pi_{r}(c) = (1,1,1)$. On the other hand, if $r = (v,v')$ is incident to $v$, then $\pi_{r}(c) = (x_{r},\beta,1)$. Let $d \in G_{\Gamma}$ be an element such that $\pi_{r}(d) = (x_{r}^{\frac{q+1}{2}},1,1)$ for all $r$ incident to $v$ and the identity in every other coordinate. Then since $N$ is normal, $d^{-1}cd \in N$, and for each $r = (v,v')$, we have 
    \begin{eqnarray*}
    \pi_{r}(d^{-1}cd) &=& (x_{r}^{-\left( \frac{q+1}{2}\right)},1,1) \cdot (x_{r},\beta,1) \cdot (x_{r}^{\frac{q+1}{2}},1,1)\\
    &=& (x_{r}^{1-\left(\frac{q+1}{2}\right)},\beta,1) \cdot (x_{r}^{\frac{q+1}{2}},1,1) \\
    &=& (x_{r}^{1-(q+1)},\beta,1) \\
    &=& (1,\beta,1).
    \end{eqnarray*}
    Here we use that $\beta$ acts on $C_{q}$ by inversion in passing from the second line to the third. Thus, we have $d^{-1}cd = (\beta)_{v} \in N$. 

    We know, then, that $\pi_{v}(N)$ is a normal subgroup of $D_{p}$ containing $\beta$, so $\pi_{v}(N) = D_{p}$ which entails $(C_{p})_{v} \subseteq N$ by Lemma \ref{lem:easy}(2). Since $(\beta)_{v}$ is also in $N$, we get $(D_{p})_{v} \subseteq N$. Moreover, if $r = (v',v'') \in R$, since $(\beta)_{v'} \in N$, there is an element $g \in N$ such that, writing $\pi_{r}(g) = (a,b,c)$, we have $\tau(b) \neq \tau(c)$.  It follows that $C_{q} \subseteq \pi_{r}(N)$, by (the proof of) Lemma \ref{lem:easy}(4).  Then, by Lemma \ref{lem:easy}(3), we obtain $(C_{q})_{r} \subseteq N$.  Since $r$ is arbitrary, we have shown $N = G_{\Gamma}$, a contradiction. 
\end{proof}

The following technical lemma will be a key ingredient in our analysis of $\mathrm{Th}(S(\widetilde{G_{\Gamma}}))$ when $\Gamma$ is an infinite graph:

\begin{lem} \label{lem: bounding}
    Let $\Gamma = (V,R)$ be a graph and $I$ a (possibly empty) set.  Suppose $A = (G_{\Gamma}\times C_{2}^{I})/N$ is a quotient of $G_{\Gamma} \times C_{2}^{I}$, where $N \unlhd G_{\Gamma} \times C_{2}^{I}$ is an open normal subgroup. If $\Gamma_{0} = (V_{0},R_{0}) \subseteq \Gamma$ is the induced subgraph of $\Gamma$ with vertices 
    $$
    V_{0} = \{v \in V : \pi_{v}(N) \lneq D_{p} \text{ or there exists } r \in R \text{ incident to } v \text{ with } \pi_{r}(N) \cap C_{q} = 1\},
    $$
    then $A$ is a finite quotient of $G_{\Gamma_{0}} \times C_{2}^{\ell}$ for some $\ell$. Moreover, we can bound $\ell$ and $|V_{0}|$ in terms of $|A|$. 
\end{lem}

\begin{proof}
    By definition of $V_{0}$, if $v \not\in V_{0}$, then $\pi_{v}(N) = D_{p}$. By Lemma \ref{lem:easy}(2), it follows that $(C_{p})_{v} \subseteq N$ for all $v \not\in V_{0}$.  The subgroup generated by the groups $(C_p)_{v}$ for $v \in V \setminus V_{0}$ is dense in the subgroup $C_{p}^{V \setminus V_{0}}$ so, since $N$ is closed, we must have $C_{p}^{(V \setminus V_{0})} \subseteq N$. Similarly, if $r \not\in R_{0}$, then $C_{q} \subseteq \pi_{r}(N)$, so, by Lemma \ref{lem:easy}(3), $(C_{q})_{r} \subseteq N$ and since $N$ is closed, we obtain $C_{q}^{R \setminus R_{0}} \subseteq N$. Thus, $A$ is a quotient of $(G_{\Gamma} \times C_{2}^{I})/(C_{p}^{(V \setminus V_{0})}C_{q}^{(R \setminus R_{0})}) \cong G_{\Gamma_{0}} \times C_{2}^{(V \setminus V_{0})} \times C_{2}^{I}$. Since $A$ is finite, we obtain that $A$ is a quotient of $G_{\Gamma_{0}} \times C_{2}^{\ell}$ for some $\ell$. 

    Now we explain how to get the bounds.  Write $|A| = 2^{k}p^{n}q^{m}$. Since $C_{p}$ is the only nontrivial normal subgroup of $D_{p}$, we may partition $V_{0} = V_{1} \cup V_{2} \cup V_{3}$, where 
    \begin{eqnarray*}
        V_{1} &=& \{v \in V_{0} : \pi_{v}(N) = 1\} \\
        V_{2} &=& \{v \in V_{0} : \pi_{v}(N) = C_{p}\} \\
        V_{3} &=& \{v \in V_{0} : \pi_{v}(N) = D_{p}\}.
    \end{eqnarray*}
    We will bound the cardinality of each $V_{i}$ as a function of $k$, $n$, and $m$. 

    \textbf{Claim 1}: $|V_{1}| \leq \min\{k,n\}$.  
    
    \emph{Proof of Claim}: If we replace $A$ by a quotient $A' := (G_{\Gamma} \times C_{2}^{I})/(NC_{q}^{R}C_{2}^{I})$, we still have $|A'| = p^{n}c$ for some $c$ with $\mathrm{gcd}(c,p) = 1$. Additionally, by Lemma \ref{lem:easy}(2), we know $C_{p}^{V_{2} \cup V_{3}} \subseteq N$. Hence, this $A'$ may itself be viewed as a quotient of 
    $$
    (G_{\Gamma} \times C_{2}^{I})/(C_{p}^{V_{2} \cup V_{3}}C_{q}^{R}C_{2}^{I})  \cong D_{p}^{V_{1}} \times C_{2}^{V_{2}} \times C_{2}^{V_{3}}.
    $$
    Thus we can write $A' = (D_{p}^{V_{1}} \times C_{2}^{V_{2}} \times C_{2}^{V_{3}})/\overline{N}$ for some normal subgroup $\overline{N}$. Quotienting further, we can replace $A'$ by $A'' = (D_{p}^{V_{1}} \times C_{2}^{V_{2}} \times C_{2}^{V_{3}})/(NC_{2}^{V_{2}}C_{2}^{V_{3}})$ and we have that $A''$ may be viewed as a quotient of $D_{p}^{V_{1}}$ and $|A''| = p^{n}c$ for some $c$ with $\mathrm{gcd}(p,c) = 1$. Write $A'' = D_{p}^{V_{1}}/N''$ and note that $\pi_{v}(N'') = 1$ for all $v \in V_{1}$ as well, by construction and the choice of $V_{1}$. In other words, $N'' = 1$ and $A'' \cong D_{p}^{V_{1}}$ so we have $|A''| = (2p)^{|V_{1}|}$. Since $|A| \geq |A''|$, we obtain $|V_{1}| \leq \min\{k,n\}$. \qed

    \textbf{Claim 2}: $|V_{2}| \leq k$. 

    \emph{Proof of Claim}: Again by Lemma \ref{lem:easy}(2), we have $(C_{p})_{v} \subseteq N$ for all $v \in V_{2}$. We know, then, that $\overline{g} \in \bigcap_{v \in V_{2}} N \ker(\pi_{v})$ if and only if $\overline{g}$ is in $C_{p}$ in the $v$ coordinate for every $v \in V_{2}$ (and can be arbitrary otherwise).  From this it follows that 
    $$
    N \bigcap_{v \in V_{2}} \ker(\pi_{v}) = \bigcap_{v \in V_{2}}N \ker(\pi_{v}). 
    $$
    Therefore, we have 
    $$
    (G_{\Gamma} \times C_{2}^{I})/ \left( N\bigcap_{v \in V_{2}} \ker(\pi_{v}) \right) = (G_{\Gamma}\times C_{2}^{I})/\bigcap_{v \in V_{2}} N \ker(\pi_{v}) \cong (C_{2})^{V_{2}}. 
    $$
    This shows $A$ has $C_{2}^{V_{2}}$ as a quotient.  Since this group has cardinality $2^{|V_{2}|}$, we get $|V_{2}| \leq k$.\qed

    \textbf{Claim 3}: $|V_{3}| \leq 2m$. 

    \emph{Proof of Claim}:  Let $R' = \{r \in R : \pi_{r}(N) \cap C_{q} = 1\}$ and let $V'$ be the vertices that are on some edge in $R'$. Each vertex $v \in V_{3}$ is on some edge in $R'$ and each edge has only two vertices so we have $|R'| \geq \frac{|V_{3}|}{2}$. Let $N' = NC_{2}^{I}C_{p}^{V}$. Then $A$ has a quotient 
    $$
    (G_{\Gamma} \times C_{2}^{I})/N' \cong (C_{q}^{R} \rtimes C_{2}^{V})/M
    $$
    for some open normal $M \unlhd C_{q}^{R} \rtimes C_{2}^{V}$. Recall that $C_{2}^{V}$ acts on each coordinate $r = (v,v')$ of $R$ by inverting if the $v$ and $v'$ coordinates of an element of $C_{2}^{V}$ disagree and otherwise acts trivially. Thus $C_{2}^{V} = C_{2}^{V'} \times C_{2}^{(V \setminus V')}$ and the $C_{2}^{(V \setminus V')}$ factor acts trivially on all of the $R'$ coordinates. Moreover, if $r \not\in R'$, then $\pi_{r}(N) \supseteq C_{q}$ and hence $(C_{q})_{r} \subseteq N$ by Lemma \ref{lem:easy}(3) (and hence $(C_{q})_{r} \subseteq M$ for all $r \not\in R'$).  It follows, then, that  
    $$
    (C_{q}^{R} \rtimes C_{2}^{V})/M \cong ((C_{q}^{R'} \rtimes C_{2}^{V'}) \times C_{2}^{(V \setminus V')})/M''
    $$
    for some open normal $M'' \unlhd ((C_{q}^{R'} \rtimes C_{2}^{V'}) \times C_{2}^{(V \setminus V')})$. Then we have $M''C_{2}^{(V \setminus V')}$ is normal in $((C_{q}^{R'} \rtimes C_{2}^{V'}) \times C_{2}^{(V \setminus V')})$ and we can write 
    $$
    ((C_{q}^{R'} \rtimes C_{2}^{V'}) \times C_{2}^{(V \setminus V')})/M''C_{2}^{(V \setminus V')} \cong C_{q}^{R'} \rtimes C_{2}^{V'}/H
    $$
    for some $H \unlhd C_{q}^{R'} \rtimes C_{2}^{V'}$ with $H \cap C_{q}^{R'} = 1$.  Note that we have obtained this group by further  quotienting $A$ by subgroups that contain no elements of order $q$. By Lemma \ref{lem:easy}(4) (applied to the graph $(V',R')$), we have that if $((a_{r})_{r \in R'},(b_{v})_{v \in V'}) \in H$, then for all $r = (v,v') \in R'$, $a_{r} = 1$ and $b_{v} = b_{v'}$. It follows that every element of $H$ has order $2$. The quotient $(C_{q}^{R'} \rtimes C_{2}^{V'})/H$ of $A$ also has order $q^{m}c$ for some $c$ with $\mathrm{gcd}(q,c) = 1$. Therefore, we have $|R'| = m$, so $|V_{3}| \leq 2m$. \qed

    Note that Claims 1, 2, and 3 give a bound for $|V_{0}|$. Note that if we have a bound for $|V_{0}|$, we also obtain a bound for $|G_{\Gamma_{0}}|$. To finish, we bound $\ell$:

    \textbf{Claim 4}: $\ell \leq \log_{2}(|A||G_{\Gamma_{0}}|)$.

    \emph{Proof of Claim}: By Lemma \ref{lem:easy}(2) and (3), we know that $N$ contains $(C_{q})_{r}$ for every edge $r$ not in $\Gamma_{0}$ and contains $(C_{p})_{v}$ for every vertex $v$ not in $\Gamma_{0}$. Therefore, $A$ is a quotient of the group 
    $$
    G_{\Gamma_{0}} \times C_{2}^{V \setminus V_{0}} \times C_{2}^{I}. 
    $$ 
    Call this group $H$ and let $\overline{N}$ denote the image of $N$ in $H$, so $A \cong H/\overline{N}$. Let $\pi_{\Gamma_{0}} : H \to G_{\Gamma_{0}}$ denote the projection to the $G_{\Gamma_{0}}$ coordinates and let $\overline{N}' = \overline{N} \cap \ker(\pi_{\Gamma_{0}})$ and let $\tilde{A} = H/\overline{N}'$.  Clearly $A$ is a quotient of $\tilde{A}$ and if $M$ is the projection of $\overline{N}'$ to the $(C_{2}^{V \setminus V_{0}} \times C_{2}^{I})$ factor, we have
    $$
    \tilde{A} = H/\overline{N}' \cong \left((C_{2}^{V\setminus V_{0}}\times C_{2}^{I})/M\right) \times G_{\Gamma_{0}}.
    $$
    Now $(C_{2}^{V \setminus V_{0}} \times C_{2}^{I})/M \cong C_{2}^{\ell'}$ for some $\ell'$ so if we bound $\ell'$ in terms of $|A|$ we are done. We calculate 
    $$
    |\tilde{A}| = [H: \overline{N} \cap \ker(\pi_{\Gamma_{0}})] \leq [H:\overline{N}][H: \ker(\pi_{\Gamma_{0}})] = |A| |G_{\Gamma_{0}}|.
    $$
    Since $\ell' \leq  \log_{2}(|\tilde{A}|)\leq  \log_{2}(|A||G_{\Gamma_{0}}|) $ and we have already bounded $|V_{0}|$ (and thus $|G_{\Gamma_{0}}|$) in terms of $|A|$, we obtain the desired bound. 
\end{proof}

\begin{rem}
    It would have been more natural to define 
    $$
    V_{0} = \{v \in V : \pi_{v}(N) \lneq D_{p}\},
    $$
    but the statement of the lemma with this definition would be incorrect. Consider the graph $\Gamma$ with vertices $a \neq b$ and an edge between them. Then $G_{\Gamma} \cong W$ and $W$ has a normal subgroup $M = \{(1,g,h) : \tau(g) = \tau(h)\}$ with the quotient $W/M \cong C_{q} \rtimes C_{2} \cong D_{q}$. But $\pi_{a}(M)$ and $\pi_{b}(M)$ are both equal to $D_{p}$. 
\end{rem}

\subsection{Maps between the groups $G_{\Gamma}$}

If $\Phi : S \to S'$ is an embedding of complete systems, then $\Phi$ induces an injective homomorphism of graphs, which we will denote $\Gamma(\Phi) : \Gamma(S) \to \Gamma(S')$.  In general, this will not be an embedding of graphs: for example, one can consider the embedding $\Lambda : S(D_{p} \times D_{p}) \to S(W)$ dual to quotient map $\lambda : W \to D_{p} \times D_{p}$.  The graph interpreted in $S(D_{p} \times D_{p})$ consists of two vertices with no edge between them, while the graph interpreted in $S(W)$ also has two vertices but now there is an edge. When $\varphi : G' \to G$, we will sometimes write $\Gamma(\varphi)$ for the induced map $\Gamma(G) \to \Gamma(G')$ which is the same as $\Gamma(\Phi)$ for the dual map $\Phi : S(G) \to S(G')$, since the map of graphs is the same, regardless of whether we see it as coming from the map of groups or the map of inverse systems. 

\begin{lem} \label{extendingauts}
Suppose $\sigma_{1},\sigma_{2} \in \text{Aut}(D_{p})$.  Then there is $\nu \in \text{Aut}(W)$ so that the following diagram commutes:
$$
\xymatrix{ 
W \ar@{->}[r]^{\lambda} \ar@{->}[d]^{\nu} & D_{p} \times D_{p} \ar@{->}[d]^{\sigma_{1} \times \sigma_{2}} \\
W \ar@{->}[r]^{\lambda} & D_{p} \times D_{p}.
}
$$
\end{lem}

\begin{proof}
Note that, because it is the only nontrivial normal subgroup, $\text{ker}(\tau) = C_{p}$ is a characteristic subgroup of $D_{p}$.  Because the image of $\tau$ is $C_{2}$, it follows that $\tau(\sigma_{i}(g)) \neq \tau(g)$ if and only if exactly one of $\sigma_{i}(g)$ and $g$ are in $C_{p}$, which is impossible.  Therefore $\tau(\sigma_{i}(g)) = \tau(g)$ for $i = 1,2$ and the bijection 
$$
\xymatrix{ 
C_{q} \times D_{p} \times D_{p} \text{  } \ar@{->}[r]^{\mathrm{id}_{C_{q}} \times \sigma_{1} \times \sigma_{2}} & \text{  }C_{q} \times D_{p} \times D_{p}.
}
$$
is compatible with the semi-direct product structure on $W = C_{q} \rtimes D_{p} \times D_{p}$ and therefore may be regarded as an automorphism $\nu \in \text{Aut}(W)$.  The desired compatibilities are clear.  
\end{proof}

Note that if $f: \Gamma \to \Gamma$ is an automorphism of the graph $\Gamma$, we obtain a `coordinate-change' automorphism $\varphi_{f}$ of $G_{\Gamma}$ defined by 
$$
\varphi_{f}((a_{v})_{v \in V}, (b_{r})_{r \in R}) = ((a_{f^{-1}(v)})_{v \in V}, (b_{f^{-1}(r)})_{r \in R}). 
$$
Note that the induced map of graphs $\Gamma(\varphi_{f})$ maps $\ker(\pi_{v}) \mapsto \ker(\pi_{v} \circ \varphi_{f})$. It is easy, then, to check that $\Gamma(\varphi_{f}) = f$. 

On the other hand, if, for each $v \in V$, we have an automorphism $\sigma_{v} \in \mathrm{Aut}(D_{p})$, and then for each $r = (v,v') \in R$, we have an automorphism $\sigma_{r} \in \mathrm{Aut}(W)$ which is compatible with $\sigma_{v}$ and $\sigma_{v'}$ as in Lemma \ref{extendingauts}, then we get an automorphism $\overline{\sigma}$ of $G_{\Gamma}$ defined by 
$$
\overline{\sigma}((a_{v})_{v \in V}, (b_{r})_{r \in R}) = ((\sigma_{v}(a_{v}))_{v \in V}, (\sigma_{r}(b_{r}))_{r \in R}).
$$
Then we clearly have $\Gamma(\overline{\sigma}) = \mathrm{id}_{\Gamma}$. The following lemma shows that we can factor any automorphism of $G_{\Gamma}$ into a composition of two automorphisms, one of which is a coordinate-change automorphism and the other which is a tuple of automorphisms applied coordinate-wise. 

\begin{lem} \label{lem:factoring}
    Suppose $\varphi \in \mathrm{Aut}(G_{\Gamma})$ is an automorphism and $f = \Gamma(\varphi)$. Then there is a tuple of automorphisms $\overline{\sigma} = ((\sigma_{v})_{v \in V}, (\sigma_{r})_{r \in R}) \in \mathrm{Aut}(D_{p})^{V} \times \mathrm{Aut}(W)^{R}$ such that 
    $$
    \varphi = \overline{\sigma} \circ \varphi_{f},
    $$
    where $\varphi_{f}$ is the coordinate-change automorphism associated to $f$. 
\end{lem}

\begin{proof}
    Since $\Gamma(\varphi \circ \varphi_{f}^{-1}) = \mathrm{id}_{\Gamma}$, $\varphi \circ \varphi_{f}^{-1}$ preserves $\ker(\pi_{v})$ for each vertex $v$ and hence induces an automorphism $\sigma_{v}$ on the quotient $G_{\Gamma}/\ker(\pi_{v}) \cong D_{p}$.  Likewise, for each $r = (v,v') \in R$, we obtain an automorphism $\sigma_{r}$ on $G_{\Gamma}/\ker(\pi_{r})$ which is compatible with $\sigma_{v}, \sigma_{v'}$ as in Lemma \ref{extendingauts}. Unravelling definitions, we have 
    $$
(\varphi \circ \varphi_{f}^{-1})((a_{v})_{v \in V}, (b_{r})_{r \in R}) = ((\sigma_{v}(a_{v}))_{v \in V}, (\sigma_{r}(b_{r}))_{r \in R}) = \overline{\sigma}((a_{v})_{v \in V},(b_{r})_{r \in R}), 
    $$
    which shows $\varphi = \overline{\sigma} \circ \varphi_{f}$. 
\end{proof}

%
%
%

\begin{lem} \label{lem: lifting}
    Suppose $\Gamma_{0} \subseteq \Gamma$ and $\Gamma'_{0} \subseteq \Gamma'$ are graphs and we are given an isomorphism of groups $\varphi_{0} : G_{\Gamma_{0}} \to G_{\Gamma'_{0}}$ and an isomorphism of graphs $f: \Gamma' \to \Gamma$ such that $f|_{\Gamma'_{0}} = \Gamma(\varphi_{0})$. Then there is an isomorphism $\varphi : G_{\Gamma} \to G_{\Gamma'}$ such that $\Gamma(\varphi) = f$ and the following diagram commutes:
    $$
\xymatrix{ 
G_{\Gamma} \ar@{->}[r]^{\varphi} \ar@{->}[d]^{\pi_{\Gamma_{0}}} & G_{\Gamma'} \ar@{->}[d]^{\pi_{\Gamma'_{0}}} \\
G_{\Gamma_{0}} \ar@{->}[r]^{\varphi_{0}} & G_{\Gamma'_{0}}.
}
$$
\end{lem}

\begin{proof}
    Let $\Gamma_{0} = (V_{0},R_{0})$, $\Gamma = (V,R)$ and, similarly, $\Gamma'_{0} = (V'_{0},R'_{0})$ and $\Gamma' = (V',R')$. Denote by $f_{0}$ the restriction $f|_{\Gamma'_{0}}$. Let $\varphi_{f} : G_{\Gamma} \to G_{\Gamma'}$ be the coordinate-change isomorphism associated to $f$.  That is, $\varphi_{f}$ is defined by 
    $$
    \varphi_{f}((a_{v})_{v \in V}, (b_{r})_{r \in R}) = ((a_{f^{-1}(v)})_{v \in V'}, (b_{f^{-1}(r)})_{r \in R'}),
    $$
    and, likewise, let $\varphi_{f_{0}} : G_{\Gamma_{0}} \to G_{\Gamma'_{0}}$ be defined as the coordinate-change isomorphism associated to $f_{0}$. Since the automorphism $(\varphi_{0} \circ \varphi_{f_{0}}^{-1}) \in \mathrm{Aut}(G_{\Gamma'_{0}})$ satisfies $\Gamma(\varphi_{0} \circ \varphi_{f_{0}}^{-1}) = \mathrm{id}_{\Gamma'_{0}}$, it follows by Lemma \ref{lem:factoring} that $\varphi_{0} \circ \varphi_{f_{0}}^{-1} = \overline{\sigma}_{0}$ for some tuple of automorphisms $\overline{\sigma}_{0} = ((\sigma_{v})_{v \in V'_{0}},(\sigma_{r})_{r \in R'_{0}}) \in \mathrm{Aut}(D_{p})^{V'_{0}} \times \mathrm{Aut}(W)^{R'_{0}}$.  Hence $\varphi_{0} = \overline{\sigma}_{0} \circ \varphi_{f_{0}}$. 
    
    Define $\sigma_{v} = \mathrm{id}_{D_{p}}$ for each $v \in V' \setminus V'_{0}$, and then, for each $r = (v,v') \in R' \setminus R'_{0}$, define $\sigma_{r} \in \mathrm{Aut}(W)$ to be an automorphism compatible with $\sigma_{v}$ and $\sigma_{v'}$, using Lemma \ref{extendingauts}.  This defines a map $\overline{\sigma} = ((\sigma_{v})_{v \in V'}, (\sigma_{r})_{r \in R'})$ and $\overline{\sigma}_{0}$ is a subtuple which can be obtained from $\overline{\sigma}$ by restriction to the $\Gamma'_{0}$ coordinates.  

    Define $\varphi = \overline{\sigma} \circ \varphi_{f} : G_{\Gamma} \to G_{\Gamma'}$. Then we have $\Gamma(\varphi) = \Gamma(\varphi_{f}) = f$ and, by construction, we have $\pi_{\Gamma'_{0}} \circ \varphi = \varphi_{0} \circ \pi_{\Gamma_{0}}$, which completes the proof.  
\end{proof}

\begin{lem} \label{lem: unique iso}
    Suppose $\Gamma = (V,R)$ is a graph and $G \cong G_{\Gamma}$. Suppose, moreover, that there is a bijection $V \to \{N \unlhd G : G/N \cong D_{p}\}$, denoted by $v \mapsto N_{v}$, and a bijection $R \to \{M \unlhd G : G/M \cong W\}$, denoted $r \mapsto M_{r}$ satisfying the following conditions:
    \begin{enumerate}
        \item For each $v \in V$, there is an isomorphism $\varphi_{v} : G_{\Gamma}/\ker(\pi_{v}) \to G/N_{v}$.
        \item For each $r = (v,v') \in R$, we have $M_{r} \subseteq N_{v} \cap N_{v'}$ and there is an isomorphism $\varphi_{r} : G_{\Gamma}/\ker(\pi_{r}) \to G/M_{r}$ such that the following diagram commutes:
        $$
    \xymatrix{ 
G_{\Gamma}/\ker(\pi_{r}) \ar@{->}[r]^{\varphi_{r}} \ar@{->}[d] & G/M_{r}  \ar@{->}[d]^{\lambda'_{r}} \\
G_{\Gamma}/\ker(\pi_{v}) \times G_{\Gamma} /\ker(\pi_{v'}) \ar@{->}[r]^{\varphi_{v} \times \varphi_{v'}} & G/N_{v} \times G/N_{v'},
}
    $$
    where the vertical arrows denote the canonical projections. 
\end{enumerate}
Then there is a unique isomorphism $\varphi : G_{\Gamma} \to G$ which induces $\varphi_{v}$ for each $v \in V$ and $\varphi_{r}$ for each $r \in R$. 
\end{lem}

\begin{proof}
    Since $G \cong G_{\Gamma}$, we know that the canonical map 
    $$
    G \to \{((a_{v})_{v \in V}, (b_{r})_{r \in R}) \in \prod_{v \in V} G/N_{v} \times \prod_{r \in R} G/M_{r} : r = (v,v') \in R \implies \lambda'_{r}(b_{r}) = (a_{v},a_{v'})\}
    $$
    is an isomorphism.  Identifying $G$ with the range of this isomorphism, we see that the conditions (1) and (2) imply that the map $G_{\Gamma} \to G$ defined by $g \mapsto ((\varphi_{v}(g))_{v \in V}, (\varphi_{r}(g))_{r \in R})$ is an isomorphism.  Moreover, any two isomorphisms that induce the same $\varphi_{v}$ and $\varphi_{r}$ will clearly induce the same isomorphism so it is unique.  
\end{proof}

\begin{cor} \label{cor: unique aut}
    Suppose $\Gamma = (V,R)$ is a graph. Working in $S(G_{\Gamma})$, assume we are given, for each $v \in V$, some automorphism $\sigma_{v} : [\ker(\pi_{v})] \to [\ker(\pi_{v})]$, and, for each $r = (v,v') \in R$, some automorphism $\sigma_{r} : [\ker(\pi_{r})] \to [\ker(\pi_{r})]$ such that the diagram 
    $$
        \xymatrix{ 
[\mathrm{\ker}(\pi_{r})] \ar@{->}[r]^{\sigma_{r}} \ar@{->}[d] & [\ker(\pi_{r})] \ar@{->}[d] \\
[\ker(\pi_{v})] \times [\ker(\pi_{v'})] \ar@{->}[r]^{(\sigma_{v},\sigma_{v'})} & [\ker(\pi_{v})] \times [\ker(\pi_{v'})] ,
}
    $$
   commutes, where the vertical arrows are the natural quotient maps (whose graph is given by the relation $C$). Then there is a unique $\sigma_{*} \in \mathrm{Aut}(S(G_{\Gamma}))$ such that $\sigma_{*}|_{[\ker(\pi_{v})]} = \sigma_{v}$ and $\sigma_{*}|_{[\ker(\pi_{r})]} = \sigma_{r}$ for all $v \in V$ and $r \in R$. 
\end{cor}

\begin{proof}
    Take $G = G_{\Gamma}$ and $v \mapsto \ker(\pi_{v})$ in Lemma \ref{lem: unique iso}. 
\end{proof}

\subsection{The universal $\ell$-Frattini cover} In general, universal Frattini covers of a profinite group $G$ can be analyzed as fiber products of Frattini covers of $G$ whose kernels are pro-$\ell$ for different primes $\ell$.  Taking this point of view, we are able to explicitly describe the universal Frattini cover $\widetilde{G_{\Gamma}}$ of $G_{\Gamma}$ as a certain semi-direct product of free pro-$\ell$ groups for $\ell \in \{2,p,q\}$. This digression will take us a little far afield from the model-theoretic issues of interest for us, but this level of granularity will turn out to be important for the description of algebraic closure and independence in the theories $\mathrm{Th}(S(\widetilde{G_{\Gamma}}))$.

We use the notation $F_{n}$ to refer to the free profinite group of rank $n$ and $F_{n}(p)$ to denote the free pro-$p$ group of rank $n$. This makes sense even when $n$ is not a natural number but an infinite cardinal. 

\begin{defn}
    Let $\ell$ be a fixed prime.
    \begin{enumerate}
        \item We say that an epimorphism of profinite groups $\pi : G \to H$ is an $\ell$\emph{-cover} if $\ker(\pi)$ is a pro-$\ell$ group.
        \item We say that an $\ell$-cover $\pi: G \to H$ is $\ell$\emph{-Frattini} if $\ker(\pi) \subseteq \Phi(G)$, i.e. $\pi$ is both an $\ell$- and Frattini cover.
    \end{enumerate}
\end{defn}

\begin{fact} \cite[Construction 22.11.5]{fried2008field}
    Let $\ell$ be a prime number. If $G$ is a profinite group, then it has a universal $\ell$-Frattini cover $\pi_{\ell}: {_{\ell}\tilde{G}} \to G$.  That is, $\pi_{\ell}$ is an $\ell$-Frattini cover and, if $\rho: H \to G$ is $\ell$-Frattini, then there is an epimorphism $\chi : {_{\ell}\tilde{G}} \to H$ so that the following diagram commutes:
       $$
\xymatrix{ & {_{\ell}\tilde{G}}\ar@{-->>}[dl]^{\chi}  \ar@{->>}[d]^{\pi_{\ell}} \\
H\ar@{->>}[r]^{\rho} & G }
$$
Moreover, the universal $\ell$-Frattini cover may be constructed as follows.  Suppose $\pi : \tilde{G} \to G$ is the universal Frattini cover. Then since $\ker(\pi)$ is a closed subgroup of $\Phi(\tilde{G})$ and $\Phi(\tilde{G})$ is pro-nilpotent \cite[Lemma 22.1.2]{fried2008field}, it follows that $\ker(\pi) = \prod_{p' \text{ prime}} K_{p'}$ where $K_{p'}$ is pro-$p'$ \cite[Proposition 2.3.8]{ribes2010profinite}.  Then, setting $K'_{\ell} = \prod_{p' \neq \ell} K_{p'}$, we let ${_{\ell}\tilde{G}} = \tilde{G}/K'_{\ell}$ and let $\pi_{\ell} : {_{\ell}\tilde{G}} \to G$ be the map induced by $\pi$.  Then $\pi_{\ell}: {_{\ell}\tilde{G}} \to G$ is a universal $\ell$-Frattini cover.
\end{fact}

The previous construction applies equally well to a set of $\Pi$ of primes. Extend the definition of $\ell$-Frattini to $\Pi$\emph{-Frattini} in the obvious way: a cover is $\Pi$\emph{-Frattini} if the kernel is a pro-$\Pi$-group (i.e. the only primes that divide the order of any of its finite quotients lie in $\Pi$). One can define $K'_{\Pi} = \prod_{\ell \not\in \Pi} K_{\ell}$ and define ${_{\Pi}\tilde{G}} = \tilde{G}/K'_{\Pi}$ and let $\pi_{\Pi}: {_{\Pi}\tilde{G}} \to G$ denote the induced map. Since $\pi_{\Pi}$ is a Frattini cover and its $\ell$-Sylows are free pro-$\ell$ for all $\ell \in \Pi$, an identical argument to \cite[Construction 22.11.5]{fried2008field} establishes that this map is a universal $\Pi$-Frattini cover.

Specializing to the case of $G = G_{\Gamma}$, the relevant primes are $2$, $p$, and $q$. We will denote by ${_{pq}\widetilde{G_{\Gamma}}}$ the universal $\{p,q\}$-Frattini cover (and refer to this as the universal $pq$-Frattini cover of $G_{\Gamma}$); similarly ${_{pq}\widetilde{H}}$ will denote the universal $pq$-Frattini cover of any profinite group $H$. 

The following fact follows by an identical argument to \cite[Lemma 22.6.3]{fried2008field}:

\begin{fact} \label{fact: pq quotient fact}
Let $\tilde{\varphi}: {_{pq}\tilde{G}} \to G$ be the universal $pq$-Frattini cover of a profinite group $G$. Then a profinite group $H$ is a quotient of ${_{pq}\tilde{G}}$ if and only if $H$ is a $pq$-Frattini cover of a quotient of $G$. 
\end{fact}

\begin{fact} \label{fact:Ribes} \cite[Theorem 3.4]{MR792361}
    Let $p_{1}, \ldots, p_{n}$ be distinct prime numbers and let $t_{1}, \ldots, t_{n}$ be natural numbers. Suppose $G$ is a profinite group that admits a subnormal series 
    $$
    G = G_{0} \vartriangleright G_{1} \vartriangleright \ldots \vartriangleright G_{n} = 1
    $$
    such that $G_{i-1}/G_{i}$ is a pro-$p_{i}$ group of rank $t_{i}$. Assume $\tilde{G}$ is a profinite group with subnormal series 
    $$
    \tilde{G} = \tilde{G}_{0} \vartriangleright \tilde{G}_{1} \vartriangleright \ldots \vartriangleright \tilde{G}_{n} = 1,
    $$
    such that $\tilde{G}_{i-1}/\tilde{G}_{i}$ is the free pro-$p_{i}$ group of rank $t_{i}$.  Then any epimorphism $\tilde{\gamma}: \tilde{G} \to G$ is a universal Frattini cover of $G$. 
\end{fact}

We want to characterize the universal Frattini cover of the group $G_{\Gamma}$ for a graph $\Gamma = (V,R)$, or more generally, of groups of the form $G_{\Gamma} \times C_{2}^{I}$ for an index set $I$ (which we may suppose is disjoint from $V$). We use the description $G_{\Gamma} = C_{q}^{R} \rtimes D_{p}^{V} = (C_{q}^{R}) \rtimes (C_{p}^{V} \rtimes C_{2}^{V})$, from which we also get a description $G_{\Gamma} \times C_{2}^{I} = (C_{q}^{R}) \rtimes (C_{p}^{V} \rtimes C_{2}^{V \cup I})$, where the $C_{2}^{I}$ factor acts trivially. Let $(z_{r})_{r \in R}$ be a set of generators converging to $1$ for $F_{|R|}(q)$, and likewise generators $(y_{v})_{v \in V}$ and $(x_{v})_{v \in V \cup I}$ for $F_{|V|}(p)$ and $F_{|V \cup I|}(2)$ respectively.  

Throughout this section, $\langle X \rangle$ denotes the closed subgroup \emph{topologically} generated by $X$ in a profinite group, not merely the abstractly generated subgroup.
If $\Gamma_{0} = (V_{0},R_{0})$ is a finite induced subgraph of $\Gamma$ with $|V_{0}| = n$ and $I_{0} \subseteq I$ is a subset with $|I_{0}| = m$, then there is a quotient $\langle y_{v} : v \in V_{0} \rangle \cong F_{n}(p)$ of $F_{|V|}(p)$ obtained by mapping $y_{v} \mapsto y_{v}$ for $v \in V_{0}$ and $y_{v} \mapsto 1$ for all $v \in V \setminus V_{0}$. We can define an action of $C_{2}^{V_{0} \cup I_{0}}$ on $\langle y_{v} : v \in V_{0} \rangle$ by setting, for all $\overline{a} = (a_{v})_{v \in V_{0} \cup I_{0}} \in C_{2}^{V_{0} \cup I_{0}}$, $\overline{a} \cdot y_{v} = y_{v}^{a_{v}}$ for all $v \in V_{0}$. Notice that these actions are compatible in the sense that if $\Gamma_{0} = (V_{0},R_{0}) \subseteq (V_{1},R_{1})$ are both finite induced subgraphs of $\Gamma$ and $I_{0} \subseteq I_{1} \subseteq I$, then we have a diagram 
 $$
    \xymatrix{ 
C_{2}^{V_{1} \cup I_{1}} \times \langle y_{v} : v \in V_{1}  \rangle \ar@{->}[r] \ar@{->}[d] & \langle y_{v} : v \in V_{1}  \rangle \ar@{->}[d] \\
C_{2}^{V_{0} \cup I_{0}} \times \langle y_{v} : v \in V_{0}  \rangle \ar@{->}[r] & \langle y_{v} : v \in V_{0}  \rangle
}
    $$
   where the horizontal maps are the respective actions, the map $C_{2}^{V_{1} \cup I_{1}} \to C_{2}^{V_{0} \cup I_{0}}$ is projection to the $V_{0} \cup I_{0}$ coordinates, and the map $\langle y_{v} : v \in V_{1} \rangle \to \langle y_{v} : v \in V_{0} \rangle$ maps $y_{v} \mapsto y_{v}$ if $v \in V_{0}$ and maps $y_{v} \mapsto 1$ otherwise. By passing to the inverse limit, we obtain a continuous action of $C_{2}^{V \cup I}$ on $F_{|V|}(p)$. This action has the property that if $\overline{a} = (a_{v})_{v \in V \cup I} \in C_{2}^{V \cup I}$, then $\overline{a} \cdot y_{v} = y_{v}^{a_{v}}$. Let $F_{|V|}(p) \rtimes C_{2}^{V \cup I}$ be the associated semidirect product and let $\rho : F_{|V|}(p) \rtimes C_{2}^{V \cup I} \to C_{2}^{V \cup I}$ denote the quotient map. 

Now we define an action of $C_{2}^{V \cup I}$ on $F_{|R|}(q)$. If $\Gamma_{0} = (V_{0},R_{0})$ is a finite induced subgraph with $|R_{0}| = n$ and $I_{0} \subseteq I$ is a finite subset, then there is an action $C_{2}^{V_{0} \cup I_{0}}$ on $\langle z_{r} : r \in R_{0} \rangle \cong F_{n}(q)$ defined so that, if $r = (v,v')$ and $\overline{a} = (a_{v})_{v \in V_{0} \cup I_{0}}$, then $\overline{a} \cdot z_{r} = z_{r}^{a_{v}a_{v'}}$. As above, these actions cohere and we obtain a continuous action of $C_{2}^{V \cup I}$ on $F_{|R|}(q)$ by passage to the inverse limit, with the property that $\overline{a} = (a_{v})_{v \in V \cup I} \in C_{2}^{V \cup I}$ and $r = (v,v')$, then $\overline{a} \cdot z_{r} = z_{r}^{a_{v}a_{v'}}$. We obtain, then, an action of $F_{|V|}(p) \rtimes C_{2}^{V \cup I}$ on $F_{|R|}(q)$ by $g \cdot h = \rho(g) \cdot h$ for all $g \in F_{|V|}(p) \rtimes C_{2}^{V \cup I}$ and $h \in F_{|R|}(q)$, where the action on the right-hand side of the equation is the action of $C_{2}^{V \cup I}$ we just defined. Let ${_{pq}(\widetilde{G_{\Gamma} \times C_{2}^{I}})}$ be defined to be the semi-direct product:
$$
{_{pq}(\widetilde{G_{\Gamma} \times C_{2}^{I}})} = F_{|R|}(q) \rtimes (F_{|V|}(p) \rtimes C_{2}^{V \cup I}).
$$
There is an epimorphism $\pi_{pq} : {_{pq}(\widetilde{G_{\Gamma} \times C_{2}^{I}})} \to (G_{\Gamma} \times C_{2}^{I})$ defined by mapping $y_{v} \mapsto (\gamma)_{v}$ for all $v \in V$, where $\gamma$ denotes a generator of $C_{p}$, and mapping $z_{r} \mapsto (\delta)_{r}$ for all $r \in R$, where $\delta$ denotes a generator of $C_{q}$. Our construction ensures that the resulting maps $F_{|V|}(p) \to C_{p}^{V}$ and $F_{|R|}(q) \to C_{q}^{R}$ are compatible with the semi-direct product structure on ${_{pq}(\widetilde{G_{\Gamma}\times C_{2}^{I}})}$ so we obtain a well-defined epimorphism. 

Note that we have explicitly constructed ${_{pq}(\widetilde{G_{\Gamma} \times C_{2}^{I}})}$ as the inverse limit of $\langle z_{r} : r \in R_{0} \rangle \rtimes (\langle y_{v} : v \in V_{0} \rangle \rtimes C_{2}^{V_{0} \cup I_{0}})$ for finite induced subgraphs $\Gamma_{0} = (V_{0},R_{0}) \subseteq \Gamma$ and finite subsets $I_{0} \subseteq I$. We may view each such group as a quotient of 
$$
\langle z_{r} : r \in R_{0} \rangle \rtimes (\langle y_{v} : v \in V_{0} \rangle \rtimes \langle x_{v}:  v \in V_{0} \cup I_{0} \rangle)
$$
where $\langle x_{v} : v \in V_{0} \cup I_{0} \rangle$ maps to $C_{2}^{V_{0} \cup I_{0}}$ by sending $x_{v} \mapsto (-1)_{v}$ for all $v \in V_{0} \cup I_{0}$. By taking the inverse limit over all finite induced subgraphs, we obtain a group
$$
\reallywidetilde{G_{\Gamma} \times C_{2}^{I}} = F_{|R|}(q) \rtimes (F_{|V|}(p) \rtimes F_{|V\cup I|}(2)). 
$$
We will denote by $\pi_{2} : \widetilde{G_{\Gamma} \times C_{2}^{I}} \to {_{pq}\widetilde{(G_{\Gamma} \times C_{2}^{I})}}$ the quotient map obtained by mapping each $x_{v} \mapsto (-1)_{v}$ for $v \in V \cup I$. 

\begin{lem}
    The map $\pi_{pq} \circ \pi_{2} : \widetilde{G_{\Gamma} \times C_{2}^{I}} \to G_{\Gamma} \times C_{2}^{I}$ is a universal Frattini cover of $G_{\Gamma} \times C_{2}^{I}$.
\end{lem}

\begin{proof}
    Note that if $\Gamma = (V,R)$ is finite, with $|V| = n$, $|R| = m$, and $|I| = k$, then we have that $G_{\Gamma} \times C_{2}^{I}$ has a subnormal series 
    $$
    G_{\Gamma} \times C_{2}^{I} = C_{q}^{R} \rtimes (C_{p}^{V} \rtimes C_{2}^{V \cup I}) \vartriangleright C_{p}^{V} \rtimes C_{2}^{V \cup I} \vartriangleright C_{2}^{V \cup I} \vartriangleright 1
    $$
    with quotients being $q-$, $p-$, and $2$-groups of rank $m$, $n$ and $n+k$, respectively.  Then, by construction, $\widetilde{G_{\Gamma} \times C_{2}^{I}}$ can be described as 
    $$
    \widetilde{G_{\Gamma} \times C_{2}^{I}} = F_{m}(q) \rtimes (F_{n}(p) \rtimes F_{n+k}(2)) \vartriangleright F_{n}(p) \rtimes F_{n+k}(2) \vartriangleright F_{n+k}(2) \vartriangleright 1,
    $$
    so by Fact \ref{fact:Ribes}, the epimorphism $\pi_{pq} \circ \pi_{2} : \widetilde{G_{\Gamma} \times C_{2}^{I}} \to G_{\Gamma} \times C_{2}^{I}$ is a universal Frattini cover. 

    For the general case, we note that the map $\pi_{pq} \circ \pi_{2}: \widetilde{G_{\Gamma} \times C_{2}^{I}} \to G_{\Gamma} \times C_{2}^{I}$ is the inverse limit (in the obvious sense) of the maps $\widetilde{G_{\Gamma_{0}} \times C_{2}^{I_{0}}} \to G_{\Gamma_{0}} \times C_{2}^{I_{0}}$ for $\Gamma_{0} \subseteq \Gamma$ ranging over the finite induced subgraphs of $\Gamma$ and $I_{0} \subseteq I$ ranging over the finite subsets of $I$. By the finite case and Fact \ref{fact: Frattini char}, the map $\pi_{pq} \circ \pi_{2}$ is a Frattini cover. Moreover, $\widetilde{G_{\Gamma} \times C_{2}^{I}}$ is projective: by construction, its $q$-, $p$-, and $2$-Sylow subgroups are isomorphic to the free pro-$\ell$ groups $F_{|R|}(q)$, $F_{|V|}(p)$, and $F_{|V \cup I|}(2)$, respectively, and a profinite group all of whose Sylow subgroups are free pro-$\ell$ is projective \cite[Proposition 22.10.4]{fried2008field}. Since a Frattini cover with projective source is the universal Frattini cover by Fact \ref{fact2}(2), the map $\pi_{pq} \circ \pi_{2}$ is the universal Frattini cover of $G_{\Gamma} \times C_{2}^{I}$.
\end{proof}

\begin{lem}
    Suppose $\Gamma = (V,R)$ is a graph, $I$ is a set, and $\pi = \pi_{pq} \circ \pi_{2} : \widetilde{G_{\Gamma} \times C_{2}^{I}} \to G_{\Gamma} \times C_{2}^{I}$ is the universal Frattini cover, with $\pi_{2}$ and $\pi_{pq}$ defined as above. Then $\ker(\pi_{2})$ is the unique closed normal subgroup $M$ of $\widetilde{G_{\Gamma} \times C_{2}^{I}}$ contained in $\ker(\pi)$ such that the induced map $\widetilde{(G_{\Gamma} \times C_{2}^{I})}/M \to \widetilde{(G_{\Gamma} \times C_{2}^{I})}/\ker(\pi)$ is a universal $pq$-Frattini cover.
\end{lem}

\begin{proof}
    It is clear that $\ker(\pi_{2})$ is at least one closed normal subgroup of $\widetilde{G_{\Gamma} \times C_{2}^{I}}$ with the desired property. Moreover, any such normal subgroup must be equal to the $2$-Sylow subgroup of $\Phi(\widetilde{G_{\Gamma}\times C_{2}^{I}})$ since $\Phi(G_{\Gamma} \times C_{2}^{I}) = 1$ and is therefore uniquely determined. 
\end{proof}

\begin{lem} \label{lem:unique pq cover}
    Suppose $\Gamma = (V,R)$ is a graph, $I$ a set, $\Gamma_{0} = (V_{0},R_{0}) \subseteq \Gamma$ is an induced subgraph and $I_{0} \subseteq I$ a finite subset. Let $\pi_{0} : G_{\Gamma} \times C_{2}^{I} \to G_{\Gamma_{0}} \times C_{2}^{I_{0}}$ be the projection map. Let $N$ be the kernel of the map $\pi_{0} \circ \pi_{pq} : {_{pq}\widetilde{(G_{\Gamma} \times C_{2}^{I})}} \to G_{\Gamma_{0}} \times C_{2}^{I_{0}}$. Then there is a unique closed normal subgroup $M \unlhd {_{pq}\widetilde{(G_{\Gamma} \times C_{2}^{I})}}$ with $M \subseteq N$ and the natural map ${_{pq}\widetilde{(G_{\Gamma}\times C_{2}^{I})}}/M \to {_{pq}\widetilde{(G_{\Gamma} \times C_{2}^{I})}}/N$ is a universal $pq$-Frattini cover.
\end{lem}

\begin{proof}
    There is clearly at least one such $M$, since $_{pq}\widetilde{(G_{\Gamma} \times C_{2}^{I})}/N \cong G_{\Gamma_{0}} \times C_{2}^{I_{0}}$ and ${_{pq}\widetilde{(G_{\Gamma_{0}}\times C_{2}^{I_{0}})}}$ is a quotient of ${_{pq}\widetilde{(G_{\Gamma}\times C_{2}^{I})}}$.  More specifically, we recall that we may identify ${_{pq}\widetilde{(G_{\Gamma_{0}}\times C_{2}^{I_{0}})}} = \langle z_{r} :r \in R_{0} \rangle \rtimes (\langle y_{v} : v \in V_{0} \rangle \rtimes C_{2}^{V_{0} \cup I_{0}})$ as a quotient of ${_{pq}\widetilde{(G_{\Gamma}\times C_{2}^{I})}}$ via the map that sends $z_{r} \mapsto z_{r}$ if $r \in R_{0}$ and $z_{r} \mapsto 1$ otherwise, $y_{v} \mapsto y_{v}$ if $v \in V_{0}$ and $y_{v} \mapsto 1$ otherwise, and $C_{2}^{V \cup I}$ mapping to $C_{2}^{V_{0} \cup I_{0}}$ by projection to the $V_{0} \cup I_{0}$ coordinates. Let $M_{*}$ denote the kernel of this mapping, so we have ${_{pq}\widetilde{(G_{\Gamma}\times C_{2}^{I})}}/M_{*} \cong {_{pq}\widetilde{(G_{\Gamma_{0}}\times C_{2}^{I_{0}})}}$ and the natural map ${_{pq}\widetilde{(G_{\Gamma} \times C_{2}^{I}})}/M_{*} \to {_{pq}\widetilde{(G_{\Gamma}\times C_{2}^{I})}}/N$ is a universal $pq$-Frattini cover.

    Now suppose $M$ is an arbitrary closed normal subgroup of ${_{pq}\widetilde{(G_{\Gamma}\times C_{2}^{I})}}$ contained in $N$ such that the natural map $\pi_{MN}: {_{pq}\widetilde{(G_{\Gamma}\times C_{2}^{I})}}/M \to {_{pq}\widetilde{(G_{\Gamma}\times C_{2}^{I})}}/N$ is a universal $pq$-Frattini cover. We will show $M = M_{*}$. Let $\pi: {_{pq}\widetilde{(G_{\Gamma} \times C_{2}^{I})}} \to {_{pq}\widetilde{(G_{\Gamma} \times C_{2}^{I})}}/M$ denote the quotient map. First note that if $v \in V \cup I$ and $(\pi_{MN} \circ \pi)((-1)_{v}) = 1$ then $(-1)_{v} \in \ker(\pi)$ since $\ker(\pi_{MN})$ has no nontrivial elements of order $2$, and thus $(-1)_{v} \in M$. It follows, then, that $(-1)_{v} \in M$ for all $v \in (V \cup I) \setminus (V_{0} \cup I_{0})$. Then, by definition of the action, since $(-1)_{v} \in M$, we have $y_{v}M = y^{-1}_{v}M$ for all $v \in V\setminus V_{0}$.  Since $y_{v}M$ cannot have order $2$, we get $y_{v} \in M$ for all $v \in V \setminus V_{0}$. Arguing similarly, if $r = (v,v') \in R \setminus R_{0}$, then, without loss of generality, $v \not\in V_{0}$ so, since $(-1)_{v} \in M$, we have $z_{r}M = z^{-1}_{r}M$ and since $z_{r}M$ cannot have order two, we again get $z_{r} \in M$. By the choice of $M_{*}$, this yields that $M_{*} \subseteq M$ so we have natural quotient map ${_{pq}\widetilde{(G_{\Gamma}\times C_{2}^{I})}}/M_{*} \to {_{pq}\widetilde{(G_{\Gamma} \times C_{2}^{I})}}/M$ and the composition with the quotient map ${_{pq}\widetilde{(G_{\Gamma}\times C_{2}^{I})}}/M \to {_{pq}\widetilde{(G_{\Gamma}\times C_{2}^{I})}}/N$ is a universal $pq$-Frattini cover, so ${_{pq}\widetilde{(G_{\Gamma} \times C_{2}^{I}})}/M_{*} \to {_{pq}\widetilde{(G_{\Gamma}\times C_{2}^{I})}}/M$ is also a universal $pq$-Frattini cover, which must therefore be an isomorphism. This shows $M = M_{*}$. 
\end{proof}

\begin{cor} \label{cor: unique subgroup}
    Suppose $\Gamma = (V,R)$ is a graph, $I$ is a set, $\Gamma_{0} = (V_{0},R_{0}) \subseteq \Gamma$ is an induced subgraph, and $I_{0} \subseteq I$ is a finite set. Let $\pi : \widetilde{(G_{\Gamma}\times C_{2}^{I})} \to (G_{\Gamma}\times C_{2}^{I})$ denote the universal Frattini cover. There is a unique closed normal subgroup $M \unlhd \widetilde{(G_{\Gamma}\times C_{2}^{I})}$ with $M \subseteq \ker(\pi_{\Gamma_{0}} \circ \pi)$ such that the induced map $\overline{\pi_{\Gamma_{0}} \circ \pi} : \widetilde{(G_{\Gamma}\times C_{2}^{I})}/M \to (G_{\Gamma_{0}} \times C_{2}^{I_{0}})$ is the universal $pq$-Frattini cover. 
\end{cor}

\begin{proof}
Let $S$ denote the $2$-Sylow subgroup of $\ker(\pi)$. Then $S$ is the unique closed normal subgroup of $\widetilde{G_{\Gamma}\times C_{2}^{I}}$ contained in $\ker(\pi)$ with $\widetilde{G_{\Gamma}\times C_{2}^{I}}/S \cong {_{pq}\widetilde{(G_{\Gamma}\times C_{2}^{I})}}$ and such that the induced map from $\widetilde{(G_{\Gamma}\times C_{2}^{I})}/S$ to $G_{\Gamma} \times C_{2}^{I}$ is the universal $pq$-Frattini cover. Let $\rho : \widetilde{(G_{\Gamma} \times C_{2}^{I})}/S \to (G_{\Gamma} \times C_{2}^{I})$ be the map induced by $\pi$. Let $\pi_{0}: G_{\Gamma} \times C_{2}^{I} \to G_{\Gamma_{0}} \times C_{2}^{I_{0}}$ denote the projection (as in Lemma \ref{lem:unique pq cover}). Since $\widetilde{(G_{\Gamma} \times C_{2}^{I})}/S \cong {_{pq}\widetilde{(G_{\Gamma}\times C_{2}^{I})}}$, we may apply Lemma \ref{lem:unique pq cover} to $(\pi_{0} \circ \rho) : \widetilde{(G_{\Gamma}\times C_{2}^{I})}/S \to (G_{\Gamma_{0}}\times C_{2}^{I_{0}})$ and pulling back to $\widetilde{(G_{\Gamma} \times C_{2}^{I})}$, to obtain a unique closed normal $\overline{M} \unlhd \widetilde{(G_{\Gamma}\times C_{2}^{I})}/S$ with $\overline{M} \subseteq \ker(\pi_{0} \circ \rho)$ such that the induced map $(\widetilde{(G_{\Gamma} \times C_{2}^{I})}/S)/\overline{M} \to (G_{\Gamma_{0}} \times C_{2}^{I_{0}})$ is the universal $pq$-Frattini cover. Such an $\overline{M}$ corresponds uniquely to a closed normal $M$ with $S \subseteq M \unlhd \widetilde{(G_{\Gamma}\times C_{2}^{I})}$ and with $M \subseteq \ker(\pi_{0} \circ \pi)$, as desired. 
\end{proof}

\begin{rem}
    Although Corollary \ref{cor: unique subgroup} is stated for finite $I_{0}$, it holds equally well when $I_{0}$ is an arbitrary (possibly infinite) subset of $I$. Indeed, write $I_{0} = \varinjlim_{F} F$ as a directed union of its finite subsets $F \subseteq I_{0}$. For each finite $F$, Corollary \ref{cor: unique subgroup} gives a unique $M_{F}$ such that $\widetilde{(G_{\Gamma} \times C_{2}^{I})}/M_{F} \to G_{\Gamma_{0}} \times C_{2}^{F}$ is the universal $pq$-Frattini cover. By uniqueness, the $M_{F}$ are compatible as $F$ varies, and $M = \bigcap_{F} M_{F}$ is the unique closed normal subgroup of $\widetilde{(G_{\Gamma} \times C_{2}^{I})}$ such that $\widetilde{(G_{\Gamma} \times C_{2}^{I})}/M \to G_{\Gamma_{0}} \times C_{2}^{I_{0}}$ is the universal $pq$-Frattini cover, using the fact that the universal $pq$-Frattini cover commutes with inverse limits of finite groups.
\end{rem}

\begin{lem} \label{lem:nfoldfiber}
    Let $\Gamma = (V,R)$ be a graph and $I$ be a set disjoint from $V$, and fix a natural number $n \geq 1$. 
    Let $N_{1} \subseteq N_{0}$ be closed normal subgroups of $\widetilde{G_{\Gamma} \times C_{2}^{I}}$ such that $\widetilde{(G_{\Gamma} \times C_{2}^{I})}/N_{0} \cong G_{\Gamma} \times C_{2}^{I}$ and the quotient map $\widetilde{(G_{\Gamma} \times C_{2}^{I})}/N_{1} \to \widetilde{(G_{\Gamma} \times C_{2}^{I})}/N_{0}$ is a universal $pq$-Frattini cover. 
    Let $S_{pq} \unlhd \widetilde{G_{\Gamma} \times C_{2}^{I}}$ denote the (normal) product of the $p-$ and $q-$Sylow subgroups of $\widetilde{G_{\Gamma} \times C_{2}^{I}}$. 

    Let $M$ be an arbitrary closed normal subgroup of $\widetilde{(G_{\Gamma} \times C_{2}^{I})}$ with $M \subseteq N_{1}$. Set $A = \widetilde{(G_{\Gamma} \times C_{2}^{I})}/M$ and $B = \widetilde{(G_{\Gamma} \times C_{2}^{I})}/MS_{pq}$. 
    
    Identifying $F_{|V \cup I|}(2)$ with $\widetilde{(G_{\Gamma} \times C_{2}^{I})}/S_{pq}$ so that the maps $\widetilde{(G_{\Gamma} \times C_{2}^{I})} \to A$, $A \to B$, and $F_{|V \cup I|}(2) \to B$ are the quotient maps, we have the following:
    \begin{enumerate}
        \item $\widetilde{(G_{\Gamma} \times C_{2}^{I})} \cong A \times_{B} F_{|V \cup I|}(2)$.
        \item Let $P_{n}$ denote the group 
        $$
        \widetilde{(G_{\Gamma} \times C_{2}^{I})} \times_{A} \widetilde{(G_{\Gamma} \times C_{2}^{I})} \times_{A} \ldots \times_{A} \widetilde{(G_{\Gamma} \times C_{2}^{I})}
        $$
        which is the $n$ fold fiber product of $\widetilde{(G_{\Gamma} \times C_{2}^{I})}$ over $A$ and let $H_{n}$ denote the group 
        $$
        F_{|V \cup I|}(2) \times_{B} F_{|V \cup I|}(2) \times_{B} \ldots \times_{B} F_{|V \cup I|}(2)
        $$
        the $n$-fold fiber product of $F_{|V \cup I|}(2)$ over $B$.  Then we have an isomorphism 
        $$
        P_{n} \cong A \times_{B} H_{n}.
        $$
    \end{enumerate}
\end{lem}

\begin{proof}
    (1)  Since $\widetilde{(G_{\Gamma} \times C_{2}^{I})}/N_{1} \cong {_{pq}\widetilde{(G_{\Gamma}\times C_{2}^{I})}}$, the subgroup $N_{1}$ is necessarily a $2$-group. Thus, since $M \subseteq N_{1}$ and the subgroup $S_{pq}$ has only quotients whose order is divisible by $p$ or $q$, we have $M \cap S_{pq} = 1$ so 
    \begin{eqnarray*}
        \widetilde{(G_{\Gamma} \times C_{2}^{I})} &=& \widetilde{(G_{\Gamma} \times C_{2}^{I})}/(M \cap S_{pq}) \\
        &\cong& (\widetilde{(G_{\Gamma} \times C_{2}^{I})}/M) \times_{\widetilde{(G_{\Gamma} \times C_{2}^{I})}/MS_{pq}} (\widetilde{(G_{\Gamma}\times C_{2}^{I}})/S_{pq}) \\
        &\cong& A \times_{B} F_{|V \cup I|}(2),
    \end{eqnarray*}
    as desired. 

    (2) We will represent elements of $P_{n}$ and $H_{n}$ as $n$-tuples from $\widetilde{(G_{\Gamma} \times C_{2}^{I})}^{n}$ or $F_{|V \cup I|}(2)^{n}$ with the property that they have the same image in $A$ or in $B$, respectively. By (1), there is an isomorphism $\psi : A \times_{B} F_{|V \cup I|}(2) \to \widetilde{(G_{\Gamma} \times C_{2}^{I})}$. Then we define a map $\overline{\psi} : A \times_{B} H_{n} \to P_{n}$ by 
    $$
    \overline{\psi}(g,\overline{h}) = (\psi(g,h_{1}), \psi(g,h_{2}), \ldots, \psi(g,h_{n})). 
    $$
    It is easy to see that this defines a map to $P_{n}$ and, since $\psi$ is an isomorphism, $\overline{\psi}$ is injective. To see that it is surjective, let $\overline{a} = (a_{1},\ldots, a_{n}) \in P_{n}$. It follows from (1) that we have a commutative diagram 
     $$
    \xymatrix{ 
\widetilde{(G_{\Gamma} \times C_{2}^{I})} \ar@{->}[r] \ar@{->}[d] & A \ar@{->}[d] \\
F_{|V \cup I|}(2) \ar@{->}[r] & B,
}
    $$
    where all the maps are quotient maps (using the identification of $F_{|V \cup I|}(2)$ with $\widetilde{(G_{\Gamma} \times C_{2}^{I})}/S_{pq}$). Thus, if we let $g$ be the (common) image of each $a_{i}$ in $A$ and let $h_{i}$ be the image of $a_{i}$ in $F_{|V \cup I|}(2)$ then $g$ and $h_{i}$ have the same image in $B$.  It follows that, letting $\overline{h} = (h_{1},\ldots, h_{n})$, we have $(g,\overline{h}) \in A \times_{B} H_{n}$ and $\overline{\psi}(g,\overline{h}) = \overline{a}$. 
\end{proof}

\begin{rem} 
To see what the above lemma is saying, it can be helpful to consider the special case that $M = N_{1}$.  Then we obtain
    \begin{enumerate}
        \item $\widetilde{(G_{\Gamma} \times C_{2}^{I})} \cong {_{pq}\widetilde{(G_{\Gamma} \times C_{2}^{I})}} \times_{C_{2}^{V \cup I}} F_{|V \cup I|}(2)$.
        \item If given 
        $$
        P_{n} := \widetilde{(G_{\Gamma} \times C_{2}^{I})} \times_{{_{pq}\widetilde{(G_{\Gamma}\times C_{2}^{I})}}} \widetilde{(G_{\Gamma} \times C_{2}^{I})} \times_{{_{pq}\widetilde{(G_{\Gamma}\times C_{2}^{I})}}} \ldots \times_{{_{pq}\widetilde{(G_{\Gamma}\times C_{2}^{I})}}} \widetilde{(G_{\Gamma} \times C_{2}^{I})}
        $$
        which is the $n$-fold fiber product of $\widetilde{(G_{\Gamma} \times C_{2}^{I})}$ over ${_{pq}\widetilde{(G_{\Gamma}\times C_{2}^{I})}}$ and 
        $$
        H_{n} := F_{|V \cup I|}(2) \times_{C_{2}^{V \cup I}} F_{|V \cup I|}(2) \times_{C_{2}^{V \cup I}} \ldots \times_{C_{2}^{V \cup I}} F_{|V \cup I|}(2),
        $$
        the $n$-fold fiber product of $F_{|V \cup I|}(2)$ over $C_{2}^{V \cup I}$, then we have an isomorphism 
        $$
        P_{n} \cong {_{pq}\widetilde{(G_{\Gamma}\times C_{2}^{I})}} \times_{C_{2}^{V \cup I}} H_{n}.
        $$
    \end{enumerate}
\end{rem}

\begin{lem} \label{lem: pq fibermaxxing}
    Suppose $\Gamma = (V,R)$ is a graph and $I$ is a set (disjoint from $V$). Let $I_{0} \subseteq I$ be a subset and assume $|I| > |I_{0}| + |V| + \aleph_{0}$. Suppose we are given a profinite group $A$ and epimorphisms $\tilde{\pi_{0}}: \widetilde{(G_{\Gamma} \times C_{2}^{I})} \to \widetilde{(G_{\Gamma} \times C_{2}^{I_{0}})}$, $\beta : \widetilde{(G_{\Gamma} \times C_{2}^{I_{0}})} \to A$, and $\alpha : A \to {_{pq}\widetilde{(G_{\Gamma} \times C_{2}^{I_{0}})}} $ such that $\tilde{\pi_{0}}$ is a lift of the projection $G_{\Gamma} \times C_{2}^{I} \to G_{\Gamma} \times C_{2}^{I_{0}}$ and $\alpha \circ \beta$ is a universal Frattini cover. Fix $n \geq 1$ and let $P_{n}$ be the group defined by 
    $$
    P_{n} := \widetilde{(G_{\Gamma} \times C_{2}^{I_{0}})} \times_{A}\widetilde{(G_{\Gamma} \times C_{2}^{I_{0}})} \times_{A} \ldots \times_{A} \widetilde{(G_{\Gamma} \times C_{2}^{I_{0}})},
    $$
    the $n$-fold fiber product of $\widetilde{G_{\Gamma} \times C_{2}^{I_{0}}}$ over $A$, where the map from $\widetilde{G_{\Gamma} \times C_{2}^{I_{0}}}$ to $A$ is given by $\beta$. Then there is an epimorphism $\gamma : \widetilde{(G_{\Gamma} \times C_{2}^{I})} \to P_{n}$ such that the following diagram commutes: 
    $$
    \xymatrix{ 
\widetilde{(G_{\Gamma} \times C_{2}^{I})} \ar@{->}[r]^{\gamma} \ar@{->}[d]^{\tilde{\pi_{0}}} & P_{n} \ar@{->}[d]^{\overline{\beta}} \\
\widetilde{(G_{\Gamma} \times C_{2}^{I_{0}})} \ar@{->}[r]^{\beta} & A.
}
    $$
    where $\overline{\beta} : P_{n} \to A$ is the map obtained by applying $\beta$ to any one of the $\widetilde{(G_{\Gamma} \times C_{2}^{I_{0}})}$ factors. 
\end{lem}

\begin{proof}
    As in the proof of Lemma \ref{lem:nfoldfiber}, we let $S_{pq} \unlhd \widetilde{(G_{\Gamma} \times C_{2}^{I})}$ denote the product of the $p-$ and $q-$Sylow subgroups of $\widetilde{(G_{\Gamma} \times C_{2}^{I})}$.  Let $D = \widetilde{(G_{\Gamma} \times C_{2}^{I})}/S_{pq}$ and let $C = \widetilde{(G_{\Gamma} \times C_{2}^{I_{0}})}/\tilde{\pi_{0}}(S_{pq})$. Set $\overline{\pi} : D \to C$ be the map induced by $\tilde{\pi_{0}}$. Note that we have $D \cong F_{|V \cup I|}(2)$ and $C \cong F_{|V \cup I_{0}|}(2)$. 

    Let $B = A/(\beta \circ \tilde{\pi_{0}})(S_{pq})$ be the maximal pro$-2$ quotient of $A$ and let $\delta : C \to B$ be the epimorphism induced by $\beta$. We define the group $H_{n}$ by setting 
    $$
    H_{n} = C \times_{B} C \times_{B} \ldots \times_{B} C,
    $$
    the $n$-fold fiber product of $C$ with itself over $B$, defined using the epimorphism $\delta$. By Lemma \ref{lem:nfoldfiber}, we may identify $P_{n}$ with $A \times_{B} H_{n}$, where the map $A \to B$ is the quotient map. 

    Then by superprojectivity (using that $D$ is a free pro-$2$ group of sufficiently high rank), there is some epimorphism $\chi : D \to H_{n}$ so that the following diagram commutes:
    $$
    \xymatrix{ & D \ar@{-->>}[dl]^{\chi}  \ar@{->>}[d] \\
    H_{n}\ar@{->>}[r] & B},
    $$
    where the horizontal map is obtained by applying $\delta$ to any of the $C$ factors, and the vertical map is the composition $\beta \circ \tilde{\pi}_{0}$ followed by the quotient map $A \to B$. Let $\tilde{\chi} : \widetilde{(G_{\Gamma} \times C_{2}^{I})} \to H_{n}$ be the map obtained by precomposing $\chi$ with the quotient map $\widetilde{(G_{\Gamma} \times C_{2}^{I})} \to D$. 

    We define a map $\gamma : \widetilde{(G_{\Gamma} \times C_{2}^{I})} \to A \times_{B} H_{n}$ by setting $\gamma(g) = ((\beta \circ \tilde{\pi_{0}})(g),\tilde{\chi}(g))$. By the choice of $\chi$, this defines a map to $A \times_{B} H_{n}$.  Let $M = \ker(\beta \circ \tilde{\pi_{0}})$ and let $M' = \ker(\tilde{\chi})$. Since $B \cong \widetilde{(G_{\Gamma} \times C_{2}^{I})}/MS_{pq}$, we have $MS_{pq} \supseteq M'$, and since $S_{pq} \subseteq M'$, and since $\widetilde{(G_{\Gamma} \times C_{2}^{I})}/M' \cong H_{n}$ is a $2$-group, it follows that $MM' = MS_{pq}$. And, by construction, $(\beta \circ \tilde{\pi_{0}})$ and $\tilde{\chi}$ induce the same isomorphism $\widetilde{G_{\Gamma} \times C_{2}^{I}}/MS_{pq} \to B$. Thus, if $(a,c) \in A \times_{B} H_{n}$ and we choose $g,h$ such that $(\beta \circ \tilde{\pi_{0}})(g) = a$ and $\tilde{\chi}(h) = c$, then we have $gMM' = hMM'$, so there are $m \in M$ and $m' \in M'$ such that $mm' = g^{-1}h$, so if we set $k = gm = h(m')^{-1}$, then we have $\gamma(k) = ((\beta \circ \tilde{\pi_{0}})(g), \tilde{\chi}(h)) = (a,c)$. This shows $\gamma$ is an epimorphism, as desired. 
\end{proof}

Finally, we will record a fact about the way normal subgroups intersect in the maximal pro-$2$ quotient of $\widetilde{G_{\Gamma} \times C_{2}^{I}}$, which will be useful in later applications.  First, we need the following fact, which seems to be folklore:

\begin{fact} \label{fact: normal sylow}
    Suppose $G$ is a profinite group with closed normal $\ell$-Sylow subgroup $S_{\ell}$. Then for all closed normal $A,B \unlhd G$, 
    $$
    S_{\ell}A \cap S_{\ell}B = S_{\ell}(A \cap B). 
    $$
\end{fact}

\begin{proof}
    Replacing $G$ with the quotient $G/(A \cap B)$, it is enough to show that if $A, B \unlhd G$ are closed normal subgroups with $A \cap B = 1$, then $S_{\ell}A \cap S_{\ell}B = S_{\ell}$. Since $A$ and $B$ are both normal, we must have $[A,B] \leq A \cap B = 1$ so $AB \cong A \times B$. Since $S_{\ell}$ is the unique pro-$\ell$ Sylow subgroup of $G$, every pro-$\ell$ subgroup of $G$ (and thus of $A$, $B$, and $AB$) must lie in $S_{\ell}$. This implies 
    $$
    S_{\ell} \cap AB = (S_{\ell} \cap A)(S_{\ell} \cap B),
    $$
    since, in the direct product $A \times B$, the pro-$\ell$ Sylow is the product of the $\ell$-Sylows of $A$ and of $B$. 

    Fix $x \in S_{\ell}A \cap S_{\ell}B$. We can write $x = s_{1}a = s_{2}b$ for $s_{1},s_{2} \in S_{\ell}$, $a \in A$, $b \in B$. Then we have
    $$
    ab^{-1} = s_{1}^{-1}s_{2} \in S_{\ell} \cap AB = (S_{\ell} \cap A)(S_{\ell} \cap B).
    $$
    Write $ab^{-1} = uv$ with $u \in (S_{\ell} \cap A)$ and $v \in (S_{\ell} \cap B)$. Since $A$ and $B$ commute and intersect in the trivial group, we must have $a = u  \in S_{\ell}$ which entails that $x = s_{1}a \in S_{\ell}$. This shows $S_{\ell}A \cap S_{\ell}B = S_{\ell}$ as desired. 
\end{proof}

\begin{cor} \label{cor: nice intersections}
    Suppose $\Gamma$ is a graph and $I$ is a set, and let $S_{pq}$ denote the product of the $p$- and $q$-Sylow subgroups of $\widetilde{G_{\Gamma} \times C_{2}^{I}}$. Then for any closed normal subgroups $M,N \unlhd \widetilde{G_{\Gamma}\times C_{2}^{I}}$, we have 
    $$
    S_{pq}M \cap S_{pq}N = S_{pq}(M \cap N). 
    $$
\end{cor}

\begin{proof}
    Let $G = \widetilde{G_{\Gamma} \times C_{2}^{I}}$. We know that, in $G$, if we let $S_{p}$ and $S_{q}$ denote the $p$- and $q$-Sylow subgroups of $G$, respectively, we have that $S_{p}$ and $S_{q}$ are normal and $S_{pq} = S_{p}S_{q}$. Then we may apply Fact \ref{fact: normal sylow} twice to obtain 
    $$
    S_{pq}(M \cap N) = S_{p}(S_{q}(M \cap N)) = S_{p}(S_{q}M \cap S_{q}N) = S_{p}S_{q}M \cap S_{p}S_{q} N,
    $$
    which gives the desired equality. 
\end{proof}

\section{An identification theorem} \label{sec: identification}

In this section, we are interested in describing arbitrary models $S$ of the theory $\mathrm{Th}(S(\widetilde{G_{\Gamma}\times C_{2}^{\mu}}))$, for a given infinite graph $\Gamma$ and a cardinal $\mu$.  In fact, we give a description of a broader class of structures which we call the \emph{graph closed} complete systems, showing that they are closely related to groups produced by the CDM graph coding. Our ultimate aim is to prove that every model of this theory is isomorphic to one of the form $S(\widetilde{G_{\Gamma'} \times C_{2}^{\mu'}})$ for some $\Gamma' \equiv \Gamma$ and some cardinal $\mu'$. Given an arbitrary complete system $S$, we introduce two distinguished subsystems $S_{-} \subseteq S_{+} \subseteq S$ and build up to the main identification theorem by proving identification lemmas for these subsystems.  Each identification lemma has two parts. First, we show that the subsystem is more or less isomorphic to an inverse system of a group produced by CDM graph coding.  Then, we show that given $S \preceq S' \models \mathrm{Th}(S(\widetilde{G_{\Gamma}\times C_{2}^{\mu}}))$, a representation for one of the distinguished subsystems of $S$ lifts to a compatible representation for the corresponding subsystem of $S'$.

\subsection{Vertex width and graph closure}

First, we fix notation for the theories that we are studying:

\begin{defn} \label{defn: distinguished subsystems}
    Suppose $\Gamma$ is an infinite graph and $\mu$ is a cardinal (possibly finite, including $0$).  
    \begin{enumerate}
        \item We define $T_{\Gamma, \mu}$ to be the $L_{\mathrm{IS}}$-theory of $S(\widetilde{G_{\Gamma} \times C_{2}^{\mu}})$. 
        \item Given $S \models T_{\mathrm{IS}}$, we define $S_{-}$ to be the subsystem of $S$ generated by the following:
        $$
        \{ \alpha \in S : [\alpha] \cong D_{p}\} \cup \{\alpha \in S : [\alpha] \cong W\}.
        $$
        \item Given $S \models T_{\mathrm{IS}}$, we define $S_{+}$ to be the subsystem of $S$ generated by
        $$
        \{\alpha \in S: [\alpha] \cong C_{2}\} \cup \{ \alpha \in S : [\alpha] \cong D_{p}\} \cup \{\alpha \in S : [\alpha] \cong W\},
        $$
        or, equivalently, the subsystem of $S$ generated by $S_{-}$ and $\{\alpha \in S : [\alpha] \cong C_{2} \}$. 
    \end{enumerate}
\end{defn}

\begin{defn}
    Suppose $\Gamma$ is an infinite graph, $\mu$ is a cardinal, and $\alpha \in S \models T_{\Gamma,\mu}$ is an element with $[\alpha] \cong C_{2}$.  We say $\alpha$ has \emph{vertex width} $n$ if $n$ is the least integer $\geq 1$ such that there are (pairwise $\sim$-inequivalent) $\beta_{0}, \ldots, \beta_{n-1} \in S$ with $[\beta_{i}] \cong D_{p}$ and $\alpha \geq \gamma$, where $\gamma = \bigwedge_{i < n} \beta_{i}$.  We say that the $\beta_{i}$ (or the classes $[\beta_{i}]$) \emph{witness} that $\alpha$ has vertex width $n$. If there are no such $\beta_{i}$, we say that $\alpha$ has vertex width $\infty$.  
\end{defn}

Note that, in the above definition, the $\beta_{i}$ correspond to a size $n$ set of vertices $V_{0}$, and $\bigwedge_{i < n} \beta_{i}$ corresponds to the $D_{p}^{V_{0}}$ quotient of $\widetilde{G_{\Gamma}}$. Saying that $\alpha$ with $[\alpha] \cong C_{2}$ has vertex width $n$, then, means that $\alpha$ corresponds to a $C_{2}$ quotient of $D_{p}^{V_{0}}$ for a set $V_{0}$ of vertices with $|V_{0}| = n$ which is not a quotient of any $D_{p}^{V'}$ for any $V'$ of size strictly smaller than $n$. 

Here is the suggested way to think about vertex width: the group $\widetilde{G_{\Gamma} \times C_{2}^{I}}$ has as a quotient the group $C_{2}^{V} \times C_{2}^{I}$ which is an elementary abelian $2$-group and thus an $\mathbb{F}_{2}$-vector space. We can write it in additive notation as $\mathbb{F}_{2}^{V} \times \mathbb{F}_{2}^{I}$. The open normal subgroups of index $2$ are in one-to-one correspondence with continuous linear functionals $(\mathbb{F}_{2}^{V} \times \mathbb{F}_{2}^{I}) \to \mathbb{F}_{2}$, thus are elements of the continuous dual of $\mathbb{F}_{2}^{V} \times \mathbb{F}_{2}^{I}$, which is a direct sum of the continuous duals of $\mathbb{F}_{2}^{V}$ and $\mathbb{F}_{2}^{I}$. In the context of the structure in which we are working, however, these two factors behave differently. The continuous dual of $\mathbb{F}_{2}^{V}$ comes with a distinguished basis, namely the coordinate projections for each vertex $v \in V$. We will see in Lemma \ref{lem: definability of vertex width} below that the elements of this distinguished basis are definable in $S(\widetilde{G_{\Gamma} \times C_{2}^{I}})$. The kernels of these maps correspond exactly to $\sim$-classes of elements $\alpha$ with $[\alpha] \cong C_{2}$ and which have vertex width $1$. Elements of vertex width $n$, then, correspond precisely to elements of the continuous dual of $\mathbb{F}_{2}^{V}$ which can be expressed as a linear combination of $n$ elements from this distinguished basis, but no fewer. Elements in the continuous dual of $\mathbb{F}_{2}^{V} \times \mathbb{F}_{2}^{I}$ which are not contained in the subspace $(\mathbb{F}_{2}^{V})^{*}$, then, are the elements of infinite vertex width. This dictionary is spelled out in the following lemma:

\begin{lem} \label{lem: dictionary}
    Let $\Gamma = (V,R)$ be a graph and $I$ a set, disjoint from $V$. Let $Q = \{[\alpha] : \alpha \in S(\widetilde{G_{\Gamma} \times C_{2}^{I}}), [\alpha] \cong C_{2}\}$.  Then there is a bijection $d: Q \to ((\mathbb{F}_{2}^{V} \times \mathbb{F}^{I}_{2})^{*} - \{0\})$ satisfying the following:
    \begin{enumerate}
        \item For $[\alpha] \in Q$, $\alpha$ has vertex width $n$ if and only if $d([\alpha])$ is a linear combination of $n$ elements of $(\mathbb{F}^{V}_{2})^{*}$ of the form $\pi_{v}$, where $\pi_{v}$ denotes the coordinate projection to the $v$ coordinate (and $n$ is minimal with this property). 
        \item If $(\gamma_{i})_{i < k}$ is a sequence of elements of $S(\widetilde{G_{\Gamma} \times C_{2}^{I}})$ with $[\gamma_{i}] \cong C_{2}$ for all $i$, then we have 
        $$
        \gamma_{i} \not\geq \bigwedge_{j < i} \gamma_{j}
        $$
        for all $i < k$ if and only if $\{d([\gamma_{i}]) : i < k\}$ are linearly independent in $(\mathbb{F}_{2}^{V} \times \mathbb{F}_{2}^{I})^{*}$. 
    \end{enumerate}
\end{lem}

\begin{proof}
    Let $\pi : \widetilde{G_{\Gamma} \times C_{2}^{I}} \to G_{\Gamma} \times C_{2}^{I}$ denote the universal Frattini cover, dual to the inclusion $S_{+} \subseteq S = S(\widetilde{G_{\Gamma} \times C_{2}^{I}})$.  We know that, since $\Phi(C_{2}) = 1$, every quotient of $\widetilde{G_{\Gamma} \times C_{2}^{I}}$ is a quotient of $G_{\Gamma} \times C_{2}^{I}$, that is, if an open normal subgroup of $\widetilde{G_{\Gamma} \times C_{2}^{I}}$ has index $2$, then it contains $\ker(\pi)$.  The quotient of $G_{\Gamma} \times C_{2}^{I}$ by the normal subgroup $C_{q}^{R} \times C_{p}^{V}$ is naturally isomorphic to $C_{2}^{V} \times C_{2}^{I}$, in a way that respects the coordinate projections. 
    
    Since $C_{2}$ has order $2$, any open normal subgroup of $\widetilde{G_{\Gamma} \times C_{2}^{I}}$ of index $2$ must contain $C_{q}^{R} \times C_{p}^{V}$. Thus, to establish the existence of the bijection $d$, it suffices to define a bijection from $\{[\alpha] : \alpha \in S(C_{2}^{V} \times C_{2}^{I}), [\alpha] \cong C_{2}\}$ to the continuous dual $(\mathbb{F}_{2}^{V} \times \mathbb{F}_{2}^{I})^{*}$.  We will switch to additive notation and write $C_{2}^{V} \times C_{2}^{I}$ as $\mathbb{F}_{2}^{V} \times \mathbb{F}_{2}^{I}$.  Then if $\alpha \in S(\mathbb{F}^{V}_{2} \times \mathbb{F}^{I}_{2})$ and $[\alpha] \cong C_{2}$, then $[\alpha] = [U]$ for some open subgroup $U \unlhd {\mathbb{F}_{2}}^{V} \times \mathbb{F}_{2}^{I}$ of index $2$. There is a unique non-zero continuous linear functional $f_{U} : \mathbb{F}_{2}^{V} \times \mathbb{F}_{2}^{I} \to \mathbb{F}_{2}$ with kernel $U$. Conversely, if $f \in (\mathbb{F}_{2}^{V} \times \mathbb{F}_{2}^{I})^{*}$ is non-zero, then $\ker(f)$ is an index $2$ subgroup of $(\mathbb{F}_{2}^{V} \times \mathbb{F}^{I}_{2})$ so $[\ker(f)] \in Q$. It follows that the map $[U] \mapsto f_{U}$ is the desired bijection. 

    Now we prove (1). Recall from Lemma \ref{lem: no unexpected quotients}(1) that each $D_{p}$ quotient of $G_{\Gamma} \times C_{2}^{I}$ is obtained by quotienting by the kernel of a coordinate projection to some coordinate $v \in V$. It follows that $\alpha \in S(\mathbb{F}^{V}_{2} \times \mathbb{F}^{I}_{2}) \subseteq S(G_{\Gamma} \times C_{2}^{I})$ has vertex width 1 if and only if $[\alpha] = [\ker(\pi_{v})]$ for the coordinate projection $\pi_{v} : \mathbb{F}^{V}_{2} \times \mathbb{F}^{I}_{2} \to \mathbb{F}_{2}$ for some $v \in V$. More generally, if $\alpha \in S(G_{\Gamma} \times C_{2}^{I})$ has vertex width $n$ and $\{\gamma_{i} : i < n\}$ is a minimal set of witnesses with $[\gamma_{i}] \cong D_{p}$ for all $i < n$, then for each $i < n$, we can choose some $\delta_{i} \geq \gamma_{i}$ with $[\delta_{i}] \cong C_{2}$. Then since $[\alpha]$ is a $2$-group and $\alpha \geq \bigwedge_{i < n} \gamma_{i}$, we must have $\alpha \geq \bigwedge_{i < n} \delta_{i}$. From this it follows that an arbitrary $\alpha$ has vertex width $n$ if and only if $n$ is minimal such that there is a set $\{\delta_{i} : i < n\}$ where $[\delta_{i}] \cong C_{2}$ and $\delta_{i}$ has vertex width $1$ for each $i$ and where $\alpha \geq \bigwedge_{i < n} \delta_{i}$. Then the conclusion follows from (2).

    So now we show (2).  Fix a sequence $(\gamma_{i})_{i < k}$ in $S(\widetilde{G_{\Gamma} \times C_{2}^{I}})$ with $[\gamma_{i}] \cong C_{2}$ and let $f_{i} = d([\gamma_{i}])$ for all $i < k$. We know from the above that, necessarily, $\gamma_{i} \in S(\mathbb{F}^{V}_{2} \times \mathbb{F}^{I}_{2})$ (viewed as a subsystem of $S(\widetilde{G_{\Gamma} \times C_{2}^{I}})$) so $[\gamma_{i}] = [U_{i}]$ for some index $2$ open subgroup $U_{i} \unlhd (\mathbb{F}^{V}_{2} \times \mathbb{F}^{I}_{2})$ and $f_{i}$ is the unique continuous linear functional with kernel $U_{i}$. 

    Suppose first that for some $i < k$, $f_{i} = \sum_{j < i} c_{j} f_{j}$.  Then set $X = \{j < i : c_{j} = 1\}$, so $f_{i} = \sum_{j \in X} f_{j}$.  This clearly gives $U_{i} = \ker(f_{i}) \supseteq \bigcap_{j \in X} U_{j}$ so we get 
    $$
    \gamma_{i} \geq \bigwedge_{j \in X} \gamma_{j} \geq \bigwedge_{j < i} \gamma_{j}.
    $$
    For the other direction, suppose now that $\gamma_{i} \geq \bigwedge_{j < i} \gamma_{j}$. Then choose $Y \subseteq \{0, \ldots, i-1\}$ minimal such that $\gamma_{i} \geq \bigwedge_{j \in Y} \gamma_{j}$. We claim that $f_{i} = \sum_{j \in Y} f_{j}$. By the minimality of $Y$, for each $\ell \in Y$, we have $U_{i} \not\supseteq \bigcap_{j \in Y - \{\ell\}} U_{j}$, so we can pick some $v_{\ell} \in \bigcap_{j \in Y - \{\ell\}} U_{j}$ with $f_{i}(v_{\ell}) = 1$. It follows, also, that $f_{\ell}(v_{\ell}) = 1$ since otherwise $v_{\ell} \in \bigcap_{j \in Y} U_{j} \subseteq U_{i}$. The images of the vectors $(v_{\ell})_{\ell \in Y}$ form a basis for $(\mathbb{F}_{2}^{V} \times \mathbb{F}_{2}^{I})/\bigcap_{\ell \in Y} U_{\ell}$, so we obtain $f_{i} = \sum_{j \in Y} f_{j}$, as desired. 
\end{proof}

\begin{lem} \label{lem: definability of vertex width}
    For all $n$, there is some $L_{\mathrm{IS}}$-formula $\varphi_{n}(x)$ that defines, for any $\Gamma, \mu$ and $S \models T_{\Gamma, \mu}$, the $\alpha$ with $[\alpha] \cong C_{2}$ of vertex width $n$. 
\end{lem}

\begin{proof}
    For each $n$, there is a formula $\psi_{n}(x)$ that asserts that $x$ has vertex width at most $n$, defined as follows:
    $$
    x \in X_{2} \wedge x \not\in X_{1} \wedge (\exists v_{0}, \ldots, v_{n-1} \in X_{2p})\left(\bigwedge_{i < j < n} v_{i} \not\sim v_{j} \wedge \bigwedge_{i < n} [v_{i}] \cong D_{p} \wedge \left(x \geq \bigwedge_{i < n} v_{i}\right)\right),
    $$
    where $\bigwedge_{i < n} v_{i}$ in the rightmost conjunct denotes the meet of the $n$ elements $v_{0}, \ldots, v_{n-1}$. Then the desired formula $\varphi_{n}(x)$ can be defined by $\psi_{1}(x)$ when $n= 1$ and by $\psi_{n}(x) \wedge \neg \psi_{n-1}(x)$ for $n >1$.
\end{proof}

\begin{lem} \label{lem: witness algebraic}
    Fix a graph $\Gamma = (V,R)$ and a cardinal $\mu$. Suppose $\alpha \in S \models T_{\Gamma,\mu}$, $[\alpha]_{\sim} \cong C_{2}$, and $\alpha$ has vertex width $n$. Suppose further that $\beta_{i}, \beta'_{i} \in S$ satisfy $[\beta_{i}] \cong [\beta'_{i}] \cong D_{p}$ for all $i < n$. Then if $\gamma = \bigwedge_{i < n} \beta_{i}$,  $\gamma' = \bigwedge_{i < n} \beta'_{i}$, and $\alpha \geq \gamma$ and $\alpha \geq \gamma'$, then $\{[\beta_{i}]_{\sim} : i < n\} = \{[\beta'_{i}]_{\sim} : i < n\}$.  In other words, if $\alpha$ has vertex width $n$, then the set of witnesses is well-defined. 
\end{lem}

\begin{proof}
    As the statement is elementary (for a fixed $n$), it suffices to prove this when $S = S(\reallywidetilde{G_{\Gamma} \times C_{2}^{I}})$ for some $I$ with $|I| = \mu$. This follows from Lemma \ref{lem: dictionary}(1) since having vertex width $n$ is equivalent to corresponding to a linear combination of $n$ elements which themselves correspond to elements of vertex width $1$, under the bijection $d$ defined there. As the elements of vertex width $1$ are linearly independent, they are completely determined. We can also give a direct argument: by Corollary \ref{cor: no unexpected frattini quotients}, each $\beta_{i}$ and $\beta'_{i}$ is an element of the subsystem $S(G_{\Gamma}) \subseteq S(\widetilde{G_{\Gamma} \times C_{2}^{\mu}})$. We know that, after possibly replacing each $\beta_{i}$ with something in the same $\sim$-class, we may assume $\beta_{i} = \ker(\pi_{v_{i}})$ and, likewise, $\beta'_{i} = \ker(\pi_{v'_{i}})$ for sets of vertices $U = \{v_{i} : i < n\}$ and $U' = \{v'_{i} : i < n\}$ from $\Gamma$. Then we have $\bigwedge_{i < n} \beta_{i} = \bigcap_{v \in U} \ker(\pi_{v})$ and $\bigwedge_{i < n} \beta'_{i} =  \bigcap_{v \in U'} \ker(\pi_{v})$. It follows, then, that
    $$
    \gamma \vee \gamma' =\left(\bigcap_{v \in U} \ker(\pi_{v})\right)\left(\bigcap_{v \in U'} \ker(\pi_{v})\right)  = \bigcap_{v \in (U \cap U')} \ker(\pi_{v}).
    $$

    As $\alpha \geq \gamma$ and $\alpha \geq \gamma'$, we must, therefore, have $\alpha \geq \bigcap_{v \in (U \cap U')} \ker(\pi_{v})$.  By the minimality of $n$, we must have $U = U \cap U' = U'$. This shows $\{[\beta_{i}] : i < n\} = \{[\beta'_{i}] : i < n\}$. 
\end{proof}

Recall that if $\Gamma = (V,R)$ is a graph, then $G_{\Gamma}$ can be described as a semi-direct product $C_{q}^{R} \rtimes D_{p}^{V}$ with the action described in Section \ref{sec: CDM}. Then since $D_{p} = C_{p} \rtimes C_{2}$ and the action of $D_{p} \times D_{p}$ factors through the quotient $C_{2} \times C_{2}$ in the definition of $W$, $G_{\Gamma}$ may be further described as 
$$
(C_{q}^{R} \times C_{p}^{V}) \rtimes C_{2}^{V},
$$
where $C_{2}^{V}$ acts on $C_{q}^{R}$ in the induced action and $C_{2}^{V}$ acts on $C_{p}^{V}$ coordinate-wise. This gives a short exact sequence 
$$
1 \to P \to G_{\Gamma} \to C_{2}^{V} \to 1
$$
where $P = C_{q}^{R} \times C_{p}^{V}$ is the product of the $p$- and $q$-Sylow subgroups of $G_{\Gamma}$. We will let $\xi_{\Gamma} : G_{\Gamma} \to C_{2}^{V}$ denote the quotient map (which, in coordinates, is just the projection map $((C_{q}^{R} \times C_{p}^{V}) \rtimes C_{2}^{V}) \to C_{2}^{V}$). 

\begin{lem} \label{lem: parity}
    Suppose $\Gamma$ is an infinite graph and $\alpha \in S(G_{\Gamma})$. Then $\alpha$ has vertex width $n$ if and only if there exists some induced subgraph $\Gamma_{0} = (V_{0},R_{0}) \subseteq \Gamma$ with $|V_{0}| = n$ and $[\alpha] = [N]$ for some open normal subgroup $N \unlhd G_{\Gamma_{0}}$ with $N \supseteq P$, where $P$ denotes the product of the $p$- and $q$-Sylow subgroups of $G_{\Gamma_{0}}$, and such that $\xi_{\Gamma_{0}}(N) \unlhd C_{2}^{V_{0}}$ is the normal subgroup consisting of those $(g_{v})_{v \in V_{0}} \in C_{2}^{V_{0}}$ such that the number of $v$ with $g_{v} = - 1$ is even. 
\end{lem}

\begin{proof}
    First, suppose $\alpha$ has vertex width $n$, and let $\beta_{0}, \ldots, \beta_{n-1}$ list the witnesses, so we have $[\beta_{i}] \cong D_{p}$ and $[\beta_{i}] \neq [\beta_{j}]$ for all $i \neq j$. By Lemma \ref{lem: no unexpected quotients}, we have $[\beta_{i}] = [\ker(\pi_{v_{i}})]$ for some $v_{i} \in V$, for all $i < n$. Let $V_{0} = \{v_{i} : i < n\}$ and let $\Gamma_{0} = (V_{0},R_{0}) \subseteq \Gamma$ be the induced subgraph with vertices $V_{0}$. Then since $\alpha \geq \bigwedge_{i < n} \beta_{i}$, we may identify $[\alpha] = [N]$ for some normal subgroup $N \unlhd G_{\Gamma_{0}}$. Note that we must have $N \supseteq P$, the product of the $p$- and $q$-Sylows, since $|[\alpha]| = 2$. So, setting $\overline{N} = \xi_{\Gamma_{0}}(N)$, we see $\xi_{\Gamma_{0}}$ induces an isomorphism $G_{\Gamma_{0}}/N \to C_{2}^{V_{0}}/\overline{N}$. Moreover, we see that $(-1)_{v} \not\in \overline{N}$ for all $v \in V_{0}$ since, if $(-1)_{v} \in \overline{N}$, then $\xi_{\Gamma_{0}}$ would factor through the projection $C_{2}^{V_{0}} \to C_{2}^{V_{0} - \{v\}}$, making $[\alpha]$ a quotient of $G_{\Gamma_{0} - \{v\}}$, contradicting minimality of $V_{0}$. Then, since $C_{2}^{V_{0}}/\overline{N} \cong C_{2}$, we have that, for all $v \neq v'$, the cosets $(-1)_{v}\overline{N}$ and $(-1)_{v'}\overline{N}$ are equal and their product is $\overline{N}$. It follows that $\overline{N}$ consists precisely of those $(g_{v})_{v \in V_{0}} \in C_{2}^{V_{0}}$ such that the number of $v$ with $g_{v} = -1$ is even.

    For the other direction, assume we are given an induced subgraph $\Gamma_{0} = (V_{0},R_{0}) \subseteq \Gamma$ with $n$ vertices such that $[\alpha] = [N] \cong C_{2}$ and $\xi_{\Gamma_{0}}(N) := \overline{N} \unlhd C_{2}^{V_{0}}$ is the normal subgroup consisting of the $(g_{v})_{v \in V_{0}}$ with the number of $v$ with $g_{v} = -1$ even. Then, for any proper induced subgraph $\Gamma' \subsetneq \Gamma_{0}$, there is some $v \in V_{0}$ which is not in $\Gamma'$ so $(-1)_{v} \in \xi_{\Gamma_{0}}(N \ker(\pi_{\Gamma_{0}-\{v\}}))$ so $\xi_{\Gamma_{0}}(N \ker(\pi_{\Gamma_{0}-\{v\}})) = C_{2}^{V_{0}}$. This shows the minimality of $\Gamma_{0}$, so $\{\ker(\pi_{v}) : v \in V_{0}\}$ is the set of witnesses establishing that $\alpha$ has vertex width $n$. 
\end{proof}

\begin{defn}
Fix a graph $\Gamma$ and set $I$. Suppose $A \subseteq S(\widetilde{G_{\Gamma} \times C_{2}^{I}})$ is a small subset.  We define the \emph{graph closure of }$A$, denoted $\mathrm{gcl}(A)$, to be the smallest subsystem containing $A$ satisfying the following conditions:
    \begin{enumerate}
        \item If $\alpha \in \mathrm{gcl}(A)$ satisfies $[\alpha]_{\sim} \cong C_{2}$ and $\alpha$ has finite vertex width, then the set of witnesses is contained in $\mathrm{gcl}(A)$ as well.
        \item If $[\beta_{0}]_{\sim}$ and $[\beta_{1}]_{\sim}$ are distinct $\sim$-classes with $\beta_{i} \in \mathrm{gcl}(A)$ and $[\beta_{i}]_{\sim} \cong D_{p}$ for $i = 0,1$, then if $\gamma \leq \beta_{0}$, $\gamma \leq \beta_{1}$ and $[\gamma]_{\sim} \cong W$, then $\gamma \in \mathrm{gcl}(A)$.
    \end{enumerate} 
     If $A = \mathrm{gcl}(A)$, we say that $A$ is \emph{graph closed}.
\end{defn}

\begin{lem} \label{lem: C2 generation}
    Fix a graph $\Gamma$ and suppose $\delta \in S(G_{\Gamma})$ with $1 \neq \delta$. Then we have 
    $$
    \delta \in \mathrm{gcl}\left( \{ \alpha : \alpha \geq \delta \text{ and } [\alpha] \cong C_{2}     \}\right).
    $$
\end{lem}

\begin{proof}
    Fix $1 \neq \delta \in S(G_{\Gamma})$ (where $1$ denotes the trivial quotient class) and let $D = \mathrm{gcl}\left( \{ \alpha : \alpha \geq \delta \text{ and } [\alpha] \cong C_{2}     \}\right)$. Since $S(G_{\Gamma})$ is the union of $S(G_{\Gamma_{0}})$ as $\Gamma_{0}$ ranges over finite induced subgraphs of $\Gamma$, we can choose some minimal finite $\Gamma_{0} = (V_{0},R_{0})$ with $\delta \in S(G_{\Gamma_{0}})$. Then we have $[\delta] = G_{\Gamma_{0}}/N$ for some proper normal subgroup $N \unlhd G_{\Gamma_{0}}$. By Lemma \ref{lem: still proper}, if $P \unlhd G_{\Gamma_{0}}$ denotes the product of the $p$- and $q$-Sylow subgroups of $G_{\Gamma_{0}}$, then $G_{\Gamma_{0}}/NP$ is also a nontrivial quotient of $G_{\Gamma_{0}}$.  Note that if for some $v \in V_{0}$, $(\beta)_{v} \in NP$ (where, recall, $(\beta)_{v}$ denotes the involution of $G_{\Gamma_{0}}$ which is the involution in the $v$ coordinate and $1$ in all other vertex coordinates), then $\ker(\pi_{\Gamma_{0}-\{v\}}) \subseteq NP$ so the natural projection
    $$
    G_{\Gamma_{0}}/NP \to G_{\Gamma_{0}}/NP\ker(\pi_{\Gamma_{0} - \{v\}})
    $$
    is an isomorphism, which entails $\delta \in S(G_{\Gamma_{0} - \{v\}})$, contradicting the minimality of $\Gamma_{0}$. 

    Write $G_{\Gamma_{0}}/NP \cong C_{2}^{V_{0}}/\overline{N}$ via the isomorphism induced by the quotient map $\xi_{\Gamma_{0}} : G_{\Gamma_{0}} \to C_{2}^{V_{0}}$ defined before Lemma \ref{lem: parity}, where we are writing $\overline{N}$ for $\xi_{\Gamma_{0}}(NP)$. Then for each $v \in V_{0}$, we know $(-1)_{v} \not\in \overline{N}$, so we may choose some normal $N_{v} \unlhd C_{2}^{V_{0}}$ with $\overline{N} \subseteq N_{v}$, $[C_{2}^{V_{0}} : N_{v}] = 2$, and $(-1)_{v} \not\in N_{v}$. Then we set $X_{v} = \{w \in V_{0} : (-1)_{w} \not\in N_{v}\}$. Then the natural projection map $C_{2}^{V_{0}} \to C_{2}^{X_{v}}$ induces an isomorphism $C_{2}^{V_{0}}/N_{v} \to C_{2}^{X_{v}}/H$, where $H$ is the normal subgroup consisting of those $(g_{w})_{w \in X_{v}}$ such that the number of $w$ with $g_{w} = -1$ is even. Then $C_{2}^{V_{0}}/N_{v}$ corresponds to some $\alpha$ with $[\alpha] \cong C_{2}$ and $\alpha \geq \delta$.  Moreover, by construction, $\alpha$ has vertex width $|X_{v}|$ with witnesses $\{\ker(\pi_{w}) : w \in X_{v}\}$, which contains $\ker(\pi_{v})$.  This shows $\ker(\pi_{v}) \in D$ for all $v \in V_{0}$.  By definition of graph closure, any $\gamma \in S(G_{\Gamma})$ with $[\gamma] \cong W$ corresponding to an edge between $v,v' \in V_{0}$, then, must also be in $D$. We therefore obtain $\delta \in S(G_{\Gamma_{0}}) \subseteq D$, as desired. 
\end{proof}

Suppose $\gamma_{0}, \ldots, \gamma_{k-1} \in S \models T_{\Gamma,\mu}$ satisfy $[\gamma_{i}] \cong C_{2}$ for all $i < k$. Say $\gamma_{0}, \ldots, \gamma_{k-1}$ are \emph{independent} if $\gamma_{i} \not\geq \bigwedge_{j < i} \gamma_{j}$ for all $i < k$ (by Lemma~\ref{lem: dictionary}(2), this is equivalent to $\{d([\gamma_{i}])\}_{i}$ being linearly independent, and is therefore independent of the ordering).

\begin{lem} \label{lem: exchange} 
    Suppose $\gamma_{0}, \ldots, \gamma_{k-1},\alpha \in S \models T_{\Gamma,\mu}$ with $[\alpha] \cong [\gamma_{i}] \cong C_{2}$ for all $i < k$. If $\gamma_{0},\ldots, \gamma_{k-1}$ are independent and $\alpha \geq \bigwedge_{i < k} \gamma_{i}$, then we have
    $$
    \gamma_{i} \geq \alpha \wedge \bigwedge_{j \neq i} \gamma_{j}
    $$
    for all $i$. 
\end{lem}

\begin{proof}
   In $S(\reallywidetilde{G_{\Gamma} \times C_{2}^{\mu}})$, this follows from Lemma \ref{lem: dictionary}(2), by the usual Steinitz exchange for vector spaces. For $S$, it follows by elementarity.  
\end{proof}

\begin{lem} \label{lem: graph closure is closure}
    Fix a graph $\Gamma$ and cardinal $\mu$ and suppose $A \subseteq S \models T_{\Gamma,\mu}$ is a small subset. 
    \begin{enumerate}
        \item $\mathrm{gcl}(\mathrm{gcl}(A)) = \mathrm{gcl}(A)$. 
        \item $\Gamma(\mathrm{gcl}(A))$ is an induced subgraph of $\Gamma$. 
        \item $\mathrm{gcl}(A) \subseteq \mathrm{acl}(A)$. 
    \end{enumerate}
\end{lem}

\begin{proof}
    (1) It is immediate from the definitions that $\mathrm{gcl}(A)$ is itself graph-closed.  

    (2) follows directly from the second graph closure condition: if $\beta_{0}, \beta_{1} \in \mathrm{gcl}(A)$ with $[\beta_{i}] \cong D_{p}$ and $[\beta_{0}] \neq [\beta_{1}]$, then any $\gamma \in S$ with $[\gamma] \cong W$ and $\gamma \leq \beta_{0} \wedge \beta_{1}$ lies in $\mathrm{gcl}(A)$.  This is precisely the condition that $\Gamma(\mathrm{gcl}(A))$ is an induced subgraph of $\Gamma$.

    (3) We build $\mathrm{gcl}(A)$ in $\omega$ stages.  Set $A_{0} = \langle A \rangle$ and, for each $n \in \omega$, define $A_{n+1}$ to be the subsystem generated by $A_{n}$ together with:
    \begin{itemize}
        \item all $\beta \in S$ with $[\beta] \cong D_{p}$ that witness the finite vertex width of some $\alpha \in A_{n}$ with $[\alpha] \cong C_{2}$ of finite vertex width; and

        \item all $\gamma \in S$ with $[\gamma] \cong W$ and $\gamma \leq \beta_{0} \wedge \beta_{1}$ for some $\beta_{0}, \beta_{1} \in A_{n}$ with $[\beta_{i}] \cong D_{p}$ and $[\beta_{0}]_{\sim} \neq [\beta_{1}]_{\sim}$.
    \end{itemize}
    Since any $C_{2}$-element of $\bigcup_{n} A_{n}$ appears at some finite stage, $\mathrm{gcl}(A) = \bigcup_{n < \omega} A_{n}$.  We show $A_{n} \subseteq \mathrm{acl}(A)$ for all $n \in \omega$ by induction; the base case is clear.

    Assume $A_{n} \subseteq \mathrm{acl}(A)$.  For witnesses, suppose $\alpha \in A_{n}$ has $[\alpha] \cong C_{2}$ and vertex width $k$, and let $\beta$ be a witness for $\alpha$.  By Lemma \ref{lem: witness algebraic}, the set of witness $\sim$-classes $\{[\beta_{i}] : i < k\}$ is uniquely determined by $\alpha$.  It is easy to see there is a formula with parameter $\alpha$ defining the union of these $k$ witness $\sim$-classes, which is a finite set of $k \cdot |D_{p}|$ elements.  In particular, $\beta$ lies in this finite $\alpha$-definable set, so $\beta \in \mathrm{acl}(\alpha) \subseteq \mathrm{acl}(A_{n}) \subseteq \mathrm{acl}(A)$.

    For $W$-elements, suppose $\beta_{0}, \beta_{1} \in A_{n}$ have $[\beta_{i}] \cong D_{p}$ and $[\beta_{0}] \neq [\beta_{1}]$, and $\gamma \leq \beta_{0} \wedge \beta_{1}$ with $[\gamma] \cong W$.  By the definition of $W$, there is at most one $\sim$-class $[\gamma'] \cong W$ with any representative $\gamma'$ satisfying $\gamma' \leq \beta_{0} \wedge \beta_{1}$.  Hence $\gamma$ lies in a $\{\beta_{0},\beta_{1}\}$-definable set of size at most $|W|$, so $\gamma \in \mathrm{acl}(\beta_{0},\beta_{1}) \subseteq \mathrm{acl}(A_{n}) \subseteq \mathrm{acl}(A)$.
\end{proof}

\subsection{First identification}

Our first step will be to identify $S_{-}$ for an arbitrary graph closed $S \subseteq S' \models T_{\Gamma,\mu}$:

\begin{lem} \label{lem: firststep}
    Suppose $\Gamma$ is an infinite graph and $\mu$ is a cardinal. 
    \begin{enumerate}
        \item Suppose $S_{*} \models T_{\Gamma,\mu}$.  Then there is an isomorphism  $\Psi :(S_{*})_{-} \to S(G_{\Gamma_{*}})$ where $\Gamma_{*} = (V_{*}, R_{*}) \models \mathrm{Th}(\Gamma)$ is the graph interpreted in $S_{*}$, so 
          \begin{eqnarray*}
        V_{*} &=& \{ [\alpha] : \alpha \in S_{*}, [\alpha] \cong D_{p}\} \\
        R_{*} &=& \{([\alpha],[\alpha']) \in (V_{*})^{2} : [\alpha] \neq [\alpha'], (\exists \beta \in S_{*})([\beta] \cong W \text{ and } \beta \leq \alpha \wedge \alpha')\}.
    \end{eqnarray*}
        The isomorphism can be chosen so that if $v = [\alpha] \in V_{*}$, then $\Psi([\alpha]) = [\ker(\pi_{v})]$ and if $r = ([\alpha],[\alpha']) \in R_{*}$ and $\beta \leq \alpha \wedge \alpha'$ satisfies $[\beta] \cong W$, then $\Psi([\beta]) = [\ker(\pi_{r})]$. 
        \item If $S = \mathrm{gcl}(S) \subseteq S_{*} \models T_{\Gamma,\mu}$ and $\Psi : (S_{*})_{-} \to S(G_{\Gamma_{*}})$ is an isomorphism as in (1), then $\Psi|_{S_{-}}$ yields an isomorphism $S_{-} \to S(G_{\Gamma(S)})$, the diagram   
        $$
\xymatrix{ 
S_{-} \ar@{->}[r]^{\Psi|_{S_{-}}} \ar@{->}[d] & S(G_{\Gamma(S)}) \ar@{->}[d]^{S(\pi_{\Gamma(S)})} \\
(S_{*})_{-} \ar@{->}[r]^{\Psi} & S(G_{\Gamma_{*}}).
}
$$
commutes, where the map $S_{-} \to (S_{*})_{-}$ is the inclusion and $S(\pi_{\Gamma(S)})$ is dual to the projection $\pi_{\Gamma(S)} : G_{\Gamma_{*}} \to G_{\Gamma(S)}$. 
\item If $S = \mathrm{gcl}(S) \subseteq S_{*} \models T_{\Gamma,\mu}$ and $\Psi : S_{-} \to S(G_{\Gamma(S)})$ is an isomorphism as in (2), then $\Psi$ extends to an isomorphism $\Psi_{*}: (S_{*})_{-} \to S(G_{\Gamma_{*}})$ such that the diagram   
        $$
\xymatrix{ 
S_{-} \ar@{->}[r]^{\Psi} \ar@{->}[d] & S(G_{\Gamma(S)}) \ar@{->}[d]^{S(\pi_{\Gamma(S)})} \\
(S_{*})_{-} \ar@{->}[r]^{\Psi_{*}} & S(G_{\Gamma_{*}}).
}
$$
commutes, where the map $S_{-} \to (S_{*})_{-}$ is the inclusion and $S(\pi_{\Gamma(S)})$ is dual to the projection $\pi_{\Gamma(S)} : G_{\Gamma_{*}} \to G_{\Gamma(S)}$. 
    \end{enumerate}
\end{lem}

\begin{proof}
(1)  First, suppose $\Gamma_{0} = (V_{0},R_{0})$ is a finite labeled graph with $V_{0} = \{v_{0}, \ldots, v_{n-1}\}$, then $S(\widetilde{G_{\Gamma} \times C_{2}^{I}})$ satisfies the first-order sentence $\varphi_{\Gamma_{0}}$ which asserts that if $[\alpha_{0}], \ldots, [\alpha_{n-1}] \cong D_{p}$ and, for $i < j$, there is some $\beta \leq \alpha_{i} \wedge \alpha_{j}$ with $[\beta] \cong W$ if and only if $(v_{i},v_{j}) \in R_{0}$, then the subsystem generated by 
$$
\langle [\alpha_{i}] : i < n \rangle \cup \langle [\beta_{i,j}] : [\beta_{i,j}] \cong W, (v_{i},v_{j}) \in R_{0} \rangle
$$
is isomorphic to $S(G_{\Gamma_{0}})$ via an isomorphism that maps $[\alpha_{i}] \mapsto [\ker(\pi_{v_{i}})]$ and, for $r = (v_{i},v_{j}) \in R_{0}$, maps $[\beta_{i,j}] \mapsto [\ker(\pi_{r})]$. Such a sentence exists since $S(G_{\Gamma_{0}})$ is a finite structure. Since $S_{*} \models T_{\Gamma,\mu}$, $\varphi_{\Gamma_{0}}$ is true in $S_{*}$ for each finite graph $\Gamma_{0}$. 

Now we construct the isomorphism $\Psi$, as in the statement. For each $v = [\alpha] \in V_{*}$, choose an isomorphism $\sigma_{v} : [\alpha] \to [\ker(\pi_{v})]$, where $\pi_{v}$ refers to the projection $\pi_{v}: G_{\Gamma_{*}} \to D_{p}$ to the $v$ coordinate.  For each $r = (v,v') = ([\alpha],[\alpha']) \in R_{*}$ and $\beta \leq \alpha \wedge \alpha'$ with $[\beta] \cong W$, choose some isomorphism $\sigma_{r} : [\beta] \to [\ker(\pi_{r})]$ (where $\pi_{r}: G_{\Gamma_{*}} \to W$ is the projection to the $r$ coordinate) such that the following diagram commutes:
$$
\xymatrix@C=6em{
[\beta] \ar@{->}[r]^{\sigma_{r}} \ar@{->}[d]^{C} & [\ker(\pi_{r})] \ar@{->}[d]^{C} \\
[\alpha]\times[\alpha'] \ar@{->}[r]^{(\sigma_{v},\sigma_{v'})} & [\ker(\pi_{v})]\times [\ker(\pi_{v'})].
}
$$
We will show that these choices determine the desired isomorphism $\Psi: (S_{*})_{-} \to S(G_{\Gamma_{*}})$.

Let $X$ be a finite subset of $ \{ \alpha \in S_{*} : [\alpha] \cong D_{p}\} \cup \{\alpha \in S_{*} : [\alpha] \cong W\}$ such that if $\alpha \not\sim \alpha' \in X$ and there is some $\beta \in S_{*}$ with $[\beta]\cong W$ and $\beta \leq \alpha \wedge \alpha'$, then $\beta \in X$. Let $\langle X \rangle$ be the subsystem of $S_{*}$ generated by $X$ and let $\gamma \in \langle X \rangle$ be a $\leq$-minimal element (which exist because $X$ is finite). Let $\{v_{0}, \ldots, v_{n-1}\} = \{[\alpha_{0}], \ldots, [\alpha_{n-1}]\}$ list the vertices of $\Gamma_{X}$. Then because $S_{*} \models \varphi_{\Gamma_{X}}$, defined above, there is an isomorphism $\chi_{X} : \langle X \rangle \to S(G_{\Gamma_{X}})$ with $\chi_{X}([\alpha_{i}]) = [\ker(\pi_{v_{i}})]$. For each $i < n$, define $\tau_{v_{i}} : [\ker(\pi_{v_{i}})] \to [\ker(\pi_{v_{i}})]$ to be the automorphism $\sigma_{v_{i}} \circ (\chi_{X}|_{[\alpha]})^{-1}$.  Likewise for each edge $r = (v_{i},v_{j})$ of $\Gamma_{X}$, define $\tau_{r} : [\ker(\pi_{r})] \to [\ker(\pi_{r})]$ to be $\sigma_{r} \circ (\chi_{X}|_{[\beta]})^{-1}$ for the unique class $[\beta]$ with $[\beta] \cong W$ and $\beta \leq \alpha_{i} \wedge \alpha_{j}$. This choice of $\tau_{r}$ is compatible with $\tau_{v_{i}}$ and $\tau_{v_{j}}$, as in Corollary \ref{cor: unique aut}. This corollary gives us, then, a unique $\tau_{X} \in \mathrm{Aut}(S(G_{\Gamma_{X}}))$ such that $\tau_{X}$ restricts to give $\tau_{v_{i}}$ or $\tau_{r}$ on the $D_{p}$ and $W$ classes corresponding to $v_{i}$ and $r$, respectively. Define $\Psi_{X} = \tau_{X} \circ \chi_{X}$. Then we have $\Psi_{X}|_{[\alpha_{i}]} = \sigma_{v_{i}}$ and if $r = (v_{i},v_{j})$ and $\beta \leq \alpha_{i} \wedge \alpha_{j}$, then $\Psi_{X}|_{[\beta]} = \sigma_{r}$. 

If we choose a larger such $X'$ with $X \subseteq X'$, we get a commutative diagram:
    $$
\xymatrix{ 
\langle X \rangle \ar@{->}[r]^{\Psi_{X}} \ar@{->}[d] & S(G_{\Gamma_{X}}) \ar@{->}[d]^{S(\pi_{\Gamma_{X}})} \\
\langle X' \rangle \ar@{->}[r]^{\Psi_{X'}} & S(G_{\Gamma_{X'}}),
}
$$
since, by construction, $\Psi_{X}$ and $\Psi_{X'}$ agree on the generators of $\langle X \rangle$ (here the map $\langle X \rangle \to \langle X' \rangle$ is the inclusion). It follows that we obtain an isomorphism on the union $\Psi : (S_{*})_{-} \to S(G_{\Gamma_{*}})$, as desired. 

(2) follows from the proof of (1).  Restricting just to the finite $X$ contained in 
$$
\{ \alpha \in S : [\alpha] \cong D_{p}\} \cup \{\alpha \in S : [\alpha] \cong W\},
$$
the resulting map on the union gives an isomorphism $\Psi_{S_{-}}: (S)_{-} \to S(G_{\Gamma(S)})$. Since, by Lemma \ref{lem: graph closure is closure}(2), $\Gamma(S)$ is an induced subgraph of $\Gamma_{*}$, the image of $\Psi_{S_{-}}$ is just the subsystem of $S(G_{\Gamma_{*}})$ dual to the quotient via $\pi_{\Gamma(S)}$, yielding the commutative diagram in (2). 

(3) By (1) and (2), there is an isomorphism $\Xi: (S_{*})_{-} \to S(G_{\Gamma_{*}})$ which restricts to $S_{-}$ to yield an isomorphism $\Xi|_{S_{-}} : S_{-} \to S(G_{\Gamma(S)})$. Write $\Gamma(S) = (V_{S},R_{S})$.  By construction, for each vertex $v \in V(S)$, $\Psi \circ \Xi^{-1}$ restricts to give an automorphism $\sigma_{v}$ of $[\ker(\pi_{v})]$ and, likewise, to give an automorphism $\sigma_{r}$ of $[\ker(\pi_{r})]$ for each edge $r \in R_{S}$. Choose an automorphism $\sigma_{v} : [\ker(\pi_{v})] \to [\ker(\pi_{v})]$ arbitrarily for each $v \in V_{*} \setminus V_{S}$. Then, for each $r \in R_{*} \setminus R_{S}$, choose an automorphism $\sigma_{r} : [\ker(\pi_{r})] \to [\ker(\pi_{r})]$ satisfying the compatibility conditions of Corollary \ref{cor: unique aut}. That corollary then gives us $\Sigma \in \mathrm{Aut}(S(G_{\Gamma_{*}}))$ which restricts to $\sigma_{v}$ and $\sigma_{r}$ on each quotient $[\ker(\pi_{v})]$ and $[\ker(\pi_{r})]$ for all $v \in V_{S}$ and $r \in R_{S}$. Unwinding definitions, we get $(\Sigma \circ \Xi)|_{S_{-}} = \Psi$, so, setting $\Psi_{*} = \Sigma \circ \Xi$, we obtain the desired isomorphism. 
\end{proof}

\subsection{Second identification}

\begin{lem} \label{lem: transfer}
    Suppose $\Gamma$ is an infinite graph, $\mu$ is a cardinal, and $S_{0} = \mathrm{gcl}(S_{0}) \subseteq S_{1} = \mathrm{gcl}(S_{1}) \subseteq S_{*} \models T_{\Gamma,\mu}$.  
    \begin{enumerate}
        \item A sequence $(\alpha_{i})_{i < \lambda}$ of elements of $(S_{0})_{+}$ with $[\alpha_{i}] \cong C_{2}$ is a maximal such sequence satisfying $\alpha_{i} \not\in \langle (S_{0})_{-}, \alpha_{<i} \rangle$ for all $i < \lambda$ if and only if it is a maximal such sequence satisfying $\alpha_{i} \not\in \langle (S_{1})_{-}, \alpha_{<i} \rangle$. 
        \item We have the equality of subsystems $(S_{1})_{-} \vee (S_{0})_{+} = (S_{0})_{-}$.
    \end{enumerate}
\end{lem}

\begin{proof}
  (1)  Let $\Gamma_{i}$ be the graph interpreted in $S_{i}$.  By Lemma \ref{lem: firststep}, we may identify $(S_{i})_{-}$ with $S(G_{\Gamma_{i}})$ for $i = 0,1$. 

    First, assume $(\alpha_{i})_{i < \lambda}$ is maximal such that $\alpha_{i} \not\in \langle (S_{0})_{-}, \alpha_{<i} \rangle$ for all $i < \lambda$. If $\alpha_{i} \not\in \langle (S_{1})_{-}, \alpha_{<i} \rangle$ for all $i$, then $(\alpha_{i})_{i < \lambda}$ is a maximal such sequence, since this is an even stronger constraint.

    Fix $i < \lambda$ and suppose towards contradiction that $\alpha_{i} \in \langle (S_{1})_{-}, \alpha_{<i} \rangle$.  Since $[\alpha_{i}] \cong C_{2}$, we can find some $\beta_{0}, \ldots, \beta_{k-1} \in (S_{1})_{-}$ with $[\beta_{j}] \cong C_{2}$ and with each $\beta_{j}$ of vertex width $1$ for all $j$, as well as $\gamma_{0}, \ldots, \gamma_{\ell-1}$, a subsequence of $\alpha_{<i}$ with
    $$
    \alpha_{i} \geq \bigwedge_{j <k} \beta_{j} \wedge \bigwedge_{j < \ell} \gamma_{j},
    $$
    and we may suppose $k,\ell$ are minimal with this property (so, by assumption, $k \geq 1$). We may also suppose that $\alpha_{i}$, as well as the $\beta_{j}$ and $\gamma_{j}$ are the identity coset of the normal subgroup in their respective equivalence classes.

    \textbf{Claim:}  There is some $\delta \in S_{1}$ with $[\delta] \cong C_{2}$ and
    $$
    \delta \geq \left( \bigwedge_{j < k} \beta_{j} \right) \vee \left( \alpha_{i} \wedge \bigwedge_{j < \ell} \gamma_{j} \right).
    $$

    \emph{Proof of Claim}:  We have
    $$
    \left[\bigwedge_{j < k} \beta_{j} \wedge \bigwedge_{j < \ell} \gamma_{j} \right] \cong C_{2}^{k+\ell}.
    $$
    If $\epsilon \geq \bigwedge_{j < k} \beta_{j} \wedge \bigwedge_{j < \ell} \gamma_{j}$, then $\epsilon$ corresponds to a normal subgroup of $C_{2}^{k+\ell}$ so we will prove the claim by studying the normal subgroups of this group. For the purposes of this claim only, we will switch to additive notation and identify this group with $\mathbb{F}_{2}^{k+\ell}$.  Under this identification, we can identify each $\beta_{j}$ with $\ker(\pi_{j})$ where $\pi_{j}$ is the projection onto the $j$th coordinate and, similarly, each $\gamma_{j}$ can be identified with $\ker(\pi_{k+j})$.

    It follows, then, that $\alpha_{i}$ can be identified with the kernel of some linear form $f : \mathbb{F}^{k+\ell}_{2} \to \mathbb{F}_{2}$.  Our assumption that $k$ and $\ell$ are minimal entails that $f(e_{j}) = 1$ where $e_{j}$ denotes the $j$th standard basis vector, for $j < k+\ell$. It follows that $\alpha_{i}$ corresponds to the normal subgroup $\{(x_{0}, \ldots, x_{k+\ell - 1}) \in \mathbb{F}^{k+\ell}_{2} : \sum_{j} x_{j} = 0\}$. Then we have that $\alpha_{i} \wedge \bigwedge_{j < \ell} \gamma_{j}$ corresponds to the normal subgroup
    $$
    \{(x_{0}, \ldots, x_{k+\ell - 1}) \in \mathbb{F}^{k+\ell}_{2} : \sum_{j} x_{j} = 0 \text{ and } x_{k} = \ldots = x_{k+\ell - 1} = 0\}.
    $$
    Let $\delta$ be the element of $S$ corresponding to the normal subgroup
    $$
    \{(x_{0}, \ldots, x_{k+\ell-1}) \in \mathbb{F}^{k+\ell}_{2} : \sum_{j < k} x_{j} = 0\}.
    $$
    So we have $\delta \geq \alpha_{i} \wedge \bigwedge_{j < \ell} \gamma_{j}$. The kernel corresponding to $\delta$ contains the kernel corresponding to $\bigwedge_{j < k} \beta_{j}$, namely
    $$
    \{(x_{0}, \ldots, x_{k+\ell - 1}) \in \mathbb{F}^{k+\ell}_{2} : x_{0} = \ldots = x_{k-1} = 0\}.
    $$
    This shows $\delta \geq \bigwedge_{j < k} \beta_{j}$, which proves the claim. \qed

    Consider the $\delta$ of the claim. Then because $\delta \geq \alpha_{i} \wedge \bigwedge_{j < \ell} \gamma_{j}$, we have $\delta \in S_{0}$. Then since $\delta \geq \bigwedge_{j < k} \beta_{j}$ and each $\beta_{j}$ has vertex width $1$, we know that $\delta$ has finite vertex width. Since $S_{0} = \mathrm{gcl}(S_{0})$, we deduce that $\beta_{j} \in (S_{0})_{-}$ for all $j < k$. This entails $\alpha_{i} \in \langle (S_{0})_{-},\alpha_{<i}\rangle$, a contradiction.

    The other direction is similar.  Assume now that $(\alpha_{i})_{i < \lambda}$ is a maximal sequence of elements of $(S_{0})_{+}$ with $[\alpha_{i}] \cong C_{2}$ such that $\alpha_{i} \not\in \langle (S_{1})_{-}, \alpha_{<i} \rangle$ for all $i < \lambda$.  If $(\alpha_{i})_{i < \lambda}$ is not maximal such that $\alpha_{i} \not\in \langle (S_{0})_{-},\alpha_{<i} \rangle$ for all $i < \lambda$, then there must be some $\alpha_{\lambda} \in (S_{0})_{+}$ with $[\alpha_{\lambda}] \cong C_{2}$ and $\alpha_{\lambda} \not\in \langle (S_{0})_{-}, \alpha_{<\lambda} \rangle$.  By , it must be the case that $\alpha_{\lambda} \in \langle (S_{1})_{-}, \alpha_{<\lambda} \rangle$.  So there are some $\beta_{0}, \ldots, \beta_{k-1} \in (S_{1})_{-}$ with $[\beta_{j}] \cong C_{2}$ and with each $\beta_{j}$ of vertex width $1$ for all $j$, as well as $\gamma_{0}, \ldots, \gamma_{\ell-1}$, a subsequence of $\alpha_{<\lambda}$ with
    $$
    \alpha_{\lambda} \geq \bigwedge_{j <k} \beta_{j} \wedge \bigwedge_{j < \ell} \gamma_{j},
    $$
    and we may suppose $k,\ell$ are minimal with this property. Then the claim above produces a $\delta \in S_{0}$ with $[\delta] \cong C_{2}$ and satisfying both $\delta \geq \alpha_{\lambda} \wedge \bigwedge_{j < \ell} \gamma_{j}$ and $\delta \geq \bigwedge_{j < k} \beta_{j}$. So $S_{0} =\mathrm{gcl}(S_{0})$ entails $\{\beta_{j} : j < k\} \subseteq S_{0}$ so $\alpha_{\lambda} \in S_{0}$, a contradiction.

    (2)  By Lemma \ref{lem: firststep}, we may identify $(S_{1})_{-}$ with $S(G_{\Gamma_{1}})$ for some induced subgraph $\Gamma_{1} \subseteq \Gamma(S_{*}) \models \mathrm{Th}(\Gamma)$. The containment $(S_{0})_{-} \subseteq ((S_{0})_{+} \vee (S_{1})_{-})$ is clear, so we will prove the reverse containment. First, consider an $\alpha \in ( (S_{0})_{+} \vee (S_{1})_{-})$ with $[\alpha] \cong C_{2}$. The fact that $\alpha \in (S_{1})_{-}$ entails there are finitely many $\beta_{0}, \ldots, \beta_{k-1} \in (S_{1})_{-}$ with $[\beta_{i}] \cong D_{p}$ for $i < k$ and $\alpha \geq \bigwedge_{i < k} \beta_{i}$. So $\alpha$ has finite vertex width and thus, from $S_{0} = \mathrm{gcl}(S_{0})$, we have each $\beta_{i} \in S_{0}$ (and thus in $(S_{0})_{-}$) for all $i < k$. This shows $\alpha \in (S_{0})_{-}$, since $(S_{0})_{-}$ is upwards closed. 

    Next, by Lemma \ref{lem: C2 generation}, if $\delta \in ((S_{0})_{+} \vee (S_{1})_{-})$ is arbitrary, we have $\delta \in \mathrm{gcl}(\{\alpha : \alpha \geq \delta, [\alpha] \cong C_{2}\}) \subseteq \mathrm{gcl}((S_{0})_{-}) = (S_{0})_{-}$, since $S_{0} = \mathrm{gcl}(S_{0})$. 
\end{proof}

\begin{lem} \label{lem: secondstep}
Suppose $\Gamma$ is an infinite graph and $\mu$ is a cardinal.
\begin{enumerate}
    \item   Suppose $S = \mathrm{gcl}(S) \subseteq S_{*} \models T_{\Gamma, \mu}$. Then 
$$
S_{+} \cong S(G_{\Gamma(S)} \times C_{2}^{I}).
$$
    \item Suppose $S_{0} = \mathrm{gcl}(S_{0}) \subseteq S_{1} = \mathrm{gcl}(S_{1}) \subseteq S_{*} \models T_{\Gamma,\mu}$. Suppose further we are given an isomorphism $\Psi_{0} : (S_{0})_{+} \to S(G_{\Gamma_{0}} \times C_{2}^{I_{0}})$, where $\Gamma_{0}$ is the graph interpreted in $S_{0}$ and $I_{0}$ is some set.  Then there exist a set $I_{1} \supseteq I_{0}$ and an isomorphism $\Psi_{1} : (S_{1})_{+}\to S(G_{\Gamma_{1}} \times C_{2}^{I_{1}})$, where $\Gamma_{1} \succeq \Gamma_{0}$ is the graph interpreted in $S_{1}$, such that the diagram 
    $$
    \xymatrix{ 
(S_{0})_{+} \ar@{->}[r]^{\Psi_{0}} \ar@{->}[d] & S(G_{\Gamma_{0}} \times C_{2}^{I_{0}}) \ar@{->}[d]^{S(\pi)} \\
(S_{1})_{+} \ar@{->}[r]^{\Psi_{1}} & S(G_{\Gamma_{1}} \times C_{2}^{I_{1}}).
}
    $$
    commutes, where the map $(S_{0})_{+} \to (S_{1})_{+}$ is the inclusion and the map $\pi$ is given by the projection $\pi_{\Gamma_{0}}: G_{\Gamma_{1}} \to G_{\Gamma_{0}}$ and the projection $\pi_{I_{0}} : C_{2}^{I_{1}} \to C_{2}^{I_{0}}$.  
\end{enumerate}
\end{lem}

\begin{proof}
    (1) By Lemma \ref{lem: firststep}(1), we know $(S)_{-} \cong S(G_{\Gamma(S)})$. By Zorn's Lemma, we can pick a maximal sequence $(\alpha_{i})_{i < \lambda}$ (for an ordinal $\lambda$) in $S$ such that $[\alpha_{i}] \cong C_{2}$ and $\alpha_{i} \not\in \langle S_{-}, \alpha_{<i} \rangle$ for all $i < \lambda$ (so, necessarily, each $\alpha_{i} \in S_{+}$).  For each $i < \lambda$, let $S_{i} = \langle S_{-}, \alpha_{<i}\rangle$.  We will prove by induction on $i < \lambda$ that there is an isomorphism
$$
\varphi_{i}: G_{\Gamma(S)} \times C_{2}^{i} \to G(S_{i}).
$$
We will also assume that for $j < i$, these isomorphisms cohere in the sense that the diagram commutes
$$
\xymatrix{ 
G_{\Gamma(S)} \times C_{2}^{i} \ar@{->}[r]^{\varphi_{i}} \ar@{->}[d]^{\pi'_{i,j}} & G(S_{i})   \ar@{->}[d]^{\pi_{i,j}} \\
G_{\Gamma(S)} \times C_{2}^{j} \ar@{->}[r]^{\varphi_{j}} & G(S_{j}).
}
$$
where $\pi_{i,j}: G(S_{i}) \to G(S_{j})$ is the epimorphism dual to the inclusion $S_{j} \subseteq S_{i}$ and $\pi'_{i,j} : G_{\Gamma(S)} \times C_{2}^{i} \to G_{\Gamma(S)} \times C_{2}^{j}$ is the projection obtained by deleting all of the $i \setminus j$ coordinates.

For $i = 0$, the existence of an isomorphism $\varphi_{0}: G_{\Gamma(S)} \to G(S_{-})$ follows from Lemma \ref{lem: firststep}(1) and duality, as noted above. Note, for a given $i < \lambda$, since $\alpha_{i} \not\in S_{i}$ and $[\alpha_{i}] \cong C_{2}$ has only $1$ as a proper quotient, we have
$$
[\alpha_{i}] \vee S_{i} = 1
$$
(meaning the coarsest common quotient is trivial) and hence
$$
G(S_{i+1}) = G(S_{i} \wedge [\alpha_{i}]) = G(S_{i}) \times [\alpha_{i}] \cong G_{\Gamma(S)} \times C_{2}^{i+1},
$$
where the isomorphism between the latter two terms follows by induction. Finally, for limit $i$, we know that $G(S_{i})$ is the inverse limit of $S_{j}$ and $G_{\Gamma(S)} \times C_{2}^{i}$ is the inverse limit of $G_{\Gamma(S)} \times C_{2}^{j}$ for $j < i$ and these are isomorphic by induction and the coherence of the isomorphisms $\varphi_{j}$ for $j < i$. Maximality of the sequence $(\alpha_{i})_{i < \lambda}$ entails that $S_{<\lambda} = (S)_{+}$ and, taking $I = \lambda$, we have $G(S_{+}) \cong G_{\Gamma(S)} \times C_{2}^{I}$, as desired. 

(2) Let $\Gamma_{0}$ and $\Gamma_{1}$ be the graphs obtained by the graph interpretation in $S_{0}$ and $S_{1}$ respectively, so $\Gamma_{0}$ is an induced subgraph of $\Gamma_{1}$ which is an induced subgraph of $\Gamma(S_{*}) \models \mathrm{Th}(\Gamma)$. Replacing $I_{0}$ with $\lambda_{0} = |I_{0}|$, we may assume $I_{0} = \lambda_{0}$.  We assume we are given an isomorphism $\Psi_{0} : (S_{0})_{+} \to S(G_{\Gamma_{0}} \times C_{2}^{\lambda_{0}})$ and we will produce the isomorphism $\Psi_{1}$ as in the statement. 

Define $\alpha_{i} \in (S_{0})_{+}$ by $\alpha_{i} = \Psi_{0}^{-1}(\ker(\pi_{i}))$ where $\pi_{i} : G_{\Gamma_{0}} \times C_{2}^{\lambda_{0}} \to C_{2}$ is the projection to the $i$th coordinate, for each $i < \lambda_{0}$. Then we have $[\alpha_{i}] \cong C_{2}$, $\alpha_{i} \not\in \langle (S_{0})_{-}, \alpha_{<i} \rangle$ for all $i < \lambda_{0}$, and $\langle (S_{0})_{-}, \alpha_{<\lambda_{0}} \rangle = (S_{0})_{+}$. 

By Lemma \ref{lem: transfer}, $(\alpha_{i})_{i < \lambda_{0}}$ is a maximal sequence of elements of $(S_{0})_{+}$ with $[\alpha_{i}] \cong C_{2}$ and $\alpha_{i} \not\in \langle (S_{1})_{-}, \alpha_{<i} \rangle$ for all $i < \lambda_{0}$.  Extend this sequence to $(\alpha_{i})_{i < \lambda_{1}}$, a maximal sequence elements of $(S_{1})_{+}$ such that $[\alpha_{i}] \cong C_{2}$ and $\alpha_{i} \not\in \langle (S_{1})_{-}, \alpha_{< i} \rangle$ for all $i < \lambda_{1}$.   Then, as in (1), we have $(S_{1})_{+} = \langle (S_{1})_{-}, \alpha_{<\lambda_{1}} \rangle$ and thus, by (1), $(S_{1})_{+} \cong S(G_{\Gamma_{1}} \times C_{2}^{\lambda_{1}})$. 

Note that, by Lemma \ref{lem: transfer}(2), we have $(S_{1})_{-} \vee (S_{0})_{+} = (S_{0})_{-}$.  It follows, then, that 
$$
G((S_{1})_{-} \wedge (S_{0})_{+}) \cong G((S_{1})_{-}) \times_{G((S_{0})_{-})} G((S_{0})_{+}).
$$
Note that the given isomorphism $\psi_{0} : G_{\Gamma_{0}} \times C_{2}^{I_{0}} \to G((S_{0})_{+})$ dual to $\Psi_{0}$ induces an isomorphism $\overline{\psi_{0}}: G_{\Gamma_{0}} \to G((S_{0})_{-})$, dual to the restriction $\Psi_{0}|_{(S_{0})_{-}}$.  By Lemma \ref{lem: firststep}(3) and duality, there is some isomorphism $\chi : G_{\Gamma_{1}} \to G((S_{1})_{-})$ such that the diagram commutes:
$$
    \xymatrix{ 
G_{\Gamma_{1}}  \ar@{->}[r]^{\chi} \ar@{->}[d]^{\pi_{\Gamma_{0}}} &  G((S_{1})_{-}) \ar@{->}[d] \\
G_{\Gamma_{0}}   \ar@{->}[r]^{\overline{\psi_{0}}} & G((S_{0})_{-}).
}
    $$
It follows, then, that $\chi$ and $\psi_{0}$ define a map $(\chi,\psi_{0})$ from $G_{\Gamma_{1}} \times_{G_{\Gamma_{0}}} (G_{\Gamma_{0}} \times C_{2}^{\lambda_{0}}) \cong G_{\Gamma_{1}} \times C_{2}^{\lambda_{0}}$ to $G((S_{1})_{-}) \times_{G((S_{0})_{-})} G((S_{0})_{+})$. Repeating the induction from (1), then, shows that we may find some $\psi_{1} : G_{\Gamma_{1}} \times C_{2}^{\lambda_{1}} \to  G((S_{1})_{+})$ such that the diagram 
$$
    \xymatrix{ 
G_{\Gamma_{1}} \times C_{2}^{\lambda_{1}} \ar@{->}[r]^{\psi_{1}} \ar@{->}[d]^{\pi_{\lambda_{0}}} & G((S_{1})_{+}) \ar@{->}[d] \\
G_{\Gamma_{1}} \times C_{2}^{\lambda_{0}} \ar@{->}[r]^{(\chi,\psi_{0})} &   G((S_{1})_{-} \wedge (S_{0})_{+}).
}
    $$
    commutes, where $\pi_{\lambda_{0}}$ denotes the projection obtained by removing the $C_{2}$-coordinates indexed by $\lambda_{1} \setminus \lambda_{0}$. This gives the desired isomorphism.
\end{proof}

\subsection{Third identification}

\begin{lem} \label{lem: plus intersection}
    Let $\Gamma$ be an infinite graph and $\mu$ a cardinal.  Suppose $S_{0} = \mathrm{gcl}(S_{0}) \subseteq S_{1} = \mathrm{gcl}(S_{1}) \subseteq S_{*} \models T_{\Gamma,\mu}$. Then $S_{0} \vee (S_{1})_{+} = (S_{0})_{+}$.
\end{lem}

\begin{proof}
    The inclusion $(S_{0})_{+} \subseteq S_{0} \vee (S_{1})_{+}$ is clear, so we prove the reverse inclusion. 
    
    By Lemma \ref{lem: secondstep}, we may identify $(S_{1})_{+}$ with $S(G_{\Gamma_{1}} \times C_{2}^{I_{1}})$ and $(S_{0})_{+}$ with $S(G_{\Gamma_{0}} \times C_{2}^{I_{0}})$ in such a way that inclusion $(S_{0})_{+} \to (S_{1})_{+}$ is dual to the projection $\pi : G_{\Gamma_{1}} \times C_{2}^{I_{1}} \to G_{\Gamma_{0}} \times C_{2}^{I_{0}}$. If $\alpha \in S_{0} \vee (S_{1})_{+}$, then, we may identify the class $[\alpha]$ with a quotient $A :=(G_{\Gamma_{1}} \times C_{2}^{I_{1}})/N$ for some open normal subgroup $N \unlhd G_{\Gamma_{1}} \times C_{2}^{I_{1}}$. As in the proof of Lemma \ref{lem: bounding}, we define 
    $$
    V_{*} = \{v \in V_{1} : \pi_{v}(N) \lneq D_{p} \text{ or }v \text{ is incident to }r\text{ and }\pi_{r}(N) \cap C_{q} = 1\}.
    $$
    We first claim that $V_{*} \subseteq V_{0}$.  To see this note that if $\pi_{v}(N) \lneq D_{p}$, then $\pi_{v}$ induces an epimorphism from $(G_{\Gamma_{1}} \times C_{2}^{I_{1}})/N \to D_{p}/\pi_{v}(N)$. In the case that $\pi_{v}(N) = 1$, this quotient corresponds to a class $[\beta] \cong D_{p}$ in $S_{0}$, which entails $v \in V_{0}$. The other possibility is that $\pi_{v}(N) = C_{p}$, in which case the quotient $D_{p}/\pi_{v}(N)$ corresponds to some class $[\beta] \cong C_{2}$ with $\beta \in S_{0}$ of vertex width $1$ so $v \in V_{0}$ since $S_{0} = \mathrm{gcl}(S_{0})$. Finally, if $\pi_{r}(N) \cap C_{q} = 1$, then, by Lemma \ref{lem:easy}(4), the quotient $W/\pi_{r}(N)$ necessarily has a $C_{2}$ quotient of vertex width $2$, so both vertices of $r$ would be contained in $V_{0}$ since $S_{0} = \mathrm{gcl}(S_{0})$. 

    For every vertex $v \in V_{1} \setminus V_{0}$ and $r \in R_{1} \setminus R_{0}$, we have $\pi_{v}(N) = D_{p}$ and $\pi_{r}(N) \supseteq C_{q}$ so by Lemma \ref{lem:easy}, $(C_{p})_{v} \subseteq N$ and $(C_{q})_{r} \subseteq N$.  Thus, the quotient $(G_{\Gamma_{1}} \times C_{2}^{I_{1}})/N$ can be identified with the quotient 
    $$
    (G_{\Gamma_{0}} \times C_{2}^{V_{1} \setminus V_{0}} \times C_{2}^{I_{1}})/N'.
    $$
    Let $\pi_{1} : G_{\Gamma_{0}} \times C_{2}^{V_{1} \setminus V_{0}} \times C_{2}^{I_{1}} \to C_{2}^{V_{1} \setminus V_{0}} \times C_{2}^{I_{1} \setminus I_{0}}$ be the projection. So $\ker(\pi_{1}) = G_{\Gamma_{0}} \times C_{2}^{I_{0}}$. The quotient
    $$
    (G_{\Gamma_{0}} \times C_{2}^{V_{1} \setminus V_{0}} \times C_{2}^{I_{1} })/N' \ker(\pi_{1})
    $$
    corresponds to a quotient $(C_{2}^{V_{1} \setminus V_{0}} \times C_{2}^{I_{1} \setminus I_{0}})/N''$. If this were nontrivial, it would have a $C_{2}$ quotient not contained in $S(G_{\Gamma_{0}} \times C_{2}^{I_{0}}) = (S_{0})_{+}$ and, therefore, not in $S_{0}$ which is impossible. This shows that this quotient is trivial, which entails that the factor $C_{2}^{V_{1} \setminus V_{0}} \times C_{2}^{I_{1} \setminus I_{0}}$ is contained in $N'$.  Therefore, $A$ is a quotient of $G_{\Gamma_{0}} \times C_{2}^{I_{0}}$ as desired. This shows $\alpha \in (S_{0})_{+}$.

\end{proof}

\begin{thm} \label{thm: thirdstep}
Let $\Gamma$ be an infinite graph and $\mu$ a cardinal. 
\begin{enumerate}
\item Suppose $S_{0} = \mathrm{gcl}(S_{0}) \subseteq S_{1} \models T_{\Gamma,\mu}$. Then the epimorphism $G(S_{0}) \to G((S_{0})_{+})$ dual to the inclusion $(S_{0})_{+} \subseteq S_{0}$ is a Frattini cover (though not necessarily the universal Frattini cover). 
    \item If $S_{*} \models T_{\Gamma,\mu}$, then 
$$
S_{*} \cong S\left( \reallywidetilde{G_{\Gamma_{*}} \times C_{2}^{I}}\right)
$$
for some $\Gamma_{*} \equiv \Gamma$ and some set $I$ and the epimorphism $G(S_{*}) \to G((S_{*})_{+})$ dual to the inclusion $(S_{*})_{+} \subseteq S_{*}$ is the universal Frattini cover. 
\item Suppose $S_{0} \preceq S_{1} \models T_{\Gamma,\mu}$ and $\Gamma_{0} \preceq \Gamma_{1}$ are the graphs interpreted in $S_{0}$ and $S_{1}$ respectively. Suppose further we are given an isomorphism $\Psi_{0} : S_{0} \to S(\reallywidetilde{G_{\Gamma_{0}} \times C_{2}^{I_{0}}})$ for some set $I_{0}$.  Then there are a set $I_{1}$ and an isomorphism $\Psi_{1}: S_{1} \to S(\reallywidetilde{G_{\Gamma_{1}} \times C_{2}^{I_{1}}})$ such that the diagram
    $$
    \xymatrix{ 
S_{0} \ar@{->}[r]^{\Psi_{0}} \ar@{->}[d] & S(\reallywidetilde{G_{\Gamma_{0}} \times C_{2}^{I_{0}}}) \ar@{->}[d]^{S(\tilde{\pi})} \\
S_{1} \ar@{->}[r]^{\Psi_{1}} & S(\reallywidetilde{G_{\Gamma_{1}} \times C_{2}^{I_{1}}}).
}
    $$
    commutes, where the map $S_{0} \to S_{1}$ is the inclusion and the map $\tilde{\pi}$ is a lift of the projection $\pi: G_{\Gamma_{1}} \times C_{2}^{I_{1}} \to G_{\Gamma_{0}} \times C_{2}^{I_{0}}$.
\end{enumerate}
\end{thm}

\begin{proof}
(1) We first make an observation about the theory $T_{\Gamma,\mu}$.  Suppose $B$ is a finite quotient of $\widetilde{G_{\Gamma} \times C_{2}^{\mu}}$.  Then by Fact \ref{fact:quotientbyfrattini}, $A := B/\Phi(B)$ is a (finite) quotient of $G_{\Gamma} \times C_{2}^{\mu}$. By Lemma \ref{lem: bounding}, there is some natural number $N = N(|A|)$ such that $A$ is a quotient of $G_{\Gamma_{0}} \times C_{2}^{\ell}$, itself a quotient of $G_{\Gamma} \times C_{2}^{\mu}$ where $\Gamma_{0} = (V_{0},R_{0}) \subseteq \Gamma$ is an induced subgraph with $|V_{0}| \leq N$ and also $\ell \leq N$.  It follows, then, that $S(\widetilde{G_{\Gamma} \times C_{2}^{\mu}})$ satisfies an $L_{\mathrm{IS}}$-sentence which naturally asserts the following:
\begin{center}
$(*)_{B}$ $\begin{cases}
\begin{minipage}{0.75\textwidth}
For every $\beta \in X_{|B|}$ with $[\beta] \cong B$, there is some $\alpha \geq \beta$ with $[\alpha] \cong A$ and there are, for some $k,k' \leq N$, $\gamma_{0}, \ldots, \gamma_{k-1} \in X_{2p}$ with each $[\gamma_{i}] \cong D_{p}$ and $\delta_{0}, \ldots, \delta_{k'-1} \in X_{2}$ with each $[\delta_{i}] \cong C_{2}$ such that if, for some $k'' \leq \binom{N}{2}$, $\epsilon_{0}, \ldots, \epsilon_{k''-1} \in X_{4p^{2}q}$ enumerate representatives of all classes $[\epsilon] \cong W$ such that there are some $i < j < k$ with $\epsilon \leq \gamma_{i} \wedge \gamma_{j}$, then
$$
\alpha \geq \left(\bigwedge_{i < k} \gamma_{i} \wedge \bigwedge_{i' < k'} \delta_{i'} \wedge \bigwedge_{i'' < k''} \epsilon_{i''}\right).
$$
\end{minipage}
\end{cases}$
\end{center}
To see this, just note that the $\gamma_{i}$ correspond to the kernels $\ker(\pi_{v})$ for each $v \in V_{0}$, the $\epsilon_{i''}$ correspond to the kernels $\ker(\pi_{r})$ for each $r \in R_{0}$, and the $\delta_{i'}$ correspond to the $C_{2}$ quotient obtained by projecting $G_{\Gamma_{0}} \times C_{2}^{\ell}$ to each $C_{2}$ coordinate (so, here, $k' = \ell$).  The equivalence class of their meet is the smallest quotient of $\widetilde{G_{\Gamma} \times C_{2}^{\mu}}$ that has them all as a quotient, which must be $G_{\Gamma_{0}} \times C_{2}^{\ell}$.  Therefore, the sentence $(*)_{B}$ simply restates that every $B$ quotient of $\widetilde{G_{\Gamma} \times C_{2}^{\mu}}$ has the property that its further quotient $B/\Phi(B)$ is itself a quotient of $G_{\Gamma_{0}} \times C_{2}^{\ell}$ for an induced subgraph $\Gamma_{0} \subseteq \Gamma$ with at most $N$ vertices and with $\ell \leq N$. 

Now to conclude, suppose $\beta \in S_{0}$ and $[\beta] \cong B$. Since $S_{0} \subseteq S_{1}$ and $S_{1} \equiv S(\widetilde{G_{\Gamma} \times C_{2}^{\mu}})$, $S(\widetilde{G_{\Gamma} \times C_{2}^{\mu}})$ must also contain an element whose equivalence class is isomorphic to $B$ and $S_{1}$ satisfies the sentence $(*)_{B}$.  For the associated bound $N$, we know there is some induced subgraph $\Gamma_{*} \subseteq \Gamma_{1}$ with at most $N$ vertices and some $\ell \leq N$ such that $G((S_{1})_{+})$ has $G_{\Gamma_{*}} \times C_{2}^{\ell}$ as a quotient and $[\beta]/\Phi([\beta])$ is a quotient of this group. Let $\gamma \in S_{0}$ be an element with $[\gamma] = [\beta]/\Phi([\beta])$. Then, since $[\beta]/\Phi([\beta])$ is a quotient of $G_{\Gamma_{1}} \times C_{2}^{I_{1}}$, we have $\gamma \in (S_{1})_{+}$. By Lemma \ref{lem: plus intersection}, then, we get $\gamma \in (S_{0})_{+}$. We have thus shown that every finite quotient of $G(S_{0})$ is a Frattini cover of a finite quotient of $G((S_{0})_{+})$. Thus, we know, by Fact \ref{fact: Frattini char}, that the map $G(S_{0}) \to G((S_{0})_{+})$ is a Frattini cover, completing the proof. 

(2) By (1), the epimorphism $G(S_{*}) \to G((S_{*})_{+})$ induced by the inclusion $(S_{*})_{+} \to S_{*}$ is a Frattini cover. Since $S(\reallywidetilde{G_{\Gamma} \times C_{2}^{\mu}})$ is projective and $S_{*} \models T_{\Gamma,\mu}$, we also have $S_{*}$ is projective by Fact \ref{fact:axiomatizable}. This shows $G(S_{*}) \to G((S_{*})_{+})$ is the universal Frattini cover.    

(3) Let $\psi_{0} : \reallywidetilde{G_{\Gamma_{0}} \times C_{2}^{I_{0}}} \to G(S_{0})$ be the isomorphism dual to $\Psi_{0}$. We have an epimorphism $G(S_{0}) \to G((S_{0})_{+})$ dual to the inclusion and an epimorphism $\reallywidetilde{G_{\Gamma_{0}} \times C_{2}^{I_{0}}} \to G_{\Gamma_{0}}\times C_{2}^{I_{0}}$, by the definition of the universal Frattini cover, so there is a unique $\chi_{0}$ such that the following diagram commutes. 
    $$
    \xymatrix{ 
\widetilde{G_{\Gamma_{0}} \times C_{2}^{I_{0}}}\ar@{->}[r]^{\psi_{0}} \ar@{->}[d] & G(S_{0})  \ar@{->}[d] \\
G_{\Gamma_{0}} \times C_{2}^{I_{0}} \ar@{->}[r]^{\chi_{0}} & G((S_{0})_{+}) .
}
    $$

By Lemma \ref{lem: secondstep}(2), there are $\Gamma_{1} \supseteq \Gamma_{0}$ and $I_{1} \supseteq I_{0}$ and an isomorphism $\chi_{1} : G_{\Gamma_{1}} \times C_{2}^{I_{1}} \to G((S_{1})_{+})$ such that there is a commutative diagram 
    $$
    \xymatrix{ 
G(S_{1}) \ar@{->}[r] \ar@{->}[d] & G((S_{1})_{+}) \ar@{->}[r]^{\chi^{-1}_{1}} \ar@{->}[d] & G_{\Gamma_{1}} \times C_{2}^{I_{1}} \ar@{->}[d]^{\pi} \\
G(S_{0}) \ar@{->}[r]& G((S_{0})_{+}) \ar@{->}[r]^{\chi^{-1}_{0}} &  G_{\Gamma_{0}} \times C_{2}^{I_{0}}.
}
    $$
where $\pi$ is the projection. By (1), we have isomorphisms $\theta_{i} :  \reallywidetilde{G_{\Gamma_{i}} \times C_{2}^{I_{i}}} \to G(S_{i})$ for $i = 0,1$.  Let $\rho : G(S_{1}) \to G(S_{0})$ denote the epimorphism dual to the inclusion $S_{0} \subseteq S_{1}$ and define $\tilde{\pi} = \theta^{-1}_{0} \circ \rho \circ \theta_{1}$. Then the diagram
    $$
    \xymatrix{ 
\widetilde{G_{\Gamma_{1}} \times C_{2}^{I_{1}}} \ar@{->}[r]^{\theta_{1}} \ar@{->}[d]^{\tilde{\pi}} & G(S_{1}) \ar@{->}[r] \ar@{->}[d]^{\rho} & G((S_{1})_{+}) \ar@{->}[r]^{\psi_{1}} \ar@{->}[d] & G_{\Gamma_{1}} \times C_{2}^{I_{1}} \ar@{->}[d]^{\pi} \\
\widetilde{G_{\Gamma_{0}} \times C_{2}^{I_{0}}} \ar@{->}[r]^{\theta_{0}} & G(S_{0}) \ar@{->}[r] & G((S_{0})_{+}) \ar@{->}[r]^{\psi_{0}} &  G_{\Gamma_{0}} \times C_{2}^{I_{0}}.
}
    $$
    commutes, which entails that $\tilde{\pi}$ is a lift of $\pi$. Then we are done by duality. 
\end{proof}

\section{Model-theoretic analysis} \label{sec: model theory}

\subsection{Characterizing the theories}

Let $p_{\infty}(x)$ be the partial type in the $X_{2}$ sort axiomatized by $\{[x] \cong C_{2}\} \cup \{\neg \varphi_{n}(x) : n < \omega\}$, where $\varphi_{n}(x)$ is the formula from Lemma \ref{lem: definability of vertex width} defining the elements of vertex width $n$. A $C_{2}$-element is a realization of $p_{\infty}$ if and only if it is an element of infinite vertex width.  Given a model $S \models T_{\Gamma,\mu}$, define $\mathrm{dim}_{\infty}(S)$ to be the maximal cardinality $\kappa$ of a sequence $(\alpha_{i})_{i < \kappa}$ in $p_{\infty}(S)$ (or, equivalently, in $p_{\infty}(S_{+})$) with $\alpha_{i} \not\in \langle S_{-},\alpha_{<i} \rangle$.  By Lemma \ref{lem: dictionary}, any two such sequences have the same cardinality, as the length of such a sequence is equal to the dimension of the continuous dual of $\ker(G(S_{+})\to G(S_{-}))$, viewed as an $\mathbb{F}_{2}$-vector space.  The following lemma, together with compactness, entails that $p_{\infty}$ is a consistent partial type, whenever $\Gamma$ is an infinite graph:

\begin{lem} \label{lem: arbitrary vertex width}
Suppose $\Gamma$ is an infinite graph.  Then, for all natural numbers $n$, there is some $\alpha \in S(\widetilde{G_{\Gamma}})$ such that $[\alpha]_{\sim} \cong C_{2}$ and the vertex width of $\alpha$ is $n$. 
\end{lem}

\begin{proof}
    The $n=1$ case is clear so we may assume $n \geq 2$. Let $V_{n}$ be an arbitrary set of $n$ vertices from $\Gamma$. We know that $D_{p}^{V_{n}}$ is a quotient of $\widetilde{G_{\Gamma}}$ and, quotienting further by the normal subgroup $C_{p}^{V_{n}}$, we obtain $C_{2}^{V_{n}}$ as a quotient.  Let $N$ denote the normal subgroup of $C_{2}^{V_{n}}$ consisting of those $\overline{g} = (g_{v})_{v \in V_{n}}$ such that the number of $v$ with $g_{v} = -1$ is even.  It is easy to check that this is a normal subgroup, $C_{2}^{V_{n}}/N \cong C_{2}$, and the natural map $C_{2}^{V_{n}} \to C_{2}^{V_{n}}/N$ does not factor through a projection $C_{2}^{V_{n}} \to C_{2}^{V'}$ for any proper $V' \subseteq V_{n}$.  
\end{proof}

\begin{rem}
    By the above Lemma, when $\Gamma$ is an infinite graph, the theory $T_{\Gamma,\mu}$ is not $\aleph_{0}$-categorical for any $\mu$.  This is because there are elements of any model of $T_{\Gamma,\mu}$ of arbitrary finite vertex width and by Lemma \ref{lem: definability of vertex width} these satisfy different formulas in the $X_{2}$-sort.  This implies by the Ryll-Nardzewski theorem that $T_{\Gamma,\mu}$ is not $\aleph_{0}$-categorical.
\end{rem}

\begin{lem} \label{lem: complete invariants}
    Suppose $\Gamma$ is an infinite graph and $\mu$ is a cardinal.  A model $S \models T_{\Gamma, \mu}$ is determined up to isomorphism by the isomorphism type of the graph $\Gamma(S)$ and $\mathrm{dim}_{\infty}(S)$. 
\end{lem}

\begin{proof}
    Using the notation for the subsystems defined in Definition \ref{defn: distinguished subsystems}, we know that if $S \models T_{\Gamma,\mu}$, then $S_{-} \cong S(G_{\Gamma(S)})$ and $S_{+} \cong S(G_{\Gamma(S)} \times C_{2}^{I})$ for some set $I$ with $|I| = \mathrm{dim}_{\infty}(S)$.  Hence if $S,S' \models T_{\Gamma,\mu}$ satisfy that $\Gamma(S) \cong \Gamma(S')$ and $\mathrm{dim}_{\infty}(S) = \mathrm{dim}_{\infty}(S')$ then for some $\Gamma_{*}$ and set $I_{*}$ we have $S_{+} \cong S(G_{\Gamma_{*}} \times C_{2}^{I_{*}}) \cong S'_{+}$ and thus both $S$ and $S'$ are isomorphic to $S(\widetilde{G_{\Gamma_{*}} \times C_{2}^{I_{*}}})$ and therefore isomorphic to one another. 
\end{proof}

Now we come to the main point of this subsection, which is to prove that the theory $T_{\Gamma,\mu}$ is independent of $\mu$:

\begin{prop} \label{prop: only one theory}
    If $\Gamma$ and $\Gamma'$ are infinite elementarily equivalent graphs and $\mu$ and $\mu'$ are cardinals (possibly finite or even $0$), then $T_{\Gamma,\mu} = T_{\Gamma',\mu'}$. 
\end{prop}

\begin{proof}
    By a standard forcing and absoluteness argument (see, e.g., \cite{halevi2023saturated}), we may assume that for some regular uncountable cardinal $\kappa$, all complete theories in a countable language have a saturated model of cardinality $\kappa$. We will show that saturated models of size $\kappa$ of $T_{\Gamma,\mu}$ and $T_{\Gamma,\mu'}$ are isomorphic.   

    First, we note that if $S \models T_{\Gamma,\mu}$ is a $\kappa$-saturated model, then $\mathrm{dim}_{\infty}(S) \geq \kappa$. By Lemma \ref{lem: arbitrary vertex width}, $p_{\infty}$ is a consistent partial type so, by saturation, $\mathrm{dim}_{\infty}(S) \geq 1$. Moreover, if $\lambda < \kappa$ and we are given a sequence $J = (\alpha_{i})_{i < \lambda}$ of elements of $p_{\infty}(S)$ with $\alpha_{i} \not\in \langle S_{-}, \alpha_{<i} \rangle$ for all $i < \lambda$, then consider the partial type over $J$ which asserts that for each $n < \omega$ and each finite $w \subseteq \lambda$, there are not $n$ $C_{2}$-elements of vertex width $1$ such that $x$ is greater than or equal to the common meet of these $n$ $C_{2}$-elements and $(\alpha_{i})_{i \in w}$ in the partial order $\geq$ on $S$. Because there are infinitely many $C_{2}$-elements in $S$, this is a consistent partial type over a set of parameters of size $<\kappa$, which must be realized in $S$.  Additionally, if $\alpha_{*}$ is a realization, then $\alpha_{*} \not\in \langle S_{-}, \alpha_{<\lambda} \rangle$. It follows that $\mathrm{dim}_{\infty}(S) \geq \kappa$. 

    To conclude, suppose $S \models T_{\Gamma,\mu}$ and $S' \models T_{\Gamma',\mu'}$ are both saturated models of size $\kappa$. Then, the interpreted graphs $\Gamma(S)$ and $\Gamma(S')$ are elementarily equivalent and saturated of size $\kappa$ so $\Gamma(S) \cong \Gamma(S')$.  Moreover, by the above, we have $\mathrm{dim}_{\infty}(S) = \kappa = \mathrm{dim}_{\infty}(S')$.  Thus, by Lemma \ref{lem: complete invariants}, $S \cong S'$. In particular, $S$ and $S'$ have the same theory, so $T_{\Gamma,\mu} = T_{\Gamma',\mu'}$. 
\end{proof}

From here on out, we will drop the subscript $\mu$ and refer to the theory of $S(\widetilde{G_{\Gamma}})$ as $T_{\Gamma}$. Let $\overline{\kappa}$ denote a sufficiently large cardinal. Given an infinite graph $\Gamma$, we will choose some $\Gamma_{*}$ which is $\overline{\kappa}$-saturated and strongly $\overline{\kappa}$-homogeneous, and we will denote by $\mathbb{S}_{\Gamma} \models T_{\Gamma}$ the model 
$$
\mathbb{S}_{\Gamma} = S(\reallywidetilde{G_{\Gamma_{*}}\times C_{2}^{\overline{\kappa}}}).
$$
This will turn out (see Lemma \ref{lem: monster} below) to be a $\overline{\kappa}$-saturated and strongly $\overline{\kappa}$-homogeneous model of $T_{\Gamma}$, and thus can serve as a monster model of $T_{\Gamma}$. We will use the word \emph{small} to mean of cardinality $<\overline{\kappa}$. If $\Gamma_{0} \subseteq \Gamma_{*}$ is an induced subgraph, then the subsystem generated by the $D_{p}$-elements corresponding to the vertices of $\Gamma_{0}$ and the $W$-elements corresponding to the edges of $\Gamma_{0}$ may be canonically identified with $S(G_{\Gamma_{0}})$.  Moreover, by Corollary \ref{cor: unique subgroup}, the subsystem generated by the $\alpha$ such that $[\alpha]$ is a Frattini $pq$-cover of some $[\beta]$ for $\beta \in S(G_{\Gamma_{0}})$ can be canonically identified with $S({_{pq}\widetilde{G_{\Gamma_{0}}}})$. Thus, we will refer to these subsystems of $\mathbb{S}_{\Gamma}$ by $S(G_{\Gamma_{0}})$ and $S({_{pq}\widetilde{G_{\Gamma_{0}}}})$.  Note, there are, in general, many subsystems $S \supseteq S(G_{\Gamma_{0}})$ of $\mathbb{S}_{\Gamma}$ such that the induced epimorphism $G(S) \to G(S(G_{\Gamma_{0}}))$ is a universal Frattini cover so we cannot write $S(\widetilde{G_{\Gamma_{0}}})$ to unambiguously refer to a subsystem of $\mathbb{S}_{\Gamma}$. 

\subsection{Types and algebraic closure}

For the rest of the section, we will work with respect to a fixed infinite graph $\Gamma$ with $\mathbb{S}_{\Gamma} = S(\reallywidetilde{G_{\Gamma_{*}} \times C_{2}^{\overline{\kappa}}})$.

\begin{lem} \label{lem: further gcl props}
Suppose $S \subseteq \mathbb{S}_{\Gamma}$ is a small subsystem.
\begin{enumerate}
\item $S$ is graph closed if and only if $S_{+}$ is graph closed if and only if $S_{-}$ is graph closed. 
\item If $\Gamma_{0} \subseteq \Gamma_{*}$ is an induced subgraph and $(\gamma_{i})_{i < \lambda}$ is an independent sequence of $C_{2}$-elements of $\mathbb{S}_{\Gamma}$ of infinite vertex width, then $\langle S(G_{\Gamma_{0}}), (\gamma_{i})_{i < \lambda} \rangle$ is a graph closed subsystem of $\mathbb{S}_{\Gamma}$. 
\item If $S$ is graph closed and $\Gamma(S) \subseteq \Gamma' \subseteq \Gamma_{*}$ where $\Gamma'$ is an induced subgraph, then $\langle S,S(G_{\Gamma'}) \rangle$ is graph closed. 
\end{enumerate}
\end{lem}

\begin{proof}
(1) This is immediate from the definitions, since any $C_{2}$-element of finite vertex width in $S$ necessarily is an element of $S_{+}$, and additionally any $D_{p}$- or $W$-elements of $S$ also live in $S_{+}$. Thus $S$ is graph closed if and only if $S_{+}$ is graph closed.

(2) Let $S' = \langle S(G_{\Gamma_{0}}), (\gamma_{i})_{i < \lambda} \rangle$. Write $\Gamma_{0} = (V_{0},R_{0})$. Because each epimorphism $G(S') \to C_{2}$ must factor through the quotient of $G(S')$ by the product of the $p$- and $q$-Sylow subgroups, we have that if $\alpha \in S'$ is a $C_{2}$-element, $\alpha \in \langle S(C_{2}^{V_{0}}), (\gamma_{i})_{i < \lambda} \rangle = \langle (\alpha_{v})_{v \in V_{0}}, (\gamma_{i})_{i < \lambda} \rangle$, where $\alpha_{v}$ is a $C_{2}$-element of vertex width $1$ corresponding to the vertex $v$ for each $v \in V_{0}$. 

Suppose now that $\alpha \in S'$ is a $C_{2}$-element of finite vertex width $n$ and let $(\beta_{i})_{i < n}$ be a sequence of $C_{2}$-elements of vertex width $1$ corresponding to the witnesses. Then using Lemma \ref{lem: dictionary}, we have, for some finite sets $X \subseteq V_{0}$, $Y \subseteq \lambda$
$$
d([\alpha]) = \sum_{x \in X} d([\alpha_{x}]) + \sum_{y \in Y} d([\gamma_{y}]) = \sum_{i < n} d([\beta_{i}]),
$$
which entails 
$$
\sum_{y \in Y} d([\gamma_{y}]) = \sum_{i < n} d([\beta_{i}]) - \sum_{x \in X} d([\alpha_{x}]),
$$
but if non-zero, the sum on the left-hand side corresponds to a $C_{2}$-element of infinite vertex width, while the one on the right corresponds to a $C_{2}$-element of finite vertex width, so both sums must be zero. This shows $d([\alpha]) = \sum_{x \in X} d([\alpha_{x}])$, or, in other words, $\alpha \geq \bigwedge_{x \in X} \alpha_{x}$, so $\alpha \in S(G_{\Gamma_{0}}) \subseteq S'$. Since $\Gamma_{0}$ is an induced subgraph of $\Gamma_{*}$, we get, then, that $S'$ is graph closed.  

(3) It suffices to consider the case that the vertex set of $\Gamma'$ contains exactly one vertex $v$ not contained in $\Gamma(S)$, since then we may obtain the general case by iteration (possibly transfinitely). Let $\alpha_{v}$ be a $C_{2}$-element of $\mathbb{S}_{\Gamma}$ of vertex width $1$ corresponding to $v$. Then since $\alpha_{v} \not\in S$, we have $G(\langle S, \alpha_{v} \rangle) \cong G(S) \times C_{2}$. Thus the Frattini subgroup $\Phi(G(\langle S,\alpha_{v} \rangle))$ is just $\Phi(G(S))$. Since $\Phi(C_{2})= 1$, any epimorphism $G(\langle S,\alpha_{v} \rangle) \to C_{2}$ must factor through the quotient $G(\langle S,\alpha_{v} \rangle)/\Phi(G(S))$ which, by duality, entails that if $\alpha \in \langle S, \alpha_{v} \rangle$ is a $C_{2}$-element, then $\alpha \in \langle S_{+},\alpha_{v} \rangle$. By Lemma \ref{lem: secondstep}, we have $S_{+} = \langle S(G_{\Gamma(S)}), (\gamma_{i})_{i < \lambda} \rangle$ where $(\gamma_{i})_{i < \lambda}$ is an independent sequence of elements of infinite vertex width. Then by the proof of (2), we see that if $\alpha \in \langle S_{+},\alpha_{v} \rangle$ is a $C_{2}$-element of finite vertex width, then the witnesses must correspond to vertices in $\Gamma'$. Since any $C_{2}$-element of $\langle S, S(G_{\Gamma'}) \rangle$ must be contained in $\langle S_{+}, \alpha_{v} \rangle$ and $\Gamma'$ was chosen to be an induced subgraph of $\Gamma_{*}$, we see $\langle S, S(G_{\Gamma'}) \rangle$ is graph closed. 
\end{proof}

\begin{lem} \label{lem: type char}
    \begin{enumerate}
    \item Suppose $S_{0}$ and $S_{1}$ are small subsystems of $\mathbb{S}_{\Gamma}$ with $\mathrm{gcl}(S_{i}) = S_{i}$ for $i = 0,1$. If $\Xi :S_{0} \to S_{1}$ is an isomorphism such that $\Gamma(\Xi) : \Gamma(S_{0}) \to \Gamma(S_{1})$ is partial elementary with respect to $\mathrm{Th}(\Gamma)$, then $\Xi$ is partial elementary, with respect to $T_{\Gamma}$. 
    \item Suppose $S \models T_{\Gamma}$ is a small model of $T_{\Gamma}$. If $\Psi : S \to \mathbb{S}_{\Gamma}$ is an embedding of $L_{\mathrm{IS}}$-structures such that $\Gamma(\Psi)$ is elementary, with respect to $\mathrm{Th}(\Gamma)$, then $\Psi$ is an elementary embedding. 
    \end{enumerate}
\end{lem}

\begin{proof}
    (1) Write $\Gamma_{i}$ for $\Gamma(S_{i})$ for $i = 0,1$. Note that we have $(S_{i})_{-} = S(G_{\Gamma_{i}})$ for $i = 0,1$. Because $\Gamma(\Xi): \Gamma_{0} \to \Gamma_{1}$ is partial elementary and $\Gamma_{*}$ is strongly $\overline{\kappa}$-homogeneous, there is some $\sigma \in \mathrm{Aut}(\Gamma_{*})$ extending $\Gamma(\Xi)$. We know that $\Xi$ restricts to $S(G_{\Gamma_{0}}) \subseteq S(\widetilde{G_{\Gamma_{0}} \times C_{2}^{I_{0}}})$ to give an isomorphism between $S(G_{\Gamma_{0}})$ and $S(G_{\Gamma_{1}})$. Thus, by Lemma \ref{lem: lifting}, there is some $\Theta : S(G_{\Gamma_{*}}) \to S(G_{\Gamma_{*}})$ such that $\Gamma(\Theta) = \sigma$ and $\Theta$ extends $\Xi|_{S(G_{\Gamma_{0}})}$. By Lemma \ref{lem: transfer}(2), $(\mathbb{S}_{\Gamma})_{-} \vee S_{i} = (S_{i})_{-}$ for $i = 0,1$ entails that 
    $$
    \langle (\mathbb{S}_{\Gamma})_{-}, S_{i} \rangle = (\mathbb{S}_{\Gamma})_{-} \wedge S_{i} \cong S(G_{\Gamma_{*}} \times_{G_{\Gamma_{i}}} G(S_{i}))
    $$
    for $i = 0,1$.  Therefore, there is a unique map
    $$
    \Psi_{0} : \langle S(G_{\Gamma_{*}}), S_{0} \rangle \to \langle S(G_{\Gamma_{*}}), S_{1} \rangle
    $$
    extending $\Xi$ and $\Theta$. Choose a maximal sequence $(\alpha_{i})_{i < \gamma}$ of elements in $S_{0}$ with $[\alpha_{i}] \cong C_{2}$ and $\alpha_{i} \not\in \langle (S_{0})_{-}, \alpha_{<i} \rangle$. Then, again using the fact that $(\mathbb{S}_{\Gamma})_{-} \vee S_{0} = (S_{0})_{-}$, we have $\alpha_{i} \not\in \langle (\mathbb{S}_{\Gamma})_{-}, \alpha_{<i} \rangle$ for all $i < \gamma$, so we may extend this sequence to $(\alpha_{i})_{i < \overline{\kappa}}$ which is a maximal sequence of elements in $(\mathbb{S}_{\Gamma})_{+}$ such that $[\alpha_{i}] \cong C_{2}$ and $\alpha_{i} \not\in \langle (\mathbb{S}_{\Gamma})_{-}, \alpha_{<i} \rangle$ for all $i < \overline{\kappa}$. Define $\alpha'_{i} = \Xi(\alpha_{i})$ for $i < \gamma$. Then $\alpha'_{i} \not\in \langle (S_{1})_{-}, \alpha'_{<i} \rangle$ for all $i < \gamma$, and we may likewise extend this sequence to a maximal sequence $(\alpha'_{i})_{i < \overline{\kappa}}$ in $(\mathbb{S}_{\Gamma})_{+}$ such that $\alpha'_{i} \not\in \langle (\mathbb{S}_{\Gamma})_{-}, \alpha'_{<i} \rangle$. By , we have 
    $$
    \langle (\mathbb{S}_{\Gamma})_{-}, \alpha_{<\overline{\kappa}}\rangle = \langle (\mathbb{S}_{\Gamma})_{-}, \alpha'_{<\overline{\kappa}} \rangle = (\mathbb{S}_{\Gamma})_{+}. 
    $$
    Moreover, using Fact \ref{fact2}(3) and the fact that $\Phi(C_{2})=1$, we have
    $$
    \alpha_{i} \not\in \tilde{S}
    $$
    for any subsystem $\tilde{S} \supseteq \langle (\mathbb{S}_{\Gamma})_{-}, \alpha_{<i}\rangle$ such that the induced map $G(\tilde{S}) \to G(\langle (\mathbb{S}_{\Gamma})_{-},\alpha_{<i} \rangle)$ is Frattini, for all $i < \overline{\kappa}$. In particular, since $G(S_{0}) \to G((S_{0})_{+})$ is necessarily Frattini, we have $\alpha_{i} \not\in \langle (\mathbb{S}_{\Gamma})_{-}, S_{0}, \alpha_{<i} \rangle$ for all $\gamma \leq i < \overline{\kappa}$. By an identical argument, we have $\alpha'_{i} \not\in \langle (\mathbb{S}_{\Gamma})_{-}, S_{1}, \alpha'_{<i} \rangle$ for all $\gamma \leq i < \overline{\kappa}$.  Thus, the map $\Psi_{0}$ extends to a map $\Psi_{1}: \langle (\mathbb{S}_{\Gamma})_{+}, S_{0} \rangle \to \langle (\mathbb{S}_{\Gamma})_{+}, S_{1} \rangle$, defined by mapping $\alpha_{i} \mapsto \alpha'_{i}$ for all $i < \overline{\kappa}$, since $\langle (\mathbb{S}_{\Gamma})_{+}, S_{0} \rangle = \langle (\mathbb{S}_{\Gamma})_{-}, S_{0}, (\alpha_{i})_{\gamma \leq i < \overline{\kappa}} \rangle$ and $\langle (\mathbb{S}_{\Gamma})_{+}, S_{1} \rangle = \langle (\mathbb{S}_{\Gamma})_{-}, S_{1}, (\alpha'_{i})_{\gamma \leq i < \overline{\kappa}} \rangle $. Then $\Psi_{1}$ lifts to an isomorphism $\Psi_{2}$ between universal Frattini covers of the domain and codomain, which is an automorphism $\Psi_{2} : \mathbb{S}_{\Gamma} \to \mathbb{S}_{\Gamma}$. 

    (2) Given a small $S \models T_{\Gamma}$, by Theorem \ref{thm: thirdstep}, we can, without loss of generality, assume $S = S(\reallywidetilde{G_{\Gamma_{0}}\times C_{2}^{I_{0}}})$ for some small $\Gamma_{0} \models \mathrm{Th}(\Gamma)$ and small set $I_{0} \subseteq \overline{\kappa}$. Since $\Gamma_{*}$ is $\overline{\kappa}$-saturated, we can find some $\Gamma' \cong \Gamma_{*}$ such that $\Gamma_{0} \preceq \Gamma'$.  The projection $\pi : G_{\Gamma'} \times C_{2}^{\overline{\kappa}} \to G_{\Gamma_{0}} \times C_{2}^{I_{0}}$ lifts to an epimorphism $\tilde{\pi} : \reallywidetilde{G_{\Gamma'}\times C_{2}^{\overline{\kappa}}} \to \reallywidetilde{G_{\Gamma_{0}} \times C_{2}^{I_{0}}}$. Then $S(\tilde{\pi}): S(\reallywidetilde{G_{\Gamma_{0}}\times C_{2}^{I_{0}}}) \to S(\reallywidetilde{G_{\Gamma'}\times C_{2}^{\overline{\kappa}}})$ is an $L_{\mathrm{IS}}$-embedding, so there is some $S' \supseteq S$ and an isomorphism $\Theta : S' \to S(\reallywidetilde{G_{\Gamma'}\times C_{2}^{\overline{\kappa}}})$ such that $\Theta|_{S} = \Psi$. Moreover, since $\Gamma' \cong \Gamma_{*}$, it is clear that $S(\reallywidetilde{G_{\Gamma'}\times C_{2}^{\overline{\kappa}}}) \cong \mathbb{S}_{\Gamma}$, so there is an isomorphism $\Theta' : S' \to \mathbb{S}_{\Gamma}$. Thus, we have $S \equiv \Theta'(S)$. And, by (1), $\Psi(S) \equiv S$. This shows $S \equiv \Psi(S)$, so $\Psi$ is elementary. 
\end{proof}

\begin{lem} \label{lem: monster}
The structure $\mathbb{S}_{\Gamma}$ is a monster model of $T_{\Gamma}$\textemdash that is, $\mathbb{S}_{\Gamma}$ is $\overline{\kappa}$-saturated and strongly $\overline{\kappa}$-homogeneous.
\end{lem}

\begin{proof}
By standard model theory, to show that $\mathbb{S}_{\Gamma}$ is $\overline{\kappa}$-saturated, we must show that it is $\overline{\kappa}$-universal and $\overline{\kappa}$-homogeneous (which is, of course, implied by being $\overline{\kappa}$-strongly homogeneous). So, towards showing universality, fix some $S \models T_{\Gamma}$ with $|S| \leq \overline{\kappa}$. By Theorem \ref{thm: thirdstep}, we may assume $S = S(\reallywidetilde{G_{\Gamma_{0}}\times C_{2}^{I_{0}}})$ for some graph $\Gamma_{0} = (V_{0},R_{0}) \models \mathrm{Th}(\Gamma)$ and set $I_{0}$ with $|\Gamma_{0}| \leq \overline{\kappa}$ and $|I_{0}| \leq \overline{\kappa}$. Since $\Gamma$ was chosen to be $\overline{\kappa}$-saturated, there is an elementary embedding $\iota_{0} : \Gamma_{0} \to \Gamma$ and, additionally, there is some injection $\iota_{1}: I_{0} \to \overline{\kappa}$. There is naturally an embedding $\Psi : S(G_{\Gamma_{0}}\times C_{2}^{I_{0}}) \to (\mathbb{S}_{\Gamma})_{-}$ with $\Psi([\ker(\pi_{v})]) = [\ker(\pi_{\iota_{0}(v)})]$ for all $v \in V_{0}$ and $\Psi([\ker(\pi_{i})]) = [\ker(\pi_{\iota_{1}(i)})]$ for all $i \in I_{0}$. If $\psi : G_{\Gamma} \times C_{2}^{\overline{\kappa}} \to G_{\Gamma_{0}} \times C_{2}^{I_{0}}$ is dual to $\Psi$, then we know there is a lift $\psi' : \reallywidetilde{G_{\Gamma} \times C_{2}^{\overline{\kappa}}} \to \reallywidetilde{G_{\Gamma_{0}}\times C_{2}^{I_{0}}}$ so the dual epimorphism $\Psi' : S \to \mathbb{S}_{\Gamma}$ gives an $L_{\mathrm{IS}}$-embedding which induces a $\mathrm{Th}(\Gamma)$-elementary embedding on interpreted graphs. By Lemma \ref{lem: type char}(2), we get $\Psi'$ is elementary. This shows $\mathbb{S}_{\Gamma}$ is $\overline{\kappa}$-universal. 

If $A,B \subseteq \mathbb{S}_{\Gamma}$ and $\Theta : A \to B$ is a partial elementary isomorphism, then $\Theta$ extends to an isomorphism between $\mathrm{acl}(A)$ and $\mathrm{acl}(B)$ so strong $\overline{\kappa}$-homogeneity follows from Lemma \ref{lem: type char}(1), using that the graph closure is contained in the algebraic closure, by Lemma \ref{lem: graph closure is closure}(3). 
\end{proof}

\begin{lem} \label{lem: pq cover system}
    Suppose $S = \mathrm{gcl}(S)$ is a small subsystem of $\mathbb{S}_{\Gamma}$. Then there is a unique subsystem $S'$ with $S \subseteq S' \subseteq \mathbb{S}_{\Gamma}$ and such that the epimorphism $G(S') \to G(S)$ induced by the inclusion $S \subseteq S'$ is a universal $pq$-Frattini cover of $G(S)$. 
\end{lem}

\begin{proof}
    By Lemma \ref{lem: secondstep}, we may represent $S_{+}$ as $S(G_{\Gamma_{0}} \times C_{2}^{I_{0}})$ and $(\mathbb{S}_{\Gamma})_{+}$ as $S(G_{\Gamma} \times C_{2}^{\overline{\kappa}})$ so that $\Gamma_{0} \subseteq \Gamma$ is an induced subgraph, $I_{0} \subseteq \overline{\kappa}$ is a subset, and the epimorphism $G_{\Gamma} \times C_{2}^{\overline{\kappa}} \to G_{\Gamma_{0}} \times C_{2}^{I_{0}}$ is the projection to the $\Gamma_{0}$ and $I_{0}$ coordinates. By Corollary \ref{cor: unique subgroup} and duality, there is a unique subsystem $S''$ with $S_{+} \subseteq S'' \subseteq \mathbb{S}_{\Gamma}$ such that the epimorphism $G(S'') \to G(S_{+})$ dual to the inclusion $S_{+} \subseteq S''$ is a $pq$-universal Frattini cover. 
    
Moreover, by Theorem \ref{thm: thirdstep}(1), the map $G(S) \to G(S_{+})$ induced by the inclusion $S_{+} \subseteq S$ is a Frattini cover (although not necessarily the universal one). If $N$ is the kernel of this epimorphism, then, since $N \subseteq \Phi(G(S))$, we know $N$ is a direct product of two closed normal subgroups $N_{2}$ and $N_{pq}$ where $N_{2}$ is a $2$-group and $N_{pq}$ has only finite quotients whose order is divisible by $p$ or $q$ \cite[Proposition 2.3.8]{ribes2010profinite}. Let $S_{2}$ and $S_{pq}$ be the subsystems of $S$ corresponding to the quotients $G(S)/N_{2}$ and $G(S)/N_{pq}$ . Then we have $S = S_{2} \wedge S_{pq}$ and $S_{+} = S_{2} \vee S_{pq}$. By uniqueness of the subsystem corresponding to the universal $pq$-Frattini cover of $S_{+}$, we must have $S_{pq} \subseteq S''$. Note that $S'' \vee S_{2} = S_{+}$. 

Let $S' = S \wedge S''$.  We clearly have $S \subseteq S'$ and we claim that the epimorphism $G(S') \to G(S)$ is a universal $pq$-Frattini cover. The $p$- and $q$-Sylows of $G(S')$ are isomorphic to the $p$-and $q$-Sylows of $G(S'')$ since the kernel of the epimorphism $G(S') \to G(S'')$ induced by the inclusion $S'' \subseteq S'$ is a $2$-group. Since $S''$ is a universal $pq$-Frattini cover, the $p$- and $q$-Sylows of $G(S'')$ and thus of $G(S')$ are free, which entails that $G(S')$ is $pq$-projective. Moreover, from $S' = S_{2} \wedge S''$ and $S_{+} = S_{2} \vee S''$, we have  
$$
G(S') \cong G(S'') \times_{G(S_{+})} G(S_{2})
$$
and, recall, we also have 
$$
G(S'') \cong G(S_{pq}) \times_{G(S_{+})} G(S_{2}).
$$
Thus if $\rho : G(S'') \to G(S_{pq})$ denotes the epimorphism dual to the inclusion $S_{pq} \subseteq S''$, we have that the inclusion $S'' \subseteq S'$ induces $(\rho,\mathrm{id}_{G(S_{2})}) : G(S'') \times_{G(S_{+})} G(S_{2}) \to G(S_{pq}) \times_{G(S_{+})} G(S_{2})$, which is Frattini, since $\rho$ is. Thus $G(S') \to G(S)$ is a universal $pq$-Frattini cover. 

Uniqueness follows because we did not make any choices in the decomposition of $S'$. More precisely, suppose $S \subseteq S_{*} \subseteq \mathbb{S}_{\Gamma}$ and the epimorphism $G(S_{*}) \to G(S)$ dual to the inclusion $S \subseteq S_{*}$ is a universal $pq$-Frattini cover, then $S_{*} = S'$. Since this epimorphism $G(S_{*}) \to G(S)$ is Frattini and the epimorphism $G(S) \to G(S_{+})$ dual to the inclusion $S_{+} \subseteq S$ is also Frattini, the composition $G(S_{*}) \to G(S_{+})$ is also Frattini. Then the kernel of this map splits as a direct product of closed normal subgroups $N_{2,*}$ and $N_{pq,*}$ as above, corresponding to subsystems $S_{*,2}$ and $S_{*,pq}$ with $S_{*,2} \vee S_{*,pq} = S_{+}$ and $S_{*,2} \wedge S_{*,pq} = S_{*}$. Then, as above, $G(S_{*,pq})$ is $pq$-projective and the induced map $G(S_{*,pq}) \to G(S_{+})$ is $pq$-Frattini so this map is the universal $pq$-Frattini cover of $S_{+}$ which is unique, so $S_{*,pq} = S''$ and, likewise, the fact that $G(S_{*}) \to G(S)$ is a $pq$-Frattini cover entails $S_{*,2} = S_{2}$, so we obtain $S_{*} = S_{*,pq} \wedge S_{*,2} = S'' \wedge S_{2} = S'$. 
\end{proof}

\begin{defn}
Given a set $A \subseteq \mathbb{S}_{\Gamma}$, define $\mathrm{gcl}_{pq}(A)$ to be the unique subsystem $S$ with $\mathrm{gcl}(A) \subseteq S \subseteq \mathbb{S}_{\Gamma}$ such that the epimorphism $G(S) \to G(\mathrm{gcl}(A))$, dual to the inclusion $\mathrm{gcl}(A) \subseteq S$, is a universal $pq$-Frattini cover (the existence and uniqueness of which is guaranteed by Lemma \ref{lem: pq cover system}). 
\end{defn}

\begin{lem} \label{lem: pq acl containment}
    For any small set $A \subseteq \mathbb{S}_{\Gamma}$, $\mathrm{gcl}_{pq}(A) \subseteq \mathrm{acl}(A)$. 
\end{lem}

\begin{proof}
    By Lemma \ref{lem: graph closure is closure}(3), it suffices to prove that $\mathrm{gcl}_{pq}(S) \subseteq \mathrm{acl}(S)$ in the case that $S = \mathrm{gcl}(S)$. Fix $\alpha \in \mathrm{gcl}_{pq}(S)$. Identify $S_{+}$ with $S(G_{\Gamma(S)} \times C_{2}^{I})$ for some set $I$. By Fact \ref{fact: pq quotient fact}, the quotient $[\alpha]$ is a $pq$-Frattini cover of a quotient of $G(S)$, so we know that there is some $\beta \in S$ with $\beta \geq \alpha$ and $[\beta] \cong [\alpha]/\Phi([\alpha])$. Then by Theorem \ref{thm: thirdstep}(1), there is some $\gamma \in S_{+}$ with $[\gamma] \cong [\beta]/\Phi([\beta])$ and $\gamma \geq \beta$. Then there is a finite induced subgraph $\Gamma_{0} \subseteq \Gamma(S)$ and a finite subset $I_{0} \subseteq I$ such that $\gamma \in S(G_{\Gamma_{0}} \times C_{2}^{I_{0}}) \subseteq S_{+}$. Set $S_{0} = \mathrm{gcl}(S(G_{\Gamma_{0}} \times C_{2}^{I_{0}}), \beta)$ and observe that $\Gamma(S_{0}) = \Gamma_{0}$ since $[\beta]/\Phi([\beta]) = [\gamma] \subseteq S(G_{\Gamma_{0}}\times C_{2}^{I_{0}})$. Since $G_{\Gamma_{0}} \times C_{2}^{I_{0}}$ is a finite group, we know that ${_{pq}(G_{\Gamma_{0}} \times C_{2}^{I_{0}})}$ is a bounded profinite group and thus has only finitely many quotients isomorphic to $[\alpha]$. By Lemma \ref{lem: pq cover system}, there is a unique subsystem $S_{0} \subseteq S'_{0} \subseteq \mathbb{S}_{\Gamma}$ such that the epimorphism $G(S_{0}') \to G(S_{0})$ is the universal $pq$-Frattini cover of $G(S_{0})$. By Fact \ref{fact: pq quotient fact}, if $\alpha' \in \mathbb{S}_{\Gamma}$ is any element with $[\alpha'] \cong [\alpha]$ and $\alpha' \leq \beta$ such that the canonical map $[\alpha'] \to [\beta]$ is a $pq$-Frattini cover, then $[\alpha']$ must be a quotient of $G(S_{0}')$ or, in other words, $\alpha' \in S_{0}'$.  This means there are only finitely many such $\alpha'$ and since $\beta \in S$, this establishes $\alpha \in \mathrm{acl}(S)$.     
\end{proof}

We write $a \ind^{a}_{C} b$ to mean $\mathrm{acl}(aC) \cap \mathrm{acl}(bC) = \mathrm{acl}(C)$.  We say a sequence $I = (a_{i})_{i < \omega}$ is $\ind^{a}$-Morley over $C$ if $I$ is $C$-indiscernible and $a_{i} \ind^{a}_{C} a_{<i}$ for all $i$. 

\begin{prop} \label{prop: acl description}
Suppose $S \subseteq \mathbb{S}_{\Gamma}$ is a subsystem. Then $S = \mathrm{acl}(S)$ if and only if $S = \mathrm{gcl}_{pq}(S)$ and $\Gamma(S) = \mathrm{acl}_{\Gamma}(\Gamma(S))$, where $\mathrm{acl}_{\Gamma}$ denotes algebraic closure in $\mathrm{Th}(\Gamma)$. 
\end{prop}

\begin{proof}
    By Lemma \ref{lem: pq acl containment}, $S = \mathrm{acl}(S)$ entails $S = \mathrm{gcl}_{pq}(S)$.  Moreover, it is clear that $\mathrm{acl}(S) = S$ entails $\Gamma(S) = \mathrm{acl}_{\Gamma}(\Gamma(S))$. So we assume $\mathrm{gcl}_{pq}(S) = S$ and $\Gamma(S) = \mathrm{acl}_{\Gamma}(\Gamma(S))$ and we prove $\mathrm{acl}(S) = S$. By Lemma \ref{lem: secondstep}, we may, without loss of generality, assume $S_{+} = S(G_{\Gamma(S)} \times C_{2}^{I})$ for some subset $I \subseteq \overline{\kappa}$ and the projection $\pi : G_{\Gamma} \times C_{2}^{\overline{\kappa}}\to G_{\Gamma(S)} \times C_{2}^{I}$ is dual to the inclusion $S_{+} \subseteq (\mathbb{S}_{\Gamma})_{+}$. 

    Pick any $\alpha \in \mathbb{S}_{\Gamma} \setminus S$. Let $\beta \in \mathbb{S}_{\Gamma}$ satisfy $\beta \geq \alpha$ and $[\beta] \cong [\alpha]/\Phi([\alpha])$. Then, necessarily, $\beta \in (\mathbb{S}_{\Gamma})_{+} = S(G_{\Gamma} \times C_{2}^{\overline{\kappa}})$. Choose some small induced subgraph $\Gamma_{0}$ with $\Gamma(S) \subseteq \Gamma_{0} \subseteq \Gamma$ and small set $I_{0}$ with $I \subseteq I_{0} \subseteq \overline{\kappa}$ such that $\beta \in S(G_{\Gamma_{0}} \times C_{2}^{I_{0}})$. In $\Gamma$, choose an $\ind^{a}$-Morley sequence $(\Gamma_{i})_{i < \omega}$ over $\Gamma(S)$ starting with $\Gamma_{0}$. Also choose some sequence $(I_{i})_{i < \omega}$ of subsets of $\overline{\kappa}$ such that $|I_{i}| = |I_{0}|$ and $I_{i} \cap I_{j} = I$ for all $i < j$. 
    
    We claim $S(G_{\Gamma_{0}}\times C_{2}^{I_{0}}) \cap S(G_{\Gamma_{1}}\times C_{2}^{I_{1}}) \subseteq S_{+}$. For an induced subgraph $\Gamma_{*} \subseteq \Gamma$ and subset $I_{*} \subseteq \overline{\kappa}$, let $\pi_{\Gamma_{*},I_{*}} : G_{\Gamma} \times C_{2}^{\overline{\kappa}} \to G_{\Gamma_{*}} \times C_{2}^{I_{*}}$ denote the projection. Then, for each $i < \omega$, we have 
    $$
    S(G_{\Gamma_{i}} \times C_{2}^{I_{i}}) = S\left( G_{\Gamma} \times C_{2}^{\overline{\kappa}}/\ker(\pi_{\Gamma_{i},I_{i}})\right),
    $$
    and we can calculate, for $i < j$, 
    \begin{eqnarray*}
        S_{i} \cap S_{j} &=& S\left( G_{\Gamma} \times C_{2}^{\overline{\kappa}}/\ker(\pi_{\Gamma_{i},I_{i}})\right) \vee S\left( G_{\Gamma} \times C_{2}^{\overline{\kappa}}/\ker(\pi_{\Gamma_{j},I_{j}})\right) \\
        &=& S\left(G_{\Gamma} \times C_{2}^{\overline{\kappa}}/\ker(\pi_{\Gamma_{i},I_{i}})\ker(\pi_{\Gamma_{j},I_{j}}) \right) \\
        &=& S\left( G_{\Gamma} \times C_{2}^{\overline{\kappa}}/\ker(\pi_{(\Gamma_{i} \cap \Gamma_{j}),(I_{i} \cap I_{j})}) \right) \\
        &\subseteq& S\left( G_{\Gamma} \times C_{2}^{\overline{\kappa}}/\ker(\pi_{\Gamma(S),I}) \right) \\
        &=& S_{+}. 
    \end{eqnarray*}
    Moreover, by Lemma \ref{lem: plus intersection}, $S_{i} \cap S = S_{+}$, and so 
    $$
    \langle S_{i},S \rangle = S_{i} \wedge S = S\left((G_{\Gamma_{i}}\times C_{2}^{I_{i}}) \times_{G((S)_{+})} G(S) \right).
    $$
    Therefore, if we pick an isomorphism $\Psi_{i} : S(G_{\Gamma_{0}} \times C_{2}^{I_{0}}) \to S(G_{\Gamma_{i}}\times C_{2}^{I_{i}})$ which is the identity on $S(G_{\Gamma(S)} \times C_{2}^{I}) = S_{+}$, it extends to an isomorphism $\Psi'_{i} : \langle S,S_{0} \rangle \to \langle S,S_{i} \rangle$ which is the identity on $S$. By Lemma \ref{lem: further gcl props}(3), $\langle S,S_{i} \rangle$ is graph closed and, therefore, by Lemma \ref{lem: type char}(1), each $\Psi'_{i}$ extends to some $\Theta_{i} \in \mathrm{Aut}(\mathbb{S}_{\Gamma}/S)$. Let $\beta_{i} = \Theta_{i}(\beta)$ and $\alpha_{i} = \Theta_{i}(\alpha)$ for $i < \omega$. Since $S_{i} \cap S_{j} \subseteq S_{+}$ for $i \neq j$ and $S_{i} \cap S = S_{+}$, we have either $\beta_{i} = \beta \in S$ for all $i$, or $\beta_{i} \in S_{i} \setminus S$ for all $i$ and are therefore pairwise distinct. 
    
    In the second case, in which the $\beta_{i}$ are pairwise distinct, we can easily conclude: since $[\alpha]$ is a finite group, it can only have finitely many quotients and therefore $\{\alpha_{i} : i < \omega \}$ is an infinite set. It follows that $\alpha \not\in \mathrm{acl}(S)$. 

    In the first case, in which $\beta_{i} = \beta \in S$ for all $i$, we have that $[\alpha]$ is a Frattini cover of a quotient of $G(S)$, so $[\alpha]$ is a quotient of the universal Frattini cover of $S$. Let $S'$ be any subsystem of $\mathbb{S}_{\Gamma}$ with $S \subseteq S'$, $\alpha \in S'$, and the epimorphism $G(S') \to G(S)$ is a universal Frattini cover of $S$. For any $n$, by Lemma \ref{lem: pq fibermaxxing} and duality, we can find $S' = S'_{0}, \ldots, S'_{n-1}$ which are pairwise isomorphic over $S$ with $S'_{i} \cap S'_{j} = S$ for $i \neq j$ (in the lemma $A = G(S)$). Now $(S'_{i})_{-} = S_{-}$ for all $i$ so $S'_{i}$ is graph closed by Lemma \ref{lem: further gcl props}(1). Then, for each $i$, we can define some $\Xi_{i} \in \mathrm{Aut}(\mathbb{S}_{\Gamma}/S)$ with $\Xi_{i}(S') = S'_{i}$ for each $i < n$. Setting $\alpha_{i} = \Xi_{i}(\alpha)$, we find $n$ distinct conjugates of $\alpha$ over $S$. Since $n$ is arbitrary, we conclude $\alpha \not\in \mathrm{acl}(S)$. Thus $S = \mathrm{acl}(S)$. 
\end{proof}

\begin{rem}
    The description of algebraic closure is a little unwieldy but the general idea for how to produce the algebraic closure of a set $X$ goes as follows: first take the substructure generated by $X$ and consider all of the $C_{2}$-elements.  Some of these will have finite vertex width, so you throw in the $D_{p}$ witnesses, and then you throw in the $W$ quotients corresponding to edges so you get an induced subgraph $\Gamma(\mathrm{gcl}(X))$. That graph may not be algebraically closed in the theory of $\Gamma$, so you replace it with a larger graph $\mathrm{acl}_{\Gamma}(\Gamma(\mathrm{gcl}(X)))$.  At this stage, you have that $\langle S(G_{\mathrm{acl}_{\Gamma}(\mathrm{gcl}(X))}),X_{2}(\langle X \rangle) \rangle$ corresponds to $S(G_{\mathrm{acl}_{\Gamma}(\mathrm{gcl}(X))} \times C_{2}^{\gamma})$, where the $C_{2}^{\gamma}$ factor is generated by the $C_{2}$-elements of infinite vertex width, which have not been swallowed up by the group associated to the graph. Then to get the full algebraic closure, one only needs to throw in the finite quotients of the universal Frattini cover of this group. 
\end{rem}

\subsection{The maximal pro-$2$ quotient}

In $\mathbb{S}_{\Gamma}$, let $D$ denote the subsystem of $\mathbb{S}_{\Gamma}$ corresponding to the quotient of $\reallywidetilde{G_{\Gamma_{*}} \times C_{2}^{\overline{\kappa}}}$ by the (normal) product of the $p$- and $q$-Sylow subgroups, which we will denote $S_{pq}$. The group $G(D)$ is a free pro-$2$ group of infinite rank. The subsystem $D$ can be viewed as a many-sorted definable subset of $\mathbb{S}_{\Gamma}$. When $S$ is a subsystem of $\mathbb{S}_{\Gamma}$, we write $D(S)$ for the elements of $D$ which are in $S$, which may be alternatively described as the inverse system corresponding to $S_{pq}N$, where $N$ is the closed normal subgroup corresponding to $S$. 

It will turn out that the analysis of the model-theoretic properties of $T_{\Gamma}$ will essentially reduce to the analysis of the model-theoretic properties of the induced structure on $D(\mathbb{S}_{\Gamma})$, which may be explicitly described relative to $\Gamma$.  Before we start the analysis, we will point out the useful fact that passing to $D$ commutes with $\wedge$ and $\vee$ for subsystems of $\mathbb{S}_{\Gamma}$:

\begin{lem} \label{lem: commutes}
Suppose $S_{1}$ and $S_{2}$ are subsystems of $\mathbb{S}_{\Gamma}$. Then we have 
$$
D(S_{1} \vee S_{2}) = D(S_{1}) \vee D(S_{2}) \text{ and } D(S_{1} \wedge S_{2}) = D(S_{1}) \wedge D(S_{2}). 
$$
\end{lem}

\begin{proof}
    Let $N_{1},N_{2} \unlhd G(\mathbb{S}_{\Gamma})$ be the closed normal subgroups of $G(\mathbb{S}_{\Gamma})$ corresponding to the subsystems $S_{1}$ and $S_{2}$, respectively. By duality, we have to show $(S_{pq}N_{1})(S_{pq}N_{2}) = S_{pq}N_{1}N_{2}$, which is clear, and we have to show $S_{pq}N_{1} \cap S_{pq}N_{2} = S_{pq}(N_{1} \cap N_{2})$, which was established in Corollary \ref{cor: nice intersections}.
\end{proof}

Let $L_{\mathrm{H}}$ be $L_{\mathrm{IS}}$ together with a new unary predicate $H$. We will view $D$ as an $L_{\mathrm{H}}$-structure by giving it the $L_{\mathrm{IS}}$-structure it inherits as a subsystem of $\mathbb{S}_{\Gamma}$ and interpreting $H(D)$ as the collection of elements $\alpha \in X_{2}(D)$ with vertex width $1$ such that $\alpha$ corresponds to the identity element of the two element group $[\alpha]$ (we are just making a canonical choice of element from each $\sim$ class of elements of vertex width $1$). By Lemma \ref{lem: definability of vertex width}, this is a reduct of the induced structure on $D$. 

We will show that, viewed as an $L_{\mathrm{H}}$-structure, $D$ is an $H$-structure in the sense of Berenstein and Vassiliev \cite{berenstein2016geometric}:

\begin{defn} \label{defn: H-structure}
 Suppose $T$ is a supersimple theory and $\mathcal{Q}(x)$ is a partial $1$-type over $\emptyset$. A structure $(M,H(M))$ consisting of a model $M \models T$ together with the interpretation $H(M)$ of a new unary predicate $H$ is called an $H$\emph{-structure} if it satisfies the following conditions:
 \begin{enumerate}
     \item For all $a \in H(M)$, $a \models \mathcal{Q}$. 
     \item For all $n$, if $a_{0}, \ldots, a_{n-1} \in H(M)$ is a sequence of distinct elements, then the $a_{i}$ are non-forking independent over $\emptyset$ in $T$. 
     \item (Density/coheir property) If $A \subseteq M$ is finite and $q  \in S_{1}(A)$ is type containing $\mathcal{Q}$ which does not fork over $\emptyset$, then $q$ has a realization in $H(M)$. 
     \item (Co-density/extension property) If $A \subseteq M$ is finite and $q \in S_{1}(A)$, then there is some $a \in M$ with $a \models q$ and $a \ind^{f}_{A} H(M)$.
 \end{enumerate}
 We write $T_{\mathrm{H}}$ for the $L_{\mathrm{H}}$-theory of $(M,H(M))$ when $(M,H(M))$ is an $H$-structure. The theory $T_{H}$ does not depend on the choice of $(M,H(M))$ (though it does depend on the choice of $\mathcal{Q}$) \cite[Corollary 2.6]{berenstein2017supersimple}. 
\end{defn}

We will show that the subsystem $D$, viewed as an $L_{\mathrm{H}}$-structure, satisfies the conditions of Definition \ref{defn: H-structure}.  But first, we will need a description of forking independence in $\mathrm{Th}_{L_{\mathrm{IS}}}(D)$.  Fortunately, this has been essentially characterized already in Chatzidakis's detailed analysis of the model theory of inverse systems of profinite groups with the embedding property:

\begin{fact} \label{fact: zoe forking char} \cite[Proposition 4.1]{chatzidakis1998model}
    Suppose $G$ is a superprojective profinite group and $T = \mathrm{Th}_{L_{\mathrm{IS}}}(S(G))$. Then in any model of $T$, we have, for subsystems $A \subseteq B$, 
    $$
    \alpha \ind^{f}_{A} B \iff \alpha \vee B \in \mathrm{acl}(A) \iff  \alpha \vee B \in \mathrm{acl}(\alpha \vee A). 
    $$
\end{fact}

\begin{lem} \label{lem: acl description for free}
    Suppose $G$ is an infinite rank free pro-$2$ group and $A \subseteq S(G)$ is a finite subset. Then $\mathrm{acl}(A) = \langle A\rangle$. 
\end{lem}

\begin{proof}
    We will assume $A = \langle A \rangle$, i.e. $A$ is a finite subsystem of $S(G)$, and then we will show that $\mathrm{acl}(A) = A$. Choose any $\beta \in S(G) \setminus A$ and let $B = \langle A, \beta \rangle$, which is also a finite subsystem of $S(G)$.  Fix any $n \geq 2$, and define $H_{n} = G(B) \times_{G(A)}G(B) \times_{G(A)} \ldots \times_{G(A)} G(B)$ to be the $n$-fold fiber product of $G(B)$ with itself, over $G(A)$, taken with respect to the epimorphism $G(B) \to G(A)$ dual to the inclusion $A \subseteq B$. Let $\pi : H_{n} \to G(A)$ be the obvious map to $A$. Note that $H_{n}$ is also a finite pro-$2$ group and thus an image of $G$. Then, by the superprojectivity of $G$, we obtain an epimorphism $\varphi: G \to H_{n}$ making the diagram commute,
     $$
\xymatrix{ & G\ar@{-->>}[dl]^{\varphi}  \ar@{->>}[d] \\
H_{n} \ar@{->>}[r]^{\pi} & G(A) }
$$
where the vertical arrow is the epimorphism $G \to G(A)$ dual to the inclusion $A \subseteq S(G)$. Let $N$ be the kernel of $\varphi$. Then the quotient map $G \to G(A)$ factors into $G \to G/N \to G(A)$. If we let $C$ the subsystem of $S(G)$ corresponding to the quotient $G/N$, then we see $A \subseteq C$ and $G(C) \cong H_{n}$. Fix an isomorphism $G(C) \cong H_{n}$ and, for each $i = 1, \ldots, n$, let $B_{i}$ be the subsystem of $C$ containing $A$ dual to the quotient $G(C) \to G(B)$ obtained by projecting $G(C)$ to the $i$th coordinate. Then we have $B_{i} \cong_{A} B$ for all $i$, and we may write $B_{i} = \langle A, \beta_{i} \rangle$ for some $\beta_{i} \in B_{i}$. By \cite[Theorem 2.4]{chatzidakis1998model}, we obtain $\beta_{i} \equiv_{A} \beta$ for all $i$ and, since $B_{i} \vee B_{j} = A$ for all $i \neq j$, we have $\beta_{i} \neq \beta_{j}$. Since $n$ was taken to be arbitrary, we get $\beta \not\in \mathrm{acl}(A)$. 
\end{proof}

\begin{cor} \label{cor: free forking char}
    If $G$ is a free pro-$2$ group of infinite rank, then for subsystems $A \subseteq B \subseteq S(G)$, 
    $$
    \alpha \ind^{f}_{A} B \iff \alpha \vee B \in A \iff  (\alpha \vee B) \geq (\alpha \vee A). 
    $$
\end{cor}

\begin{proof}
    Immediate by Fact \ref{fact: zoe forking char} and Lemma \ref{lem: acl description for free}, using that free pro-$2$ groups are superprojective. 
\end{proof}

From here on out, we will write $\ind^{D}$ to denote non-forking independence in $D$, viewed as an $L_{\mathrm{IS}}$-structure. 

\begin{fact} \label{fact: weak EI} \cite[Theorem 4.6]{chatzidakis2019amalgamation}
If $G$ is a superprojective profinite group (so, in particular, if $G$ is a free or free pro-$p$ profinite group), then $\mathrm{Th}(S(G))$ weakly eliminates imaginaries. 
\end{fact}

We use $\ind^{D}$ to denote non-forking independence in $\mathrm{Th}_{L_{\mathrm{IS}}}(D)$ rather than $\mathrm{Th}_{L_{H}}(D)$, so $\ind^{D}$ has the description given in Corollary \ref{cor: free forking char}. The point of referring to this relation as non-forking independence is just to remind the reader that it has the familiar properties, such as symmetry and transitivity, which we will use freely. 

\begin{fact} \label{fact: totally categorical} \cite[Theorem 7.1]{chatzidakis1998model}
    If $G$ is a free pro-$\ell$ group (of any rank), then $\mathrm{Th}_{L_{\mathrm{IS}}}(S(G))$ is totally categorical, and in particular $\omega$-stable.
\end{fact}

\begin{lem}
    The subsystem $D$, viewed as an $L_{\mathrm{H}}$-structure is a $H$-structure, with respect to the (complete) type $\mathcal{Q}(x) = \{x \in X_{2} \wedge x \not\in X_{1} \wedge P(x,x,x)\}$, asserting that $x$ is a $C_{2}$-element and the identity element of its $\sim$-class. 
\end{lem}

\begin{proof}
    We have to check that $(D,H(D))$ satisfies conditions (1)-(4) of Definition \ref{defn: H-structure}. Condition (1) is satisfied by the choice of $H(D)$.  It is also clear from Lemma \ref{lem: dictionary} that (2) is satisfied, since, when viewed as elements of the continuous dual $(\mathbb{F}_{2}^{V})^{*}$, any $n$ $\sim$-inequivalent $C_{2}$-elements of vertex width $1$ are linearly independent. Thus if $\{\alpha_{i} : i < n\}$ is a set of distinct elements from $H(D)$, then we have $\alpha_{i} \not\in \langle \alpha_{<i} \rangle$ for all $i < n$ and thus, by Corollary \ref{cor: free forking char}, $\alpha_{i} \ind^{D} \alpha_{<i}$, showing (2). For (3), note that if $A$ is a finite subsystem of $D$ and $q \in S_{1}(A)$ is a non-forking extension of $\mathcal{Q}$, then $q$ can be axiomatized by 
    $$
    q(x) = \{x \in X_{2} \wedge x \not\in X_{1} \wedge P(x,x,x) \wedge x \neq \beta : \beta \in A \},
    $$
    and this will be realized by any element of $H(D)$ outside of $H(D) \cap A$. 

    The condition (4) is more involved. Let $A$ be a finite subset of $D$ and $q \in S_{1}(A)$.  We may assume $A = \langle A \rangle$. Let $B = A \vee \langle H(D) \rangle$.
    
    \textbf{Claim:}  There is a subsystem $D'$ with $A \subseteq D' \subseteq D$ such that $G(D')$ is a free pro-$2$ group of infinite rank and $D' \vee \langle H(D) \rangle = B$. 
    
    \emph{Proof of Claim:} Let $F$ be the subsystem of $D$ corresponding to $G(D)/\Phi(G(D))$. Then $\langle H(D) \rangle \subseteq F$ and $G(F)$ is an elementary abelian $2$-group (isomorphic to $C_{2}^{V} \times C_{2}^{\overline{\kappa}}$).  Let $C = A \vee F$. So $B \subseteq C$ and $G(B)$ and $G(C)$ are both finite elementary abelian $2$-groups. Moreover, $G(A) \to G(C)$ is a Frattini cover. 

    Let $(\alpha_{i})_{i \in I_{0}}$ be a maximal sequence of independent $C_{2}$-elements in $B$ from $H(D)$. Then $B = \langle (\alpha_{i})_{i \in I_{0}} \rangle$ and we may extend this sequence to a maximal sequence of independent $C_{2}$-elements $(\alpha_{i})_{i \in I_{0} \cup I_{1}}$ in $H(D)$ (which is, necessarily, an enumeration of all of $H(D)$) and a maximal sequence of independent $C_{2}$-elements $(\alpha_{i})_{i \in I_{0} \cup I_{2}}$ in $C$. Then by our assumption that $A \ind^{D}_{B} H(D)$, we have that $(\alpha_{i})_{i \in I_{0} \cup I_{1} \cup I_{2}}$ is an independent sequence of $C_{2}$-elements in $\langle C, H(D) \rangle$. Finally, extend the sequence to $(\alpha_{i})_{i \in I_{0} \cup I_{1} \cup I_{2} \cup I_{3}}$, a maximal independent sequence of $C_{2}$-elements in $F$. Thus $F$ is generated by this sequence and $G(F) \cong C_{2}^{I_{0} \cup I_{1} \cup I_{2} \cup I_{3}}$ via the isomorphism dual to the map $\ker(\pi_{i}) \mapsto \alpha_{i}$ for all $i$. Note that we have 
    $$
    (\alpha_{i})_{i \in I_{0} \cup I_{2}} \ind^{D}_{\emptyset} (\alpha_{i})_{i \in I_{1} \cup I_{3}},
    $$
    and hence
    $$
    C \ind^{D}_{\emptyset} (\alpha_{i})_{i \in I_{1} \cup I_{3}}.
    $$
    Moreover, by the choice of $C = A \vee F$, we have 
    $$
    A \ind^{D}_{C} F
    $$
    which entails $A \ind^{D} \langle (\alpha_{i})_{i \in I_{1} \cup I_{3}} \rangle$, by transitivity. This shows $G(\langle A, (\alpha_{i})_{i \in I_{1} \cup I_{3}}\rangle) \cong G(A) \times C_{2}^{I_{1} \cup I_{3}}$. Thus we may identify $G(\langle A, (\alpha_{i})_{i \in I_{1} \cup I_{3}}\rangle)$ with $G(A) \times C_{2}^{I_{1} \cup I_{3}}$, $G(F)$ with $C_{2}^{I_{0} \cup I_{1} \cup I_{2} \cup I_{3}}$, $G(C)$ with $C_{2}^{I_{0} \cup I_{2}}$ and the induced map $G(\langle A, (\alpha_{i})_{i \in I_{1} \cup I_{3}} \rangle) \to G(F)$ with the product of the induced map $G(A) \to G(C)$ and the identity map $C_{2}^{I_{1}} \times C_{2}^{I_{3}} \to C_{2}^{I_{1}} \times C_{2}^{I_{3}}$. 

    Let $x_{i} \in G(F)$ be the element which is $-1$ in the $i$ coordinate and $1$ in every other coordinate. Then the sequence $\langle x_{i} : i \in I_{0} \cup I_{2} \rangle$ is a minimal set of generators for $G(C)$. Since the map $G(A) \to G(C)$ is Frattini, the sequence $\langle x_{i} : i \in I_{0} \cup I_{2} \rangle$ lifts to a set $\langle x'_{i} : i \in I_{0} \cup I_{2} \rangle$ of elements of $G(A)$ which form a minimal set of generators. Then since the map $G(D) \to G(\langle A,F \rangle)$ is Frattini, the set $\{x_{i} : i \in I_{1} \cup I_{3} \} \cup \{x'_{i} : i \in I_{0} \cup I_{2}\}$ lifts to $\{y_{i} : i \in I_{0} \cup I_{1} \cup I_{2} \cup I_{3}\}$, a set of minimal generators in $G(D)$. Then any subset of the $y_{i}$ generate a free pro-$2$ group.  Define a map $G(D) \to \langle y_{i} : i \in I_{0} \cup I_{2} \cup I_{3} \rangle$ by sending each $y_{i} \mapsto y_{i}$ for $i \in I_{0} \cup I_{2} \cup I_{3}$ and mapping $y_{i} \mapsto 1$ for all $i \in I_{1}$. Let $D'$ be the subsystem of $D$ dual to this quotient. 

    Since the $y_{i}$ lift the $x_{i}$, we have that $D' \vee F = \langle \alpha_{i} : i \in I_{0} \cup I_{2} \cup I_{3} \rangle$, and thus $D' \vee \langle H(D) \rangle = \langle \alpha_{i} : i \in I_{0} \cup I_{2} \rangle = B$. This proves the claim. $\square$

    Now the given $q$ must be realized since $\mathrm{Th}(D')$ is totally categorical, by Fact \ref{fact: totally categorical} (as $G(D')$ is a free pro-$2$ group of infinite rank), so every model of its theory is $\aleph_{0}$-saturated. This proves (4). 
\end{proof}

\begin{defn}
    Suppose $(M,H(M))$ is an $H$-structure. Then a subset $A \subseteq M$ is called \emph{$H$-independent} if $A \ind^{D}_{H(A)} H(M)$. 
\end{defn}

We need some general facts about $H$-structures:

\begin{fact} \label{fact: H structure facts}
    Suppose $(M,H(M))$ is an $H$-structure where $M$ is an $L$-structure and $(M,H(M))$ is an expansion of $M$ to the language $L_{H} = L \cup \{H\}$. 
    \begin{enumerate}
        \item \cite[Proposition 4.5]{berenstein2017supersimple} If $C$ is an $H$-independent subset of $M$, and $a$ is a finite tuple from $M$, then there is a unique smallest finite subset $H_{0} \subseteq H(M)$ such that $a \ind^{f}_{CH_{0}} H(M)$. This set will be referred to as the $H$\emph{-basis} of $a$ over $C$, denoted $\mathrm{HB}(a/C)$. When $A$ is a (possibly infinite) set, we write $\mathrm{HB}(A/C)$ for the union of the $H$-bases of all finite tuples from $A$. 
        \item \cite[Proposition 2.5]{berenstein2017supersimple} If $a$ and $a'$ are both $H$-independent tuples, then 
        $$
        \mathrm{tp}_{L_{H}}(a) = \mathrm{tp}_{L_{H}}(a') \iff \mathrm{tp}_{L}(a,H(a)) = \mathrm{tp}_{L}(a',H(a')). 
        $$
        \item \cite[Corollary 4.15]{berenstein2017supersimple} If $A \subseteq M$, then the algebraic closure of $A$ in the $H$-structure $(M,H(M))$, denoted $\mathrm{acl}_{H}(A)$, satisfies 
        $$
        \mathrm{acl}_{H}(A) = \mathrm{acl}(A,\mathrm{HB}(A))
        $$
        where $\mathrm{acl}$ refers to algebraic closure in the underlying theory $T$, without the $H$ predicate. 
        \item \cite[Theorem 5.3, Proposition 5.9]{berenstein2017supersimple} If the underlying theory $T$ is superstable, then $T_{\mathrm{H}}$ is also superstable and non-forking independence in $T_{\mathrm{H}}$, denoted $\ind^{\mathrm{H}}$, is characterized by 
        $$
        a \ind^{\mathrm{H}}_{C} B \iff \mathrm{HB}(a/B) = \mathrm{HB}(a/C) \text{ and } a \ind_{CH} B,
        $$
        when $C = \mathrm{acl}_{H}(C) \subseteq B = \mathrm{acl}_{H}(B)$, where $\ind$ denotes non-forking independence in $T$. 
    \end{enumerate}
\end{fact}

We now introduce a third language $L_{\mathrm{R}}$ which extends $L_{\mathrm{H}}$ with the addition of a new binary relation symbol $R$. Then $D$ will be viewed as an $L_{\mathrm{R}}$-structure by interpreting $R$ so that $R^{D}$ consists of the pairs $(h,h') \in H(D)^{2}$ such that $h$ and $h'$ are $C_{2}$-elements coming from vertices in $\Gamma$ which are connected by an edge. This makes sense since $H(D)$ is, by definition, the set of $C_{2}$-elements of $\mathbb{S}_{\Gamma}$ which have vertex width $1$ and thus each lie above a unique $D_{p}$-element associated to a vertex of the graph. Note that the graph $(H(D), R^{D})$ may be naturally identified with $\Gamma$. We will eventually show that the $L_{R}$-structure on $D$ gives the full induced structure on $D$ coming from $\mathbb{S}_{\Gamma}$ and that certain model-theoretic properties of $\Gamma$ are transferred to the $L_{\mathrm{R}}$-structure $D$. This will be a key step towards concluding that model-theoretic properties are transferred from $\Gamma$ to $\mathbb{S}_{\Gamma}$. 

\begin{lem} \label{lem: type char in D}
If $a$ and $a'$ are $H$-independent tuples from $D$, then the following are equivalent:
    \begin{enumerate}
            \item $a \equiv^{L_{R}} a'$.
            \item $a \equiv^{L_{H}} a'$ and  $H(a) \equiv^{\Gamma} H(a')$.
            \item $(a,H(a)) \equiv^{L_{\mathrm{IS}}} (a',H(a'))$ and $H(a) \equiv^{\Gamma} H(a')$.
    \end{enumerate}
\end{lem}

\begin{proof}
    (1)$\implies$(2)$\implies$(3) is immediate, so we prove (3)$\implies$(1). The proof will follow \cite[Proposition 2.5]{berenstein2017supersimple}. Write $a = (a_{0},a_{1},h)$ where $a_{0}$ is independent over $H$, $h = H(a)$, and $a_{1} \in \langle a_{0},h \rangle$. Similarly, write $a' = (a_{0}',a_{1}',h')$ in a corresponding manner. It suffices to show that, given any $b \in D$, there are $h_{1},h_{1}'$, and $b'$ such that $ah_{1}b$ and $a'h_{1}'b'$ are each $H$-independent, $a_{0}a_{1}hh_{1}b \equiv^{L_{\mathrm{IS}}} a_{0}'a_{1}'h'h_{1}'b'$, $b \in H$ if and only if $b' \in H$, and $H(ah_{1}b) \equiv^{\Gamma} H(a'h_{1}'b')$.

    \textbf{Case 1}: $b \in H$ and $b \in \langle a \rangle$. 

    In this case, since $a$ is $H$-independent, we must have $b \in H(a)$ so there is, by assumption, a corresponding $b' \in H(a')$. We can, in this case, take $h_{1}$ and $h'_{1}$ to be empty. It is clear that all the conditions are satisfied. 

    \textbf{Case 2}:  $b \in H$, $b \not\in \langle a \rangle$. 

    Pick $b' \in H$ with $H(a)b \equiv^{\Gamma} H(a')b'$. Note that we cannot have $b' \in \langle a' \rangle$ since, as in Case 1, this would entail that $b' \in h'$ and thus $b \in h$, contrary to our assumption. Under our assumptions, we have $G(\langle a,b \rangle) \cong G(\langle a \rangle) \times C_{2} \cong G(\langle a' \rangle) \times C_{2} \cong G(\langle a',b' \rangle)$ so the map $a \mapsto a'$ extends to an $L_{\mathrm{IS}}$-isomorphism $\langle a,b \rangle \to \langle a',b' \rangle$ sending $b$ to $b'$ and inducing a partial $\Gamma$-elementary map on the $H$-elements. In this case too we can take $h_{1} = h_{1}' = \emptyset$ and all the conditions are satisfied. 

    \textbf{Case 3}:  $b \not\in H$. 

    We can find a finite tuple $h_{1}$ such that $b \ind^{D}_{ah_{1}} H$. We may assume $h_{1}$ is disjoint from $a$, which entails that $ah_{1}$ is also $H$-independent. Then, by Case 2, we can find some $h_{1}'$ from $H$ with $a_{0}a_{1}hh_{1} \equiv^{L_{\mathrm{IS}}} a_{0}'a_{1}'h'h_{1}'$ and $H(ah_{1}) \equiv^{\Gamma} H(a'h_{1}')$. Now, by the co-density/extension property for $H$-structures, we can find some $b'$ with $a_{0}a_{1}hh_{1}b \equiv^{L_{\mathrm{IS}}} a_{0}'a_{1}'h'h_{1}'b'$ and $b' \ind^{D}_{a'h_{1}'} H$. Then $a'h_{1}'b'$ is $H$-independent, so we are done. 
\end{proof}

\begin{cor} \label{cor : from D to S}
    Suppose $S_{0} = \mathrm{acl}(S_{0})$ is a small subsystem of $\mathbb{S}_{\Gamma}$. 
    \begin{enumerate}
        \item Suppose $D_{1}$ and $D_{2}$ are subsystems of $D$ containing $D(S_{0})$ with $\mathrm{acl}_{H}(D_{i}) = D_{i}$ for $i = 1,2$. If $D_{1} \equiv_{D(S_{0})}^{L_{R}} D_{2}$ in $D$, then $D_{1} \equiv_{S_{0}} D_{2}$ in $\mathbb{S}_{\Gamma}$. 
        \item If $S_{0} \subseteq S_{1} = \mathrm{gcl}(S_{1}) \subseteq \mathbb{S}_{\Gamma}$ and $D'$ is a subsystem of $D$ with $D' \equiv^{L_{R}}_{D(S_{0})} D(S_{1})$, then there is some $S' \subseteq \mathbb{S}_{\Gamma}$ such that $S_{1} \equiv_{S_{0}} S'$ in $\mathbb{S}_{\Gamma}$ and $D(S') = D'$. 
    \end{enumerate}
\end{cor}

\begin{proof}
    (1) First, note that, for $i = 1,2$, $D(S_{0}) = D_{i} \vee S_{0}$, since the containments $D(S_{0}) \subseteq D_{i}$ and $D(S_{0}) \subseteq S_{0}$ are obvious and $D_{i} \vee S_{0}$ is a subsystem corresponding to a pro-$2$ quotient of $G(S_{0})$ so is, by definition of $D$, contained in $D(S_{0})$. Then, for $i = 1,2$, we have 
    $$
    G(\langle D_{i},S_{0} \rangle) \cong G(D_{i}) \times_{G(D(S_{0}))} G(S_{0}).
    $$
    It follows, by duality, that the identity map on $S_{0}$ and the isomorphism $D_{1} \to D_{2}$ over $D(S_{0})$ extend uniquely to an isomorphism $\langle D_{1},S_{0} \rangle \to \langle D_{2}, S_{0} \rangle$ fixing $S_{0}$. By Lemma \ref{lem: extends}, this map is partial elementary. 

    (2) Let $\sigma : D(S_{1}) \to D'$ be an isomorphism of subsystems which is partial elementary in $D$ (viewed as an $L_{R}$-structure).  By (1) and the strong homogeneity of $\mathbb{S}_{\Gamma}$, $\sigma$ extends to some $\tilde{\sigma} \in \mathrm{Aut}(\mathbb{S}_{\Gamma}/S_{0})$. We put $S' = \tilde{\sigma}(S_{1}).$  
\end{proof}

\begin{lem} \label{lem: H control}
Suppose $D_{*}  = \mathrm{acl}_{H}(D_{*}) \subseteq D$ is a small subsystem and $D_{0},D_{1}$ are subsystems of $D$ containing $D_{*}$ with $D_{0} \ind^{D}_{D_{*}H(D_{0})} H$ and $D_{1} \ind^{D}_{D_{*}H(D_{1})} H D_{0}$. Then $H(\langle D_{0},D_{1} \rangle) \subseteq H(D_{0})H(D_{1})$. 
\end{lem}

\begin{proof}
Note that, since $D_{*} = \mathrm{acl}_{H}(D_{*})$, it is, in particular $H$-independent, i.e. $D_{*} \ind^{D}_{H(D_{*})} H$, and thus, by base monotonicity, $D_{*} \ind^{D}_{H(D_{0})} H$. Then, the assumption that $D_{0} \ind^{D}_{D_{*}H(D_{0})} H$, together with transitivity, entails $D_{0} \ind^{D}_{H(D_{0})} H$. By base monotonicity again, we get $D_{0} \ind^{D}_{H(D_{0})H(D_{1})} H$. 

Next, since $D_{1} \ind^{D}_{D_{*}D_{0}H(D_{1})} H$, we have $D_{1} \ind^{D}_{D_{0}H(D_{1})} H$ so, by transitivity, we get $D_{0}D_{1} \ind^{D}_{H(D_{0})H(D_{1})} H$. This entails that 
$$
H \cap \langle D_{0}, D_{1} \rangle = H(\langle D_{0},D_{1}\rangle) \subseteq \langle H(D_{0}),H(D_{1}) \rangle.
$$
Since $H$ consists of independent elements, $H \cap \langle H(D_{0}),H(D_{1}) \rangle = H(D_{0})H(D_{1})$, so we get $H(\langle D_{0},D_{1}\rangle) \subseteq H(D_{0})H(D_{1})$ as desired. 
\end{proof}

\begin{lem} \label{lem: just adding the graph}
    Let $S \subseteq \mathbb{S}_{\Gamma}$ be a small subsystem. 
\begin{enumerate}
\item Suppose $S = \mathrm{gcl}(S)$. Then, identifying $\Gamma$ and $H$, we have $\Gamma(S) = \mathrm{HB}(D(S))$. More generally, if $C \subseteq D$ is a small $H$-independent subset, then
$$
\mathrm{HB}(D(S)/C) = \Gamma(S) - \Gamma(C).
$$
\item The subsystem $S$ is graph closed if and only if $S$ satisfies the following:
\begin{enumerate}
\item If $\alpha \geq \beta$, $\alpha \in D(S)$ is a $C_{2}$-element and $\beta$ is a $D_{p}$-element, then $\beta \in S$. 
\item $\Gamma(S) \subseteq \Gamma$ is an induced subgraph
\item $D(S)$ is $\mathrm{acl}_{H}$-closed in $D$.
\end{enumerate} 
\item Suppose $\mathrm{acl}_{H}(D(S)) = D(S)$. Then $D(S) = D(\mathrm{gcl}(S))$ and thus $H(S) = H(\mathrm{gcl}(S))$. 
\end{enumerate}
\end{lem}

\begin{proof}
(1) By Lemma \ref{lem: transfer} and Lemma \ref{lem: plus intersection}, we have $S \vee (\mathbb{S}_{\Gamma})_{-} = S(G_{\Gamma(S)})$ and, thus, $D(S) \ind^{D}_{H(D(S))} H$, since unraveling definitions yields $H(D(S(G_{\Gamma(S)}))) = H(D(S))$. This shows $\mathrm{HB}(D(S)) \subseteq H(D(S))$. Conversely, if $\mathrm{HB}(D(S)) \subsetneq H(D(S))$, then we can pick some $\alpha \in H(D(S)) - \mathrm{HB}(D(S))$. Since the elements of $H$ are independent, we must have $\alpha \not\in \langle \mathrm{HB}(D(S)) \rangle$ but $\alpha \in D(S) \cap H$, which contradicts $D(S) \ind^{D}_{\mathrm{HB}(D(S))} H$. This shows the desired equality. The argument for the `more generally' clause in the statement is identical. 

(2) If $S = \mathrm{gcl}(S)$, then conditions (a) and (b) are clear from the definition of graph closure and condition (c) stating $\mathrm{acl}_{H}(D(S)) = D(S)$ follows by (1) and the description of algebraic closure in $H$-structures (Fact \ref{fact: H structure facts}). For the other direction, assume $\mathrm{acl}_{H}(D(S)) = D(S)$. We first need to show that if $\alpha$ is a $C_{2}$-element of $S$ of vertex width $n$ and $\beta_{1},\ldots, \beta_{n}$ are the $n$ distinct $C_{2}$-elements of vertex width $1$ associated to the witnesses, then $\beta_{1},\ldots, \beta_{n} \in S$. We may assume $\alpha$ and each $\beta_{i}$ are the identity element of its equivalence class, hence $\alpha \in D(S)$ and $\beta_{i} \in H$ for $i = 1, \ldots, n$. Since $\alpha \in \langle \beta_{1},\ldots, \beta_{n} \rangle$ and $\alpha \ind^{D}_{H(D(S))} H$, we must have $\beta_{1}, \ldots, \beta_{n} \in H(D(S))$. Then, by condition (a), it follows that if $\alpha \in D(S)$ is a $C_{2}$-element of finite vertex width, then the $D_{p}$-elements that witness this are also in $S$. Finally, since $\Gamma(S)$ is an induced subgraph, we have that all conditions of being graph closed are satisfied so $S = \mathrm{gcl}(S)$. 

(3) The inclusion $H(S) \subseteq H(\mathrm{gcl}(S))$ is immediate from $S \subseteq \mathrm{gcl}(S)$, so we prove the reverse inclusion. Under the identification of $\Gamma(\mathbb{S}_{\Gamma})$ and $H$, let $\Gamma' = (H(D(S)),R')$ denote the induced subgraph of $\Gamma(\mathbb{S}_{\Gamma})$ with vertex set $H(D(S))$. 

It is immediate from the definition of graph closure that $\langle S,S(G_{\Gamma'}) \rangle = S \wedge S(G_{\Gamma'})$ is graph closed. Moreover, $D(S(G_{\Gamma'})) = \langle H(D(S)) \rangle$.  Thus, Lemma \ref{lem: commutes} gives 
$$
D(S \wedge S(G_{\Gamma'})) = D(S) \wedge \langle H(D(S)) \rangle = D(S),
$$
since $H(D(S)) \subseteq D(S)$. Since $S \wedge S(G_{\Gamma'})$ is graph closed and contains $S$, we have $\mathrm{gcl}(S) \subseteq S \wedge S(G_{\Gamma'})$, whence $D(S) \subseteq D(\mathrm{gcl}(S)) \subseteq D(S \wedge S(G_{\Gamma'})) = D(S)$. Thus $D(\mathrm{gcl}(S)) = D(S)$ and $H(\mathrm{gcl}(S)) = H(D(\mathrm{gcl}(S))) = H(D(S)) = H(S)$. \end{proof}

\begin{cor} \label{cor: easy H}
    Suppose $\mathrm{acl}(S_{i}) = S_{i} \subseteq \mathbb{S}_{\Gamma}$ are small subsystems for $i = 0,1,2$ with $S_{1} \cap S_{2} = S_{0}$. Then $D(S_{1}) \ind^{H}_{D(S_{0})} D(S_{2})$ if and only if $D(S_{1}) \ind^{D}_{D(S_{0})H} D(S_{2})$.
\end{cor}

\begin{proof}
    If $D(S_{1}) \ind^{H}_{D(S_{0})} D(S_{2})$, then $D(S_{1}) \ind^{D}_{D(S_{0})H} D(S_{2})$ by Fact \ref{fact: H structure facts}(4). On the other hand, if $D(S_{1}) \ind^{D}_{D(S_{0})H} D(S_{2})$ then, again by Fact \ref{fact: H structure facts}(4), in order to show $D(S_{1}) \ind^{H}_{D(S_{0})} D(S_{2})$, we just need to show that $\mathrm{HB}(D(S_{1})/D(S_{0})) = \mathrm{HB}(D(S_{1})/D(S_{2}))$. By our hypothesis, $S_{1} \cap S_{2} = S_{0}$ and hence $\Gamma(S_{1}) \cap \Gamma(S_{2}) = \Gamma(S_{0})$. Thus, by Lemma \ref{lem: just adding the graph}(1), we have
    $$
    \mathrm{HB}(D(S_{1})/D(S_{0})) = \Gamma(S_{1}) - \Gamma(S_{0}) = \Gamma(S_{1}) - \Gamma(S_{2}) = \mathrm{HB}(D(S_{1})/D(S_{2})),
    $$
    which completes the proof. 
\end{proof}

\begin{lem} \label{lem: extends}
    Suppose $S_{0},S_{1},$ and $S_{2}$ are small subsystems of $\mathbb{S}_{\Gamma}$ such that $S_{0} = \mathrm{gcl}(S_{0}) \subseteq S_{1},S_{2}$ and $D(S_{i}) = \mathrm{acl}_{H}(D(S_{i}))$ for $i = 0,1,2$. Suppose we are given an isomorphism $\Psi: S_{1} \to S_{2}$ over $S_{0}$ which restricts to an $L_{R}$-partial elementary map $D(S_{1}) \to D(S_{2})$ over $D(S_{0})$. Then $\Psi$ extends to an isomorphism $\Xi: \mathrm{gcl}(S_{1}) \to \mathrm{gcl}(S_{2})$.  In particular, $\Psi$ is partial elementary. 
\end{lem}

\begin{proof}
    Under the identification of $H$ and $\Gamma(\mathbb{S}_{\Gamma})$, let $\Gamma_{i}$ be the induced subgraph of $\Gamma(\mathbb{S}_{\Gamma})$ with vertices $H(D(S_{i}))$ for $i = 0,1,2$. Then the map $\Psi$ induces an isomorphism $f: \Gamma_{1} \to \Gamma_{2}$ over $\Gamma_{0}$ (which is partial elementary with respect to $\mathrm{Th}(\Gamma)$). 

    We first argue that if $\beta$ is a $D_{p}$-element of $\mathbb{S}_{\Gamma}$ corresponding to a vertex of $\Gamma_{1}$, then we may extend $\Psi$ to some $\Psi': \langle S_{1},\beta \rangle \to \langle S_{2},\beta'\rangle$ for some $D_{p}$-element $\beta' \in \mathbb{S}_{\Gamma}$ corresponding to the vertex $f(v)$. There is nothing to show if $\beta \in S_{1}$, so we may assume $\beta \not\in S_{1}$, so $[\beta] \vee S_{1} = [\alpha_{v}]$ where $\alpha_{v}$ is a $C_{2}$-element in $H(D(S_{1}))$. Then let $\alpha_{f(v)} = \Psi(\alpha_{v}) \in H(D(S_{2}))$ and let $\beta'$ be a $D_{p}$-element corresponding to the vertex $f(v)$ (so with $\beta' \leq \alpha_{f(v)}$. Now we must have $\beta' \not\in S_{2}$ so we have $[\beta'] \vee S_{2} = [\alpha_{f(v)}]$. Since we have 
    \begin{eqnarray*}
        G([\beta'] \vee S_{2}) \cong G([\beta] \vee S_{1}) \cong D_{p} \times_{C_{2}} G(S_{1})
    \end{eqnarray*}
    we may extend $\Psi$ to an isomorphism $\Psi' : \langle S_{1},\beta \rangle \to \langle S_{2},\beta' \rangle$, an isomorphism over $S_{0}$. 
    
    Applying this inductively to all vertices of $\Gamma_{1}$, we may assume that we have extended $\Psi$ to some $\Psi' : S_{1}' \to S_{2}'$ over $S_{0}$ such that if $\beta \in \mathbb{S}_{\Gamma}$ is a $D_{p}$-element corresponding to a vertex of $\Gamma_{1}$, then $\beta \in S_{1}'$ (which entails the corresponding property for $S_{2}'$). Now we need to extend to handle the $W$-elements corresponding to edges, so fix some $\gamma \in \mathbb{S}_{\Gamma}$, a $W$-element corresponding to an edge $r = (v,v')$ with $v,v'$ vertices in $\Gamma_{1}$. We want to extend $\Psi'$ to some $\Psi'' : \langle S_{1}',\gamma \rangle \to \langle S_{2}',\gamma' \rangle$ where $\gamma'$ is a $W$-element associated to the edge $r' = (f(v),f(v'))$. As above, we may assume $\gamma\not\in S_{1}'$. Let $\beta_{v},\beta_{v'}$, $\beta_{f(v)}$, and $\beta_{f(v')}$ be the $D_{p}$-elements associated to the vertices $v,v', f(v)$, and $f(v')$, respectively, which we may choose so that $\Psi'(\beta_{v}) = \beta_{f(v)}$ and $\Psi'(\beta_{v'}) = \beta_{f(v')}$. By the construction of $S_{1}'$ and $S_{2}'$, we have 
    \begin{eqnarray*}
    G(S'_{1} \vee \gamma) &\cong& [ \beta_{v} \wedge \beta_{v'}] \cong D_{p} \times D_{p} \\
    G(S'_{2} \vee \gamma') &\cong& [\beta_{f(v)} \wedge \beta_{f(v')}] \cong D_{p} \times D_{p}
    \end{eqnarray*}
    so we get 
    $$
    G(\langle S_{1}', \gamma \rangle) \cong G(S_{1}') \times_{(D_{p} \times D_{p})} W \cong G(S'_{2}) \times_{(D_{p}\times D_{p})} W \cong G(\langle S_{2}',\gamma' \rangle).
    $$
    By Lemma \ref{extendingauts}, the isomorphism $[\beta_{f(v)} \wedge \beta_{f(v')}] \to [\beta_{v} \wedge \beta_{v'}]$ induced by duality by $\Psi'$ is induced by an isomorphism $[\gamma'] \to [\gamma]$. This, together with the isomorphism $\psi': G(S_{2}') \to G(S_{1}')$ dual to $\Psi'$ combine to yield an isomorphism $G(\langle S_{2}',\gamma'\rangle) \to G(\langle S_{1}',\gamma \rangle)$ which, dualizing again, yield the desired extension $\Psi'' : \langle S_{1}',\gamma \rangle \to \langle S_{2}',\gamma' \rangle$, which is still over $S_{0}$. 

    Applying this inductively to all edges of $\Gamma_{1}$, using Lemma \ref{extendingauts}, we obtain the desired extension $\Xi : \mathrm{gcl}(S_{1}) \to \mathrm{gcl}(S_{2})$ over $S_{0}$.  The `in particular' clause of the statement follows by Lemma \ref{lem: type char}, since isomorphisms of graph closed subsystems which induce elementary maps on graphs are partial elementary. 
\end{proof}

\begin{lem} \label{lem: disjoint in D}
    Suppose $S_{i} = \mathrm{acl}(S_{i})$ for $i = 0,1,2$ and $S_{0} \subseteq S_{1},S_{2}$. If $D(S_{1}) \vee D(S_{2}) = D(S_{0})$, then $S_{1} \vee S_{2} = S_{0}$. 
\end{lem}

\begin{proof}
    We will prove the lemma via a series of claims. Fix $\alpha \in S_{1} \cap S_{2}$. Our goal is to show that $\alpha \in S_{0}$.
    
    \textbf{Claim 1}: Suppose $[\alpha] \cong G_{\Gamma_{0}}$ for some induced subgraph $\Gamma_{0} \subseteq \Gamma(\mathbb{S}_{\Gamma})$. Then $\alpha \in S_{0}$. 

    \emph{Proof of Claim 1}: Let $\beta \geq \alpha$ be an element with $[\beta] \cong C_{2}^{V_{0}}$ corresponding to the quotient of $[\alpha]$ by the normal product of the $p$- and $q$-Sylow subgroups.  Then $\beta \in D(S_{1} \vee S_{2}) = D(S_{1}) \vee D(S_{2}) = D(S_{0})$. Then $\alpha \in \mathrm{gcl}(\beta)$ so $\alpha \in S_{0}$ since $S_{0}$ is, in particular, graph closed. \qed

    From this, it will easily follow when $\alpha$ corresponds to a quotient that also has an additional $C_{2}^{\ell}$ factor:

    \textbf{Claim 2}: Suppose $[\alpha] \cong G_{\Gamma_{0}} \times C_{2}^{\ell}$ for some induced subgraph $\Gamma_{0} \subseteq \Gamma(\mathbb{S}_{\Gamma})$. Then $\alpha \in S_{0}$. 

    \emph{Proof of Claim 2}: Let $\beta, \gamma \geq \alpha$ be chosen so that $[\beta] \cong G_{\Gamma_{0}}$, $[\gamma] \cong C_{2}^{\ell}$, and $\alpha = \beta \wedge \gamma$. Then $\beta \in S_{1} \vee S_{2}$ so, by Claim 1, $\beta \in S_{0}$. Also, since $[\gamma]$ is a $2$-group, we have $\gamma \in D(S_{1} \vee S_{2}) = D(S_{0}) \subseteq S_{0}$ so we have $\gamma \in S_{0}$. Then $\alpha = \beta \wedge \gamma \in S_{0}$. \qed

    Now we handle the case of Frattini covers of the groups handled in Claim 2:
    \textbf{Claim 3}: Suppose $\beta \geq \alpha$ and $[\beta] \cong [\alpha]/\Phi([\alpha]) \cong G_{\Gamma_{0}} \times C_{2}^{\ell}$ for some finite induced subgraph $\Gamma_{0} \subseteq \Gamma$ and $\ell < \omega$, so the canonical map $[\alpha] \to [\beta]$ is a Frattini cover. Then $\alpha \in S_{0}$. 

    \emph{Proof of Claim}: Because $\Phi([\alpha])$ is nilpotent, we can decompose $\Phi([\alpha]) = \Phi_{2}([\alpha])\Phi_{pq}([\alpha])$ where $\Phi_{2}([\alpha])$ is a $2$-group and the order of $\Phi_{pq}([\alpha])$ only divisible by $p$ and $q$. Then we find $\delta$ and $\epsilon$ satisfying the following:
    \begin{enumerate}
        \item $\alpha \leq \delta, \epsilon \leq \beta$. 
        \item $\alpha = \delta \wedge \epsilon$, $\beta = \delta \vee \epsilon$. 
        \item $[\delta] \cong [\alpha] /\Phi_{2}([\alpha])$ so the canonical map $[\delta] \to [\beta]$ is $pq$-Frattini.
        \item $[\epsilon] \cong [\alpha]/\Phi_{pq}([\alpha])$ so the canonical map $[\epsilon] \to [\beta]$ is $2$-Frattini. 
    \end{enumerate}
    By Claim 2, we know $\beta \in S_{0}$ so we have $\delta \in S_{0}$ by the description of algebraic closure in Proposition \ref{prop: acl description}. 

    Choose $\epsilon' \geq \epsilon$ with $[\epsilon']$ corresponding to the quotient of $[\epsilon]$ by the normal product of the $p$- and $q$-Sylow subgroups. If we let $K_{1}$ be the kernel of the canonical map $[\epsilon] \to [\beta]$ and let $K_{2}$ denote the kernel of the canonical map $[\epsilon] \to [\epsilon']$, we have $K_{1} \cap K_{2} = 1$ since $K_{1}$ is a $2$-group and $K_{2}$ is only divisible by primes $p$ and $q$. Thus, $[\epsilon] \cong [\epsilon]/(K_{1} \cap K_{2})$ which entails $\epsilon = \beta \wedge \epsilon'$. Now since $[\epsilon']$ is a $2$-group, we have $\epsilon' \in D(S_{1} \vee S_{2}) = D(S_{0})$. Thus $\epsilon = \beta \wedge \epsilon' \in S_{0}$. Then, finally, we obtain $\alpha = \delta \wedge \epsilon \in S_{0}$, proving the claim. \qed
    
    Now we conclude by considering the case of arbitrary $\alpha$.  Let $\beta \in S_{1} \cap S_{2}$ be an element with $\alpha \leq \beta$ and $[\beta] \cong [\alpha]/\Phi([\alpha])$, so the canonical map $[\alpha] \to [\beta]$ is a Frattini cover and $\beta \in (\mathbb{S}_{\Gamma})_{+}$. Thus, by Lemma \ref{lem: bounding}, there exists $\beta' \leq \beta$ with $[\beta'] \cong G_{\Gamma_{0}} \times C_{2}^{\ell}$ for some finite induced subgraph $\Gamma_{0} \subseteq \Gamma(\mathbb{S}_{\Gamma})$ and $\ell < \omega$ and, by inspection of the proof, we see that $\beta' \in \mathrm{acl}(\beta)$ so we have $\beta' \in S_{1} \vee S_{2}$. By Fact \ref{fact2}(3), any Frattini cover of $[\beta]$ will be a quotient of a Frattini cover of $[\beta']$ so are reduced to considering the case $\beta = \beta'$, but this has been handled by Claim 3.  
\end{proof}

\subsection{Preservation results}

Define an independence relation $\ind^{*}$ on $\mathbb{S}_{\Gamma}$ as follows:  we declare $A \ind^{*}_{C} B$ if and only if 
\begin{itemize}
    \item $\mathrm{acl}(AC) \cap \mathrm{acl}(BC) = \mathrm{acl}(C)$, i.e. $A \ind^{a}_{C} B$. 
    \item $D(\mathrm{acl}(AC)) \ind^{D}_{HD(\mathrm{acl}(C))} D(\mathrm{acl}(BC))$. 
\end{itemize}
By Corollary \ref{cor: easy H}, the second bullet point could be equivalently written $D(\mathrm{acl}(AC)) \ind^{H}_{D(\mathrm{acl}(C))} D(\mathrm{acl}(BC))$. With this notation, we may state more simply the following key proposition:

\begin{prop}[Relative Stationarity] \label{prop: relative stationarity}
Suppose we are given small subsystems $S_{i} = \mathrm{acl}(S_{i}) \subseteq \mathbb{S}_{\Gamma}$ for $i = 0,1,2$ and $S_{0} \subseteq S_{1} \cap S_{2}$. Suppose, further, that we are given some $S_{1}'$ satisfying the following conditions:
\begin{enumerate}
    \item $S_{1} \ind^{*}_{S_{0}} S_{2}$ and $S_{1}' \ind^{*}_{S_{0}} S_{2}$. 
    \item There is an $L_{\mathrm{IS}}$-isomorphism $S_{1} \to S_{1}'$ over $S_{0}$ which induces a partial elementary map $\Gamma(S_{1}) \to \Gamma(S_{1}')$ over $\Gamma(S_{2})$. 
\end{enumerate}
Then $S_{1} \equiv_{S_{2}} S_{1}'$.
\end{prop}

\begin{proof}
    Since $S_{1} \vee S_{2} = S_{0}$, $S_{1}' \vee S_{2} = S_{0}$ and $S_{1}$ is isomorphic to $S_{1}'$ over $S_{0}$, we have that the given isomorphism $S_{1} \to S_{1}'$ over $S_{0}$ extends to an isomorphism $\Psi : \langle S_{1}, S_{2} \rangle \to \langle S_{1}',S_{2} \rangle$ over $S_{2}$. We need to check that this is partial elementary over $S_{0}$. We will check the conditions of Lemma \ref{lem: extends}. 
    
    For simplicity, write $D_{i}$ for $D(S_{i})$ for $i = 0,1,2$ and $D_{1}'$ for $D(S_{1}')$. Since $S_{i} = \mathrm{acl}(S_{i})$, we know that $\mathrm{acl}_{H}(D(S_{i})) = D(S_{i})$ for all $i$. In particular, this entails that $D_{i}$ is $H$-independent, so, by symmetry and base monotonicity, we have 
    $$
    D_{i} \ind^{D}_{D_{0}H(D_{i})} H.
    $$
    for $i = 1,2$ since $D_{0} \subseteq D_{1},D_{2}$. By our hypothesis and transitivity, then, we get 
    $$
    D_{2} \ind^{D}_{D_{0}H(D_{2})} D_{1}H. 
    $$
    This means we can apply Lemma \ref{lem: H control} to obtain $H(\langle D_{1},D_{2} \rangle) = H(D_{1})H(D_{2})$. Similarly, we have $H(\langle D_{1}',D_{2} \rangle) = H(D_{1}')H(D_{2})$. In particular, we have $\mathrm{acl}_{H}(\langle D_{1},D_{2} \rangle) = \langle D_{1},D_{2} \rangle$ and $\mathrm{acl}_{H}(\langle D_{1}',D_{2} \rangle) = \langle D_{1}',D_{2} \rangle$. Moreover, by Lemma \ref{lem: commutes}, we know $D(\langle S_{1},S_{2} \rangle) = \langle D_{1},D_{2} \rangle$ and $D(\langle S_{1}',S_{2} \rangle) = \langle D_{1}',D_{2} \rangle$. 

     By our independence assumptions and stationarity of $\ind^{D}$, we have $\langle D(S_{1})D(S_{2}) \rangle \equiv_{D(S_{0})} \langle D(S_{1}') D(S_{2}) \rangle$ in the $L_{\mathrm{IS}}$-structure $D$, and by Lemma \ref{lem: type char in D}, $\langle D(S_{1}),D(S_{2}) \rangle$ and $\langle D(S_{1}'),D(S_{2}) \rangle $ are $H$-independent.  Then, by the equivalence (2)$\Leftrightarrow$(1) in Lemma \ref{lem: type char in D}, the given $L_{\mathrm{IS}}$-isomorphism $S_{1} \to S_{1}'$ over $S_{0}$ induces an $L_{R}$-partial elementary map $D(S_{1}) \to D(S_{1}')$ over $D(S_{2})$. Lemma \ref{lem: extends}, then, entails that $\Psi$ is partial elementary, completing the proof. 
\end{proof}

\begin{rem} \label{rem: equiv form}
An equivalent formulation of relative stationarity, which is sometimes easier to use, is the following: given algebraically closed subsystems $S_{1},S_{1}',S_{2},S_{2}'$ which contain an algebraically closed $S_{0}$, if $S_{1} \equiv_{S_{0}} S_{1}'$, $S_{2} \equiv_{S_{0}} S_{2}'$, $S_{1} \ind^{*}_{S_{0}} S_{2}$, $S_{1}' \ind^{*}_{S_{0}} S_{2}'$ and $\Gamma(S_{1})\Gamma(S_{2}) \equiv_{\Gamma(S_{0})} \Gamma(S_{1}')\Gamma(S_{2}')$, then $S_{1}S_{2} \equiv_{S_{0}} S_{1}'S_{2}'$. To see that this follows from relative stationarity, one can, by an automorphism, pick $S_{1}''$ such that $S_{1}''S_{2}' \equiv_{S_{0}} S_{1}S_{2}$. Then, by relative stationarity, we get $S_{1}'' \equiv_{S_{2}'} S_{1}'$, so we have 
$$
S_{1}S_{2} \equiv_{S_{0}} S_{1}''S_{2}' \equiv_{S_{0}} S_{1}'S_{2}'.
$$
\end{rem}

\begin{defn}
    \begin{enumerate}
        \item We say a theory $T$ has SOP$_{1}$ if there is a formula $\varphi(x,y)$ and a collection of tuples $(a_{\eta})_{\eta \in 2^{<\omega}}$ satisfying the following conditions:
        \begin{enumerate}
            \item (Consistency for paths) For all $\eta \in 2^{\omega}$, $\{\varphi(x;a_{\eta|i}) : i < \omega\}$ is a consistent set of formulas.
            \item (Inconsistency for certain incomparables) If $\eta \perp \nu \in 2^{<\omega}$, $\eta \unrhd (\eta \wedge \nu)^{\frown}\langle 0 \rangle$, and $\nu = (\eta \wedge \nu)^{\frown} \langle 1 \rangle$, then $\{\varphi(x,a_{\eta}), \varphi(x,a_{\nu})\}$ is inconsistent. 
        \end{enumerate}
        \item We say a theory $T$ has SOP$_{n}$ for $n \geq 3$ if there is a formula $\varphi(x,y)$ and an indiscernible sequence $(a_{i})_{i < \omega}$ satisfying the following conditions:
        \begin{enumerate}
            \item (Order property) $\models \varphi(a_{i},a_{j}) \iff i < j$,
            \item (No loops of size $n$) The set of formulas 
            $$
            \{\varphi(x_{i},x_{i+1}) : i < n-1\} \cup \{\varphi(x_{n-1},x_{0})\}
            $$
            is inconsistent.
        \end{enumerate}
    \end{enumerate}
\end{defn}

There is also a property SOP$_{2}$, defined exactly the same as SOP$_{1}$, except condition (2) is replaced by saying $\{\varphi(x,a_{\eta}), \varphi(x,a_{\nu})\}$ is inconsistent for \emph{all} incomparable pairs $\eta \perp \nu \in 2^{<\omega}$. It is known by \cite{mutchnik2026nsop2} that a theory has SOP$_{1}$ if and only if it has SOP$_{2}$ so we will not give a separate treatment for the property SOP$_{2}$. 

\begin{defn}
The following definitions make sense in an arbitrary complete theory $T$:
\begin{enumerate}
    \item For $M \models T$, we write $a \ind^{u}_{M} b$ to mean $\mathrm{tp}(a/Mb)$ is finitely satisfiable in $M$. If $(a_{i})_{i < \omega}$ is an $M$-indiscernible sequence with $a_{i} \ind^{u}_{M} a_{<i}$, we say $(a_{i})_{i < \omega}$ is a \emph{coheir sequence} over $M$. 
    \item For $M \models T$, we say a formula $\varphi(x;a)$ \emph{Kim-divides over }$M$ if there is some coheir sequence $(a_{i})_{i < \omega}$ over $M$ with $a_{0} = a$ such that $\{\varphi(x;a_{i}) : i < \omega\}$ is inconsistent. 
    \item We write $a \ind^{K}_{M} b$ to mean that $\mathrm{tp}(a/Mb)$ contains no formula that Kim-divides over $M$. 
\end{enumerate}
\end{defn}

\begin{rem}
The above definition of Kim-independence differs from the ``official definition'' since it is defined in terms of coheir sequences instead of Morley sequences in a global invariant type. These notions are easily seen to be equivalent in NSOP$_{1}$ theories, and it will be useful for us to work with coheir sequences since coheir sequences remain coheir sequences in any reduct (or, more generally, the image of a coheir sequence in an interpreted structure is again a coheir sequence). 
\end{rem}

\begin{lem} \label{lem: one direction of Kim char}
Suppose $S_{*} \models T_{\Gamma}$ and $S_{**}$ and $S_{0}$ are small algebraically closed subsystems of $\mathbb{S}_{\Gamma}$ with $S_{**} \ind^{K}_{S_{*}} S_{0}$. Then $S_{**} \ind^{*}_{S_{*}} S_{0}$. 
\end{lem}

\begin{proof}
Fix a coheir sequence $(S_{i})_{i < \omega}$ over $S_{*}$ starting with $S_{0}$. As $S_{**} \ind^{K}_{S_{*}} S_{0}$, there is some $S_{**}' \equiv_{S_{0}} S_{**}$ with $(S_{i})_{i < \omega}$ indiscernible over $S'_{**}$. Then, in particular, $D(S_{i})_{i < \omega}$ is a coheir sequence over $D(S_{*})$ which is $D(S'_{**})$-indiscernible in $D$ viewed as an $L_{H}$-structure. Thus, by Kim's Lemma in the simple theory $\mathrm{Th}_{L_{H}}(D)$, we have $D(S'_{**}) \ind^{H}_{D(S_{*})} D(S_{0})$, so $D(S_{**}) \ind^{H}_{D(S_{*})} D(S_{0})$.  It is also completely standard that Kim-independence entails $S_{**} \cap S_{0} = S_{*}$, so we get $S_{**} \ind^{*}_{S_{*}} S_{0}$ (one can see this by applying the above argument to the reduct of $\mathbb{S}_{\Gamma}$ to the language of equality instead of the interpretation of the $L_{H}$-structure $D$).   
\end{proof}

\begin{fact} \label{fact: weak IT} \cite[Theorem 5.1]{ArtemNick}
The theory $T$ is NSOP$_{1}$ if and only if, given any $M \models T$, if $b_{0}c_{0} \equiv_{M} b_{1}c_{1}$, $c_{1} \ind^{u}_{M} c_{0}$, and $c_{i} \ind^{u}_{M} b_{i}$ for $i = 0,1$, then there is $b$ such that $bc_{0} \equiv_{M} bc_{1} \equiv_{M} b_{0}c_{0}$.
\end{fact}

\begin{thm} \label{thm: preserving Kim}
    Recall that $T_{\Gamma} = \mathrm{Th}(\mathbb{S}_{\Gamma})$. 
\begin{enumerate}
\item The theory $T_{\Gamma}$ is NSOP$_{1}$ if and only if $\mathrm{Th}(\Gamma)$ is NSOP$_{1}$. Moreover, if $\mathrm{Th}(\Gamma)$ is NSOP$_{1}$, $S_{*} \models T_{\Gamma}$, and $S_{0}$ and $S_{1}$ are small algebraically closed subsystems containing $S_{*}$, we have 
$$
a \ind^{K}_{S_{*}} b\iff a \ind^{*}_{S_{*}} b \text{ and } \Gamma(\mathrm{acl}(a,S_{*})) \ind^{K}_{\Gamma(S_{*})} \Gamma(\mathrm{acl}(b,S_{*})),
$$
where, on the left-hand side $\ind^{K}$ refers to Kim-independence in $T_{\Gamma}$ and, on the right-hand side, $\ind^{K}$ refers to Kim-independence in $\mathrm{Th}(\Gamma)$. 
\item The theory $T_{\Gamma}$ is simple if and only if $\mathrm{Th}(\Gamma)$ is simple.
\item The theory $T_{\Gamma}$ is stable if and only if $\mathrm{Th}(\Gamma)$ is stable. 
\end{enumerate}
\end{thm}

\begin{proof}
(1) Since the graph is interpretable in $\mathbb{S}_{\Gamma}$, we clearly have that if $\mathrm{Th}(\Gamma)$ has SOP$_{1}$, then $T_{\Gamma}$ must as well. Thus it suffices to go the other direction.

So now we assume $\mathrm{Th}(\Gamma)$ is NSOP$_{1}$, and we will use Fact \ref{fact: weak IT} to show that $T_{\Gamma}$ is NSOP$_{1}$. Note that if $a \ind^{u}_{S_{*}} b$ then also $\mathrm{acl}(a,S_{*}) \ind^{u}_{S_{*}} \mathrm{acl}(b,S_{*})$. Thus it suffices to show that, given algebraically closed subsystems $S_{**}, S_{**}'$, $S_{0}$, $S_{1}$ all containing $S_{*}$, if $S_{**}S_{0} \equiv_{S_{*}} S_{**}'S_{1}$, $S_{0} \ind^{u}_{S_{*}} S_{**}$, $S_{1} \ind^{u}_{S_{*}} S_{**}'$, and $S_{1} \ind^{u}_{S_{*}} S_{0}$, then there is some $S_{**}''$ with
$$
S_{**}''S_{0} \equiv_{S_{*}} S_{**}''S_{1} \equiv_{S_{*}} S_{**}S_{0}.
$$
In this situation, since the graph is interpreted in each subsystem, we have the following:
\begin{itemize}
\item $\Gamma(S_{*}) \models \mathrm{Th}(\Gamma)$.
\item $\Gamma(S_{**})\Gamma(S_{0}) \equiv_{\Gamma(S_{*})} \Gamma(S_{**}')\Gamma(S_{1})$.
\item $\Gamma(S_{0}) \ind^{u}_{\Gamma(S_{*})} \Gamma(S_{**})$.
\item $\Gamma(S_{1}) \ind^{u}_{\Gamma(S_{*})} \Gamma(S'_{**})$
\item $\Gamma(S_{1}) \ind^{u}_{\Gamma(S_{*})} \Gamma(S_{0})$.
\end{itemize} 
Thus, by Fact \ref{fact: weak IT} and our assumption that $\mathrm{Th}(\Gamma)$ is NSOP$_{1}$, there is some $\Gamma''_{**}$ (which we can assume is in $\Gamma(\mathbb{S}_{\Gamma})$) such that 
$$
\Gamma_{**}''\Gamma(S_{0}) \equiv_{\Gamma(S_{*})} \Gamma_{**}''\Gamma(S_{1}) \equiv_{\Gamma(S_{*})} \Gamma(S_{**})\Gamma(S_{0}). 
$$
Under the identification of $H$ and $\Gamma(\mathbb{S}_{\Gamma})$, we have, then, 
$$
\Gamma''_{**}H(D(S_{0})) \equiv^{L_{R}}_{H(D(S_{*}))}\Gamma_{**}''H(D(S_{1})) \equiv^{L_{R}}_{H(D(S_{*}))} H(D(S_{**}))H(D(S_{0}))
$$
in the $L_{R}$-structure $D$. 

\textbf{Claim}:  We can upgrade the previous equation to having the same type over $D(S_{*})$, that is:
$$
\Gamma''_{**}H(D(S_{0})) \equiv^{L_{R}}_{D(S_{*})}\Gamma_{**}''H(D(S_{1})) \equiv^{L_{R}}_{D(S_{*})} H(D(S_{**}))H(D(S_{0})).
$$
\emph{Proof of Claim}:  Because $S_{*}$ was chosen to be a model, it is in particularly algebraically closed so $D(S_{*})$ is $H$-independent. So we have $D(S_{*}) \ind^{D}_{H(D(S_{*}))} H$.  Then, in particular, we have, by base monotonicity, we have that the tuples $D(S_{*})\Gamma''_{**}H(D(S_{0}))$, $D(S_{*})\Gamma_{**}''H(D(S_{1}))$, and $D(S_{*})H(D(S_{**}))H(D(S_{0}))$ are all $H$-independent. Thus, by Lemma \ref{lem: type char in D}, it suffices to show  
$$
\Gamma''_{**}H(D(S_{0})) \equiv_{D(S_{*})}^{D} \Gamma_{**}''H(D(S_{1})) \equiv^{D}_{D(S_{*})} H(D(S_{**}))H(D(S_{0})). 
$$
Note that, by hypothesis, we have $H(D(S_{**})) \ind^{D}_{H(D(S_{*}))} H(D(S_{0}))$ so, by invariance, we also have $\Gamma''_{**} \ind^{D}_{H(D(S_{*}))} H(D(S_{0}))$. Since $D(S_{*}) \ind^{D}_{H(D(S_{*}))} H$ we also have, by base monotonicity, $\Gamma''_{**} \ind^{D}_{H(D(S_{0}))} D(S_{0})$ and $H(D(S_{**})) \ind^{D}_{H(D(S_{0}))} D(S_{0})$ so, by transitivity, we have $\Gamma''_{**} \ind^{D}_{H(D(S_{*}))} D(S_{*}) H(D(S_{0}))$ and $H(D(S_{**})) \ind^{D}_{H(D(S_{*}))} D(S_{*}) H(D(S_{0}))$.  We know $\Gamma''_{**} \equiv_{H(D(S_{*}))} H(D(S_{**}))$ and $\mathrm{acl}(H(D(S_{**}))) = \langle H(D(S_{**})) \rangle = \mathrm{dcl}(H(D(S_{**})))$ so $\Gamma''_{**}$ and $H(D(S_{**}))$ have the same type over $H(D(S_{*}))$ and, by weak elimination of imaginaries in $D$ (Fact \ref{fact: weak EI}), types over algebraically closed sets are stationary. Thus by stationarity, we obtain 
$$
\Gamma''_{**}D(S_{*}) H(D(S_{0})) \equiv^{D}_{H(D(S_{*}))} H(D(S_{**})) D(S_{*}) H(D(S_{0}))
$$
and thus $\Gamma''_{**}H(D(S_{0})) \equiv^{D}_{D(S_{*})} H(D(S_{**})) H(D(S_{0}))$. An identical argument, replacing $S_{0}$ and $S_{1}$ and $S_{**}$ with $S_{**}'$ shows that $\Gamma''_{**}H(D(S_{1})) \equiv^{D}_{D(S_{*})} H(D(S'_{**}))H(D(S_{1}))$. Since, by assumption, $H(D(S_{**}))H(D(S_{0})) \equiv_{H(D(S_{*}))} H(D(S'_{**}))H(D(S_{1}))$, this proves the claim.\qed

Since $\Gamma''_{**} \equiv^{L_{R}}_{D(S_{*})} D(S_{**})$, we can, by an automorphism, pick $D''_{**}$ in $D$ so that $H(D_{**}'') = \Gamma''_{**}$ and $D''_{**} \equiv_{D(S_{*})} D(S_{**})$. By extension, we may assume that $D''_{**}$ has been chosen so that $D''_{**} \ind^{D}_{\Gamma''_{**}D(S_{*})} HD(S_{0})D(S_{1})$. Finally, by Corollary \ref{cor : from D to S}, we have $D''_{**} \equiv_{S_{*}} D(S_{**})$ in $\mathbb{S}_{\Gamma}$, so we may choose, by an automorphism, some $S''_{**}$ in $\mathbb{S}_{\Gamma}$ with $D(S''_{**}) = D''_{**}$. 

By transitivity, we have $D''_{**} \ind^{D}_{D(S_{*})} D(S_{0})$ and $D''_{**} \ind^{D}_{D(S_{*})} D(S_{1})$ so by Lemma \ref{lem: disjoint in D}, we have $S''_{**} \ind^{*}_{S_{*}} S_{0}$ and $S''_{**} \ind^{*}_{S_{*}} S_{1}$. Then, by Proposition \ref{prop: relative stationarity}, we have $S''_{**}S_{0} \equiv_{S_{*}} S_{**} S_{0}$ and $S''_{**}S_{1} \equiv_{S_{*}} S'_{**}S_{1}$. Thus, by Fact \ref{fact: weak IT}, we conclude that $T_{\Gamma}$ is NSOP$_{1}$. 

Now that we know $T_{\Gamma}$ is NSOP$_{1}$, we know 
$$
a \ind^{K}_{S_{*}} b \iff \mathrm{acl}(a,S_{*}) \ind^{K}_{S_{*}} \mathrm{acl}(b,S_{*}),
$$
by \cite[Corollary 5.17]{MR4081726}. Thus, we have $a \ind^{K}_{S_{*}} b$ implies $a \ind^{*}_{S_{*}} b$ by Lemma \ref{lem: one direction of Kim char}. For the other direction, we set $S_{**} = \mathrm{acl}(a,S_{*})$ and $S_{0} = \mathrm{acl}(b,S_{*})$, and assume $S_{**} \ind^{*}_{S_{*}} S_{0}$ and $\Gamma(S_{**}) \ind^{K}_{\Gamma(S_{*})} \Gamma(S_{0})$. Then we take a coheir sequence $(S_{i})_{i < \omega}$ over $S_{*}$ starting with $S_{0}$. The above argument (applied to $(S_{i})_{i < \omega}$ rather than to just $S_{0}$ and $S_{1}$) produces some $S'_{**}$ with $S'_{**}S_{i} \equiv_{S_{*}} S_{**}S_{0}$ for all $i$, which entails that $S'_{**} \ind^{K}_{S_{*}} S_{0}$, completing the characterization of Kim-independence. 

(2) By \cite[Proposition 8.8]{MR4081726}, it suffices to show that if $\mathrm{Th}(\Gamma)$ is simple, then $\ind^{K}$ satisfies base monotonicity over models. So we assume $\mathrm{Th}(\Gamma)$ is simple and we fix $S_{*} \preceq S_{0} \preceq S_{1} \models T_{\Gamma}$ and $S_{2} = \mathrm{acl}(S_{2}) \supseteq S_{1}$. We suppose $S = \mathrm{acl}(S,S_{*})$ is a small subsystem of $\mathbb{S}_{\Gamma}$ with $S \ind^{K}_{S_{*}} S_{2}$. Setting $S' = \mathrm{acl}(S,S_{1})$, we must show $S' \ind^{K}_{S_{1}} S_{2}$. We will make use of the characterization of Kim-independence in (1).

By Lemma \ref{lem: commutes}, $D(S') = \langle D(S),D(S_{1}) \rangle$. By base monotonicity of $\ind^{H}$, we have $D(S) \ind^{H}_{D(S_{1})} D(S_{2})$ and thus $D(S') \ind^{H}_{D(S_{1})} D(S_{2})$. In particular, we have $D(S') \vee D(S_{2}) = D(S_{1})$ so $S' \vee S_{2} = S_{1}$ by Lemma \ref{lem: disjoint in D}, which shows $S' \ind^{*}_{S_{1}} S_{2}$. Moreover, by Lemma \ref{lem: H control}, we have $H(\langle D(S), D(S_{1}) \rangle) = H(D(S))H(D(S_{1}))$ so, using the identification of $H$ and $\Gamma(\mathbb{S}_{\Gamma})$, we have $\Gamma(S') = \Gamma(S)\Gamma(S_{1})$. Thus, by base monotonicity of Kim-independence in $\mathrm{Th}(\Gamma)$, we have $\Gamma(S') \ind^{K}_{\Gamma(S_{1})} \Gamma(S_{2})$. 

(3) In a stable or simple theory, non-forking independence agrees with Kim-independence. If $\mathrm{Th}(\Gamma)$ is stable then it follows from (2) and Proposition \ref{prop: relative stationarity} that non-forking independence in $T_{\Gamma}$ is stationary over models, which implies that the theory $T_{\Gamma}$ is stable. 
\end{proof}

\begin{thm} \label{thm: preserving nsopn}
Fix $n \geq 3$. Suppose $\Gamma$ is a graph and $\mathrm{Th}(\Gamma)$ is NSOP$_{n}$. Then $T_{\Gamma} = \mathrm{Th}(S(\widetilde{G_{\Gamma}}))$ is NSOP$_{n}$. 
\end{thm}

\begin{proof}
Suppose $S_{*} \preceq \mathbb{S}_{\Gamma}$ is a small submodel and $(S_{i})_{i < \omega}$ is a coheir sequence of algebraically closed subsystems of $\mathbb{S}_{\Gamma}$ over $S_{*}$, with $S_{*}$ contained in each $S_{i}$. Suppose $p(X,Y) = \mathrm{tp}(S_{0},S_{1}/S_{*})$ and $(S_{i},S_{j})$ realize $p(X,Y)$ if and only if $i < j$. Then we must show 
$$
p(X_{0},X_{1}) \cup p(X_{1},X_{2}) \cup \ldots \cup p(X_{n-2},X_{n-1}) \cup p(X_{n-1},X_{0})
$$
is consistent. 

Let $q(z,w) = \mathrm{tp}_{\Gamma}(\Gamma(S_{0}), \Gamma(S_{1})/\Gamma(S_{*}))$. By our assumption that $\mathrm{Th}(\Gamma)$ is NSOP$_{n}$, we can find $\Gamma_{0}, \Gamma_{1}, \ldots, \Gamma_{n-1}$ in $\Gamma(\mathbb{S}_{\Gamma})$ such that $(\Gamma_{i},\Gamma_{i+1}) \models q$ for all $i = 0, \ldots, n-2$ and also $(\Gamma_{n-1},\Gamma_{0}) \models q$. 

By the choice of $q$ (and the $S_{*}$-indiscernibility of $(S_{i})_{i < \omega}$), we have $\Gamma_{i} \equiv_{\Gamma(S_{*})} \Gamma(S_{0})$ for all $i < n$. Using the identification of $H$ and $\Gamma$ and Lemma \ref{lem: type char in D}, we have
$$
\Gamma_{i} \equiv^{L_{R}}_{H(D(S_{*}))} H(D(S_{0}))
$$
in $D$. Since $S_{*}$ is algebraically closed, we have $D(S_{*}) \ind^{D}_{H(D(S_{*}))} H$ so, in particular, we have 
$$
D(S_{*}) \ind^{D}_{H(D(S_{*}))} \Gamma_{i} H(D(S_{0})).
$$
Since $\mathrm{acl}(H(D(S_{*}))) = \langle H(D(S_{*})) \rangle = \mathrm{dcl}(H(D(S_{*})))$ and types over algebraically closed sets are stationary (by Fact \ref{fact: weak EI}), we obtain 
$$
D(S_{*})H(D(S_{*}))\Gamma_{i} \equiv^{D} D(S_{*})H(D(S_{*}))H(D(S_{0})).
$$
Additionally, we know both $D(S_{*})H(D(S_{*}))\Gamma_{i}$ and $D(S_{*})H(D(S_{*}))H(D(S_{0}))$ are $H$-independent so we get $\Gamma_{i} \equiv^{L_{R}}_{D(S_{*})} H(D(S_{0}))$ by Lemma \ref{lem: type char in D}. 

Pick $D_{0} \equiv_{D(S_{*})} D(S_{0})$ with $H(D(S_{0})) = \Gamma_{0}$. By extension, we can assume $D_{0} \ind^{D}_{D(S_{*})\Gamma_{0}} H$. Then, repeating this inductively, for each $i < n$, we can pick $D_{i} \equiv_{D(S_{*})} D(S_{0})$ with $H(D_{i}) = \Gamma_{i}$ and $D_{i} \ind^{D}_{D(S_{*})\Gamma_{i}} H D_{<i}$. 

By Corollary \ref{cor : from D to S}, we have $D_{i} \equiv_{S_{*}} D(S_{0})$ in $\mathbb{S}_{\Gamma}$ for all $i < n$. Then, by an automorphism, we can pick some $S'_{i} \equiv_{S_{*}} S_{0}$ with $D(S'_{i}) = D_{i}$ for all $i < n$. By construction, we have $S'_{i} \ind^{*}_{S_{*}} S'_{i+1}$ for all $i < n-1$ and $S'_{n-1} \ind^{*}_{S_{*}} S'_{0}$. Moreover, for all $i < n-1$,
$$
\Gamma(S'_{i})\Gamma(S'_{i+1}) \equiv_{\Gamma(S_{*})} \Gamma(S'_{n-1})\Gamma(S'_{0}) \equiv_{\Gamma(S_{*})} \Gamma(S_{0})\Gamma(S_{1}).
$$
Thus, by relative stationarity, Proposition \ref{prop: relative stationarity} (and Remark \ref{rem: equiv form}), we have $(S'_{i},S'_{i+1}) \models p$ for all $i < n-1$ and $(S'_{n-1},S'_{0}) \models p$. This shows $T_{\Gamma}$ is NSOP$_{n}$. 
\end{proof}

\section{PAC fields} \label{sec: fields}

\subsection{Preliminaries on PAC fields}

If $K$ is a field, we write $K^{s}$ for the separable closure of $K$, $K^{\text{alg}}$ for the algebraic closure of $K$, and we write $\mathcal{G}(K)$ for the absolute Galois group of $K$.

The following fact lets us realize the groups we have constructed in the previous sections as absolute Galois groups of PAC fields:

\begin{fact} \label{fact: projective=absolute} \cite{ax1968elementary} \cite{lubotzky1981subgroups}
    The class of projective profinite groups is equal to the class of absolute Galois groups of PAC fields.
\end{fact}

Write $\ind^{\mathrm{SCF}}$ for non-forking independence in the (stable) theory of separably closed fields. The following theorem of Chatzidakis reduces the study of amalgamation in PAC fields to the study of amalgamation in the inverse system of the Galois group:

\begin{fact} \cite[Theorem 2.1]{chatzidakis2019amalgamation} \label{fact: zoes theorem}
    Let $F$ be a PAC field and let $A$, $B$, $C_{1}$, and $C_{2}$ be algebraically closed subsets of $F$ containing a common algebraically closed subset $E$, which contains a $p$-basis if the degree of imperfection of $F$ is finite. Assume $A \cap B = E$, $C_{1} \ind^{\mathrm{SCF}}_{E} A$, $C_{2} \ind^{\mathrm{SCF}}_{E} B$, and there is an $E^{s}$-isomorphism $\varphi : C_{1}^{s} \to C_{2}^{s}$ such that $\varphi(C_{1}) = C_{2}$. Suppose further that there is $S_{0} \subseteq S(\mathcal{G}(F))$ and elementary isomorphisms
    \begin{eqnarray*}
    S\Psi_{1} : \langle S(\mathcal{G}(C_{1})), S(\mathcal{G}(A)) \rangle &\to& \langle S_{0},S(\mathcal{G}(A)) \rangle \\
    S\Psi_{2} : \langle S(\mathcal{G}(C_{2})),S(\mathcal{G}(B)) \rangle &\to& \langle S_{0},S(\mathcal{G}(B)) \rangle
    \end{eqnarray*}
    such that
    \begin{enumerate}
        \item $S\Psi_{1}$ is the identity on $S(\mathcal{G}(A))$, $S\Psi_{2}$ is the identity on $S(\mathcal{G}(B))$, and $S\Psi_{i}(S(\mathcal{G}(C_{i}))) = S_{0}$,
        \item if $S\Phi : S(\mathcal{G}(C_{1})) \to S(\mathcal{G}(C_{2}))$ is the morphism double dual to $\varphi$, then
        $$
        S\Psi_{2} \circ S\Phi = S \Psi_{1}|_{S(\mathcal{G}(C_{1}))}.
        $$
    \end{enumerate}
    Then in some elementary extension $F'$ of $F$, there is $C \models \mathrm{tp}(C_{1}/A) \cup \mathrm{tp}(C_{2}/B)$ with $C \ind^{\mathrm{SCF}}_{E} AB$ and $S(\mathcal{G}(C)) = S_{0}$. 
\end{fact}

Finally, we have a preservation theorem for the SOP$_{n}$ hierarchy (for $n \geq 3$): 

\begin{fact} \cite[Theorem 2.9]{chatzidakis2019amalgamation} \label{fact: zoe nsopn fact}
    Suppose $n \geq 3$. If $F$ is a PAC field and $\mathrm{Th}(S(\mathcal{G}(F)))$ is NSOP$_{n}$, then $\mathrm{Th}(F)$ is NSOP$_{n}$. 
\end{fact}

\subsection{NSOP$_{1}$ and Kim-independence}

Interpretability of $S(\mathcal{G}(F))$ in $(F^{\text{alg}},F)$ is proved in \cite[Proposition 5.5]{chatzidakis2002properties}.  The ``moreover'' clause is clear from the proof.  

\begin{fact} \label{interpretability}
Both $F$ and $S(\mathcal{G}(F))$ are interpretable in $(K,F)$ where $K$ is any algebraically closed field containing $F$.  Call the interpretation $\pi$.  Moreover, if $L \subseteq F$ is a subfield such that $F$ is a regular extension of $L$, then the restriction of $\pi$ to $(K,L)$ produces an interpretation of $S(\mathcal{G}(L))$, which embeds into $S(\mathcal{G}(F))$ via the natural restriction map $\mathcal{G}(F) \to \mathcal{G}(L)$.  
\end{fact}

\begin{lem} \label{lem : inv system implication}
Let $F$ be a large sufficiently saturated and homogeneous field (i.e. a monster model of its theory) and $M \prec F$ a small elementary substructure.  Suppose $A = \text{acl}(A)$, $B = \text{acl}(B)$ are subsets of $F$ with $M \subseteq A \cap B$.  
\begin{enumerate}
\item If $A \equiv_{M} B$ in $F$, then $S(\mathcal{G}(A)) \equiv_{S(\mathcal{G}(M))} S(\mathcal{G}(B))$.
\item If $(A_{i})_{i \in I}$ is an $I$-indexed indiscernible over $M$ in $\text{tp}(A/M)$, then $(S(\mathcal{G}(A_{i})))_{i \in I}$ is an $I$-indexed indiscernible over $S(\mathcal{G}(M))$.  
\item If $A \ind^{u}_{M} B$ in $F$, then $S(\mathcal{G}(A)) \ind^{u}_{S(\mathcal{G}(M))} S(\mathcal{G}(B))$ in $S(\mathcal{G}(M))$.  
\end{enumerate}
\end{lem}

\begin{proof}
(1)  If $A \equiv_{M} B$ in $F$, then there is an automorphism $\sigma \in \text{Aut}(F/M)$ with $\sigma(A) = B$.  The map $\sigma$ has an extension $\tilde{\sigma}$ to $F^{\text{alg}}$ which is, then, an automorphism of the pair $(F^{\text{alg}},F)$ taking $A$ to $B$ and fixing $M$ pointwise.  It follows $A \equiv_{M} B$ in the pair $(F^{\text{alg}},F)$.  Since $A = \text{acl}(A)$ and $B = \text{acl}(B)$, we know $F$ is a regular extension of $A$ and of $B$ (see, e.g., \cite[Section 1.17]{ZoePAC2}).  By Fact \ref{interpretability}, we have $S(\mathcal{G}(A)) \equiv_{S(\mathcal{G}(M))} S(\mathcal{G}(B))$.  

(2)  Suppose $(A_{i})_{i \in I}$ is an $I$-indexed indiscernible over $M$. Given two $k$-tuples $\overline{i} = (i_{0},\ldots, i_{k-1})$ and $\overline{j} = (j_{0},\ldots, j_{k-1})$ from $I$ with $\text{qftp}(\overline{i}) = \text{qftp}(\overline{j})$, we know $A_{i_{0}}\ldots A_{i_{k-1}} \equiv_{M} A_{j_{0}}\ldots A_{j_{k-1}}$ so $\text{acl}(A_{i_{0}}\ldots A_{i_{k-1}}) \equiv_{M} \text{acl}(A_{j_{0}}\ldots A_{j_{k-1}})$.  Then by (1) $S(\mathcal{G}(\text{acl}(A_{i_{0}}\ldots A_{i_{k-1}})))\equiv_{S(\mathcal{G}(M))} S(\mathcal{G}(\text{acl}(A_{j_{0}}\ldots A_{j_{k-1}})))$, which implies $(S(\mathcal{G}(A_{i})))_{i \in I}$ is an $I$-indexed indiscernible over $S(\mathcal{G}(M))$.

(3)  In \emph{any} theory, if $\pi$ is an interpretation of the structure $X$ in the structure $Y$, and $A \ind^{u}_{C} B$ in $Y$, then $\pi(A) \ind^{u}_{\pi(C)} \pi(B)$.  It follows that if $A \ind^{u}_{M} B$ in $F$, then $S(\mathcal{G}(A)) \ind^{u}_{S(\mathcal{G}(M))} S(\mathcal{G}(B))$ by Fact \ref{interpretability}.  
\end{proof}

Our candidate for Kim-independence in a PAC field with NSOP$_{1}$ absolute Galois group will be the conjunction of non-forking independence in the ambient separably closed field and Kim-independence in the sense of the Galois group.  

\begin{defn}
Suppose $F$ is a field and $\text{Th}(S(\mathcal{G}(F)))$ is NSOP$_{1}$.  If $a,b$ are tuples in some elementary extension of $F$, we say $a$ and $b$ are \emph{weakly independent} over $F$ if, letting $A = \text{acl}(aF)$ and $B = \text{acl}(bF)$, 
\begin{enumerate}
\item $A \ind^{\mathrm{SCF}}_{F} B$
\item $S(\mathcal{G}(A)) \ind^{K}_{S(\mathcal{G}(F))} S(\mathcal{G}(B))$.  
\end{enumerate}
\end{defn}

\begin{rem}
This differs from the definition given by Chatzidakis \cite{chatzidakis2019amalgamation} in the context where $\text{Th}(S(\mathcal{G}(F)))$ is simple, as we use Kim-independence and Chatzidakis used non-forking independence in the definition of weak independence.  If the base is a model, these definitions agree, and the definition we give here is well-behaved in the broader context of NSOP$_{1}$ Galois groups.  
\end{rem}

In order to show that the theory of a PAC field with NSOP$_{1}$ Galois group (in the inverse system language) is NSOP$_{1}$, we will use the following characterization of NSOP$_{1}$ in terms of a weak variant of the independence theorem:

\begin{prop} \label{prop: NSOP1}
    Suppose $F$ is a PAC field and $\text{Th}(S(\mathcal{G}(F)))$ is NSOP$_{1}$.  Then $\mathrm{Th}(F)$ is NSOP$_{1}$.
\end{prop}

\begin{proof}
    Let $L \models \mathrm{Th}(F)$.  Suppose $a_{0}b_{0} \equiv_{L} a_{1}b_{1}$ with $b_{i} \ind^{u}_{L} a_{i}$ for $i = 0,1$ and $b_{1} \ind^{u}_{L} b_{0}$. Let $A_{i} = \mathrm{acl}(a_{i}L)$ and $B_{i} = \mathrm{acl}(b_{i}L)$ for $i = 0,1$. Then we have $B_{i} \ind^{u}_{L} A_{i}$ for $i = 0,1$ and $B_{1} \ind^{u}_{L} B_{0}$. Then by Lemma \ref{lem : inv system implication}(3), we know $S(\mathcal{G}(B_{i})) \ind^{u}_{S(\mathcal{G}(L))} S(\mathcal{G}(A_{i}))$ for $i = 0,1$ and $S(\mathcal{G}(B_{1})) \ind^{u}_{S(\mathcal{G}(L))} S(\mathcal{G}(B_{0}))$. Then, because $\mathrm{Th}(S(\mathcal{G}(L)))$ is NSOP$_{1}$, there is some $S_{*}$ such that $S_{*}S(\mathcal{G}(B_{0})) \equiv_{S(\mathcal{G}(L))} S_{*}S(\mathcal{G}(B_{1})) \equiv_{S(\mathcal{G}(L))} S(\mathcal{G}(A_{0})) S(\mathcal{G}(B_{0}))$, by Fact \ref{fact: weak IT}.  Then by Fact \ref{fact: zoes theorem}, there is some $A_{*}$ with $A_{*}B_{0} \equiv_{L} A_{*}B_{1} \equiv_{L} A_{0}B_{0}$. In particular, letting $a_{*}$ be the subtuple of $A_{*}$ corresponding to $a_{0}$ in $A_{0}$, we have $a_{*}b_{0}\equiv_{L} a_{*}b_{1} \equiv_{L} a_{0}b_{0}$. By the other direction of Fact \ref{fact: weak IT}, we have shown $\mathrm{Th}(F)$ is NSOP$_{1}$. 
\end{proof}

\begin{thm} \label{Kimchar}
Suppose $F$ is a PAC field and $\text{Th}(S(\mathcal{G}(F)))$ is NSOP$_{1}$.  Then given $L \models \text{Th}(F)$, $a \ind^{K}_{L} b$ if and only if $a$ and $b$ are weakly independent over $L$.
\end{thm}

\begin{proof}
Let $A = \text{acl}(aL)$ and $B = \text{acl}(bL)$.  By Proposition \ref{prop: NSOP1}, we already know $\mathrm{Th}(F)$ is NSOP$_{1}$ and therefore $a \ind^{K}_{L} b$ if and only if $A \ind^{K}_{L} B$ \cite[Corollary 5.17]{MR4081726}, so we show $A \ind^{K}_{L} B$. 

First, assume $A \ind^{K}_{L} B$ and we will show $A$ and $B$ are weakly independent over $L$.  Let $(B_{i})_{i < \omega}$ be a Morley sequence in a global type finitely satisfiable in $L$ extending $\text{tp}(B/L)$.  As $A \ind^{K}_{L} B$, we may assume $(B_{i})_{i < \omega}$ is $A$-indiscernible.  Then by Kim's lemma in the stable theory SCF, we know $A \ind^{\text{SCF}}_{L} B$.  Also $(S(\mathcal{G}(B_{i})))_{i < \omega}$ is a Morley sequence in a global type finitely satisfiable in $S(\mathcal{G}(L))$ which is moreover $S(\mathcal{G}(A))$-indiscernible.  As $\text{Th}(S(\mathcal{G}(L)))$ is NSOP$_{1}$, this implies $S(\mathcal{G}(A)) \ind^{K}_{S(\mathcal{G}(L))} S(\mathcal{G}(B))$ by Kim's lemma for Kim-dividing.  In other words, $A$ and $B$ are weakly independent over $L$.  

Now assume $A$ and $B$ are weakly independent over $L$ and we will show $A \ind^{K}_{L} B$.  Let $(B_{i})_{i < \omega}$ be an $L$-finitely satisfiable Morley sequence over $L$ with $B_{0} = B$.  Put $C_{n,0} = \text{acl}(B_{0},\ldots, B_{2^{n}-1})$ and $C_{n,1} = \text{acl}(B_{2^{n}},\ldots, B_{2^{n+1} - 1})$.  Note that $C_{n,1} \ind^{u}_{L} C_{n,0}$ and $C_{n,0} \equiv_{L} C_{n,1}$ for all $n < \omega$.  By induction on $n$, we will choose $A_{n}$ so that $A_{0} = A$ and for all $n \geq 0$ the following conditions are satisfied:
\begin{enumerate}
\item $A_{n+1} \equiv_{C_{n,0},C_{n,1}} A_{n}$.
\item $A_{n}$ is weakly independent from $\text{acl}(C_{n,0},C_{n,1}) = C_{n+1,0}$ over $L$.
\item $A_{n}C_{n,0} \equiv_{L} A_{n}C_{n,1}$.
\end{enumerate}
Suppose we are at stage $n \geq 0$ of the induction.  Pick $A'_{n}$ so that $A_{n}C_{n+1,0} \equiv_{L} A_{n}'C_{n+1,1}$.  Then, in particular, $S(\mathcal{G}(A_{n})) \equiv_{S(\mathcal{G}(L))} S(\mathcal{G}(A_{n}'))$, $S(\mathcal{G}(A_{n})) \ind^{K}_{S(\mathcal{G}(L))} S(\mathcal{G}(C_{n+1,0}))$, and $S(\mathcal{G}(A_{n}')) \ind^{K}_{S(\mathcal{G}(L))} S(\mathcal{G}(C_{n+1,1}))$.  Moreover, because $C_{n+1,1} \ind^{u}_{L} C_{n+1,0}$, we have $S(\mathcal{G}(C_{n+1,1})) \ind^{K}_{S(\mathcal{G}(L))} S(\mathcal{G}(C_{n+1,0}))$.  As $\text{Th}(S(\mathcal{G}(F)))$ is NSOP$_{1}$, we may apply the independence theorem for $\ind^{K}$ to obtain $S$ so that $S \equiv_{S(\mathcal{G}(C_{n+1,0}))} S(\mathcal{G}(A_{n}))$, $S \equiv_{S(\mathcal{G}(C_{n+1,1}))} S(\mathcal{G}(A_{n}'))$ and $S \ind^{K}_{S(\mathcal{G}(L))} S(\mathcal{G}(C_{n+1,0})) S(\mathcal{G}(C_{n+1,1}))$.  Note that $S(\mathcal{G}(C_{n+2,0})) \supseteq S(\mathcal{G}(C_{n+1,0})) S(\mathcal{G}(C_{n+1,1}))$ so, by extension, we may assume $S$ has been chosen so that $S \ind^{K}_{S(\mathcal{G}(L))} S(\mathcal{G}(C_{n+2,0}))$.  By Fact \ref{fact: zoes theorem}, there is $A_{n+1}$ with $A_{n+1} \equiv_{C_{n+1,0}} A_{n}$, $A_{n+1} \equiv_{C_{n+1,1}} A'_{n}$, $A_{n+1} \ind^{\text{SCF}}_{L} C_{n+1,0}C_{n+1,1}$ and $S(\mathcal{G}(A_{n+1})) = S$.  Note that this implies $A_{n+1} \ind^{\text{SCF}}_{L} C_{n+2,0}$ because $C_{n+2,0} = \text{acl}(C_{n+1,0}C_{n+1,1})$ and $S(\mathcal{G}(A_{n+1})) \ind^{K}_{S(\mathcal{G}(L))} S(\mathcal{G}(C_{n+2,0}))$ by the choice of $S$ so $A_{n+1}$ and $C_{n+2,0}$ are weakly independent over $L$.  This completes the induction.  

Let $p(X;B) = \text{tp}(A/B)$.  Observe that, by construction, $A_{n} \models \bigcup_{i < 2^{n}} p(X;B_{i})$ so $\bigcup_{i < \omega} p(X;B_{i})$ is consistent by compactness.  As the sequence $\langle B_{i} : i < \omega\rangle$ is an arbitrary $L$-finitely satisfiable Morley sequence over $L$, this shows $A \ind^{K}_{L} B$.  
\end{proof}

\begin{rem}
The above result first appeared in our dissertation \cite[Theorem 7.2.6]{ramsey2018independence} though there was an error, pointed out to us by Daniel Hoffmann and Jung-Uk Lee: we used that $a \ind^{K}_{M} b$ if and only if $\mathrm{acl}(aM) \ind^{K}_{M} \mathrm{acl}(bM)$ in our proof of NSOP$_{1}$, though we only know this equivalence \emph{assuming} that the theory is NSOP$_{1}$. The fix is to first prove NSOP$_{1}$ by invoking the equivalence of NSOP$_{1}$ with the weak independence theorem, Fact \ref{fact: weak IT}, and then, having done that, using the same proof to identify Kim-independence. 
\end{rem}

\begin{fact} \cite[Section 28.10]{fried2008field} \label{fact: interpretability in the field}
If $K$ is a field, viewed as a structure in the language of rings, with absolute Galois group $G$, then $\Gamma(S(G))$ is interpretable in $K$. 
\end{fact}

Now we come to the main theorem:

\begin{thm} \label{thm: main thm}
Given a graph $\Gamma$, let $K_{\Gamma}$ be a PAC field with absolute Galois group $\widetilde{G_{\Gamma}}$ (which exists by Fact \ref{fact: projective=absolute}). View $K_{\Gamma}$ as a structure in the language of rings. Then, for all $n \geq 1$, the field $K_{\Gamma}$ has NSOP$_{n}$ theory if and only if $\mathrm{Th}(\Gamma)$ is NSOP$_{n}$.
\end{thm}

\begin{proof}
Note that, by Fact \ref{fact: interpretability in the field}, if $\mathrm{Th}(\Gamma)$ has SOP$_{n}$, then $\mathrm{Th}(K_{\Gamma})$ also has SOP$_{n}$.  On the other hand, for $n = 1$, if $\mathrm{Th}(\Gamma)$ is NSOP$_{1}$, then $T_{\Gamma} = \mathrm{Th}(S(\widetilde{G_{\Gamma}}))$ is NSOP$_{1}$ by Theorem \ref{thm: preserving Kim}, and thus $\mathrm{Th}(K_{\Gamma})$ is NSOP$_{1}$ by Proposition \ref{prop: NSOP1}. By \cite{mutchnik2026nsop2}, the $n = 2$ case is equivalent to the $n = 1$ case. For $n \geq 3$, we know that if $\mathrm{Th}(\Gamma)$ is NSOP$_{n}$, then $T_{\Gamma}$ is NSOP$_{n}$ by Theorem \ref{thm: preserving nsopn}, which entails $\mathrm{Th}(K_{\Gamma})$ is NSOP$_{n}$ by Fact \ref{fact: zoe nsopn fact}. 
\end{proof}

\begin{cor}
The SOP$_{n}$ hierarchy is strict in fields. That is, for all $n \geq 3$, there is a (pure) field $F$ such that $\mathrm{Th}(F)$ is SOP$_{n}$ and NSOP$_{n+1}$.
\end{cor}

\begin{proof}
Immediate from Theorem \ref{thm: main thm}, since the SOP$_{n}$ hierarchy is strict in countable theories \cite[Claim 2.8]{shelah1996toward} and it is well known that every countable theory is bi-interpretable with a graph (see, e.g., \cite[Chapter 5]{hodges1993model}).
\end{proof}

\subsection*{Acknowledgements}

Some of this work was done, or at least begun, as part of our dissertation research under the direction of Tom Scanlon, whom we would like to thank. We are also grateful to Daniel Hoffmann and Jung-Uk Lee for catching an error in an early version of Section 5.

\bibliographystyle{alpha}
\bibliography{ms.bib}{}

@incollection{ribes2010profinite,
  title={Profinite groups},
  author={Ribes, Luis and Zalesskii, Pavel},
  booktitle={Profinite Groups},
  pages={19--74},
  year={2010},
  publisher={Springer}
}

@article{conant2025enriching,
  title={Enriching a predicate and tame expansions of the integers},
  author={Conant, Gabriel and d’Elb{\'e}e, Christian and Halevi, Yatir and Jimenez, L{\'e}o and Rideau-Kikuchi, Silvain},
  journal={Journal of Mathematical Logic},
  volume={25},
  number={01},
  pages={2450011},
  year={2025},
  publisher={World Scientific}
}

@article{kruckman2022interpolative,
  title={Interpolative fusions II: Preservation results},
  author={Kruckman, Alex and Tran, Minh Chieu and Walsberg, Erik},
  journal={arXiv preprint arXiv:2201.03534},
  year={2022}
}

@phdthesis{ramsey2018independence,
  title={Independence, amalgamation, and trees},
  author={Ramsey, Samuel Nicholas},
  year={2018},
  school={University of California, Berkeley}
}

@article{MR4081726,
	author = {Kaplan, Itay and Ramsey, Nicholas},
	doi = {10.4171/jems/948},
	fjournal = {Journal of the European Mathematical Society (JEMS)},
	issn = {1435-9855},
	journal = {J. Eur. Math. Soc. (JEMS)},
	mrclass = {03C45 (03C55 03C60)},
	mrnumber = {4081726},
	mrreviewer = {Junguk Lee},
	number = {5},
	pages = {1423--1474},
	title = {On {K}im-independence},
	url = {https://mathscinet.ams.org/mathscinet-getitem?mr=4081726},
	volume = {22},
	year = {2020}}

@article{chatzidakis2002properties,
	author = {Chatzidakis, Zo{\'e}},
	journal = {The Journal of Symbolic Logic},
	number = {3},
	pages = {957--996},
	publisher = {Cambridge University Press},
	title = {Properties of forking in $\omega$-free pseudo-algebraically closed fields},
	volume = {67},
	year = {2002}}

@article{chatzidakis2019amalgamation,
  title={Amalgamation of types in pseudo-algebraically closed fields and applications},
  author={Chatzidakis, Zo{\'e}},
  journal={Journal of Mathematical Logic},
  volume={19},
  number={02},
  pages={1950006},
  year={2019},
  publisher={World Scientific}
}

@article{berenstein2017supersimple,
  title={Supersimple structures with a dense independent subset},
  author={Berenstein, Alexander and Carmona, Juan Felipe and Vassiliev, Evgueni},
  journal={Mathematical Logic Quarterly},
  volume={63},
  number={6},
  pages={552--573},
  year={2017},
  publisher={Wiley Online Library}
}

@article{halevi2023saturated,
  title={Saturated models for the working model theorist},
  author={Halevi, Yatir and Kaplan, Itay},
  journal={Bulletin of Symbolic Logic},
  volume={29},
  number={2},
  pages={163--169},
  year={2023},
  publisher={Cambridge University Press}
}

@phdthesis{frohn2011model,
	author = {Frohn, Nina},
	school = {Albert-Ludwigs-Universit\"at Freiburg},
	title = {Model theory of absolute Galois groups},
	year = {2011}}

@article{cherlin1980elementary,
  title={The elementary theory of regularly closed fields},
  author={Cherlin, Gregory and van den Dries, Lou and Macintyre, Angus},
  journal={preprint},
  year={1980}
}

@article{lubotzky1981subgroups,
  title={Subgroups of free profinite groups and large subfields of},
  author={Lubotzky, Alexander and van den Dries, Lou},
  journal={Israel Journal of Mathematics},
  volume={39},
  number={1},
  pages={25--45},
  year={1981},
  publisher={Springer}
}

@article {MR792361,
    AUTHOR = {Ribes, Luis},
     TITLE = {Frattini covers of profinite groups},
   JOURNAL = {Arch. Math. (Basel)},
  FJOURNAL = {Archiv der Mathematik},
    VOLUME = {44},
      YEAR = {1985},
    NUMBER = {5},
     PAGES = {390--396},
      ISSN = {0003-889X,1420-8938},
   MRCLASS = {20E18 (20D25)},
  MRNUMBER = {792361},
MRREVIEWER = {Alexander\ Lubotzky},
       DOI = {10.1007/BF01229320},
       URL = {https://doi.org/10.1007/BF01229320},
}

@article{chatzidakis1998model,
	author = {Chatzidakis, Zo{\'e}},
	journal = {Illinois Journal of Mathematics},
	number = {1},
	pages = {70--96},
	publisher = {University of Illinois at Urbana-Champaign, Department of Mathematics},
	title = {Model theory of profinite groups having the Iwasawa property},
	volume = {42},
	year = {1998}}

@article{ArtemNick,
	author = {Chernikov, Artem and Ramsey, Nicholas},
	doi = {10.1142/S0219061316500094},
	eprint = {http://www.worldscientific.com/doi/pdf/10.1142/S0219061316500094},
	journal = {Journal of Mathematical Logic},
	pages = {1650009},
	title = {On model-theoretic tree properties},
	url = {http://www.worldscientific.com/doi/abs/10.1142/S0219061316500094},
	year = {2016}}

@book{fried2008field,
	author = {Fried, Michael D. and Jarden, Moshe},
	title = {Field Arithmetic},
	edition = {3},
	series = {A Series of Modern Surveys in Mathematics},
	volume = {11},
	publisher = {Springer-Verlag},
	address = {Berlin},
	year = {2008}}

@inproceedings{ervsov1974theories,
  title={THEORIES OF NONABELIAN VARIETIES OF GROUPS},
  author={Er{\v{s}}ov, Ju. L. },
  booktitle={Proceedings},
  volume={25},
  pages={255},
  year={1974},
  organization={American Mathematical Soc.}
}

@incollection{ZoePAC2,
	author = {Chatzidakis, Zo{{\'e}}},
	booktitle = {Models and computability ({L}eeds, 1997)},
	doi = {10.1017/CBO9780511565670.004},
	mrclass = {03C60 (03C45 12L12)},
	mrnumber = {1721163 (2000i:03051)},
	mrreviewer = {M. A. Ta{\u \i}tslin},
	pages = {41--61},
	publisher = {Cambridge Univ. Press, Cambridge},
	series = {London Math. Soc. Lecture Note Ser.},
	title = {Simplicity and independence for pseudo-algebraically closed fields},
	url = {http://dx.doi.org/10.1017/CBO9780511565670.004},
	volume = {259},
	year = {1999}}

@article{mekler1981stability,
  title={Stability of nilpotent groups of class 2 and prime exponent},
  author={Mekler, Alan H},
  journal={The Journal of Symbolic Logic},
  volume={46},
  number={4},
  pages={781--788},
  year={1981},
  publisher={Cambridge University Press}
}

@article{mutchnik2026nsop2,
  title={On NSOP2 theories},
  author={Mutchnik, Scott},
  journal={Journal of the European Mathematical Society: JEMS.},
  volume={28},
  number={8},
  pages={3475},
  year={2026},
  publisher={European Mathematical Society (EMS)}
}

@article{berenstein2016geometric,
  title={Geometric structures with a dense independent subset},
  author={Berenstein, Alexander and Vassiliev, Evgueni},
  journal={Selecta Mathematica},
  volume={22},
  number={1},
  pages={191--225},
  year={2016},
  publisher={Springer}
}

@article{boissonneau2024mekler,
  title={Mekler's Construction and Murphy's Law for 2-Nilpotent Groups},
  author={Boissonneau, Blaise and Papadopoulos, Aris and Touchard, Pierre},
  journal={arXiv preprint arXiv:2403.20270},
  year={2024}
}

@article{ahn2019mekler,
  title={Mekler's construction and tree properties},
  author={Ahn, JinHoo},
  journal={arXiv preprint arXiv:1903.07087},
  year={2019}
}

@article{chernikov2019mekler,
  title={Mekler’s construction and generalized stability},
  author={Chernikov, Artem and Hempel, Nadja},
  journal={Israel Journal of Mathematics},
  volume={230},
  number={2},
  pages={745--769},
  year={2019},
  publisher={Springer}
}

@article{baudisch2002mekler,
  title={Mekler's construction preserves $CM$-triviality},
  author={Baudisch, Andreas},
  journal={Annals of Pure and Applied Logic},
  volume={115},
  number={1-3},
  pages={115--173},
  year={2002},
  publisher={Elsevier}
}

@article{ax1968elementary,
  title={The elementary theory of finite fields},
  author={Ax, James},
  journal={Annals of Mathematics},
  volume={88},
  number={2},
  pages={239--271},
  year={1968},
  publisher={JSTOR}
}

@article{tyrrell2024finite,
  title={Finite undecidability in PAC and PRC fields},
  author={Tyrrell, Brian},
  journal={Annals of Pure and Applied Logic},
  volume={175},
  number={10},
  pages={103465},
  year={2024},
  publisher={Elsevier}
}

@book{hodges1993model,
	author = {Hodges, Wilfrid},
	publisher = {Cambridge University Press Cambridge},
	title = {Model theory},
	volume = {42},
	year = {1993}}

@article{shelah1996toward,
	author = {Shelah, Saharon},
	journal = {Annals of Pure and Applied Logic},
	number = {3},
	pages = {229--255},
	publisher = {Elsevier},
	title = {Toward classifying unstable theories},
	volume = {80},
	year = {1996}}

@article{hempel2016n,
  title={On n-dependent groups and fields},
  author={Hempel, Nadja},
  journal={Mathematical Logic Quarterly},
  volume={62},
  number={3},
  pages={215--224},
  year={2016},
  publisher={Wiley Online Library}
}

\end{document}